\documentclass {amsart}
\usepackage{parskip}
\usepackage[english]{babel}
\usepackage{amsmath}
\usepackage{amsthm}
\usepackage{amssymb}
\usepackage{enumitem}
\usepackage{hyperref}
\usepackage{pdfsync}
\usepackage{footmisc}

\usepackage{float}

\usepackage{geometry, comment}  

\usepackage{graphicx}

\usepackage{tikz}
\usepackage{pgf}
\usepackage{tikz-cd}

\usetikzlibrary{decorations.markings}
\usetikzlibrary{arrows, automata}

\newcommand{\BZ}{\mathbb{Z}}
\newcommand{\GG}{\mathbb{G}}
\newcommand{\DD}{\mathfrak{D}}
\newcommand{\CC}{\mathfrak{C}}
\newcommand{\FF}{\mathfrak{F}}
\newcommand{\ti}{\widetilde}
\DeclareMathOperator{\modu}{mod}

\newtheorem{theorem}{Theorem}[section]
\newtheorem{lemma}[theorem]{Lemma}
\newtheorem{corollary}[theorem]{Corollary}
\theoremstyle{definition}
\newtheorem{defn}[theorem]{Definition}
\theoremstyle{question}
\newtheorem{question}[theorem]{Question}
\newtheorem{prop}[theorem]{Proposition}
\newtheorem{rem}[theorem]{Remark}
\newtheorem*{conj}{Conjecture}
\newtheorem*{thm*}{Theorem}

\begin{document}

\tikzset{middlearrow/.style={
        decoration={markings,
            mark= at position 0.5 with {\arrow{#1}} ,
        },
        postaction={decorate}
    }
}

\title[Commensurability of RAAGs]{On commensurability of some right-angled Artin groups II: RAAGs defined by paths}

\thanks{This work was supported by the ERC Grant 336983, by the Basque Government grant IT974-16 and by the grants MTM2014-53810-C2-2-P and MTM2017-86802-P of the Ministerio de Econom\'ia y Competitividad of Spain. The third author was supported by CMUP (UID/MAT/00144/2013), which is funded by FCT (Portugal) with national (MEC) and European structural funds (FEDER), 
under the partnership agreement PT2020, and by the Russian Foundation for Basic Research (project no.  15-01-05823)}

\author[M. Casals-Ruiz]{Montserrat Casals-Ruiz}
\address{Ikerbasque - Basque Foundation for Science and Matematika Saila,  UPV/EHU,  Sarriena s/n, 48940, Leioa - Bizkaia, Spain}
\email{montsecasals@gmail.com}

\author[I. Kazachkov]{Ilya Kazachkov}
\address{Ikerbasque - Basque Foundation for Science and Matematika Saila,  UPV/EHU,  Sarriena s/n, 48940, Leioa - Bizkaia, Spain}
\email{ilya.kazachkov@gmail.com}

\author[A. Zakharov]{Alexander Zakharov}
\address{Centre of Mathematics, University of Porto,
R. Campo Alegre 687, 4169-007 Porto, Portugal}
\email{zakhar.sasha@gmail.com}


\begin{abstract}
In this paper we continue the study of right-angled Artin groups up to commensurability initiated in \cite{CKZ}. We show that RAAGs defined by different paths of length greater than 3 are not commensurable. We also characterise which RAAGs defined by paths are commensurable to RAAGs defined by trees of diameter 4. More precisely, we show that a RAAG defined by a path of length $n>4$ is commensurable to a RAAG defined by a tree of diameter 4 if and only if $n \equiv 2\, (\modu 4)$. These results follow from the connection that we establish between the classification of RAAGs up to commensurability and linear integer-programming.
\end{abstract}
\maketitle

\section{Introduction}

\subsection{Context of the problem}
One of the basic problems on locally compact topological groups is to classify their lattices up to commensurability. Recall that two lattices $\Gamma_1, \Gamma_2 <G$ are commensurable if and only if there exists $g\in G$ such that $\Gamma_1 \cap \Gamma_2^g$ has finite index in both $\Gamma_1$ and $\Gamma_2^g$. In particular, commensurable lattices have covolumes that are commensurable real numbers, that is, they have a rational ratio.
	
The notion of commensurability was generalized to better suit topological and large-scale geometric properties and to compare groups without requiring them to be subgroups of a common group. More precisely, we say that two groups $H$ and $K$ are (abstractly) commensurable if they have isomorphic finite index subgroups. In this paper, we will only be concerned with the notion of abstract commensurability and we simply refer to it as commensurability. 
	
As we mentioned, commensurability is closely related to the large-scale geometry of the group. Indeed, any finitely generated group can be endowed with a natural word-metric which is well-defined up to quasi-isometry and since any finitely generated group is quasi-isometric to any its finite index subgroup, it follows that commensurable groups are quasi-isometric. 
	
Gromov suggested to study groups from this geometric point of view and understand the relation between these two concepts. More precisely, a basic problem in geometric group theory is to classify commensurability and quasi-isometry classes (perhaps within a certain class) of finitely generated groups and to understand whether or not these classes coincide. 
	
The classification of groups up to commensurability (both in the abstract and classical case) has a long history and a number of famous solutions for very diverse classes of groups such as Lie groups and more generally, locally compact topological groups, hyperbolic 3-manifold groups, pro-finite groups, Grigorchuk-Gupta-Sidki groups, etc, see for instance \cite{5, 13, 20, 25, 26, GrW, Gar}.

In this paper, we focus on the question of classification of right-angled Artin groups, RAAGs for short, up to commensurability. Recall that a RAAG is a finitely presented group $\GG(\Gamma)$ which can be described by a finite simplicial graph $\Gamma$, the commutation graph, in the following way: the vertices of $\Gamma$ are in bijective correspondence with the generators of $\GG(\Gamma)$ and the set of defining relations of $\GG(\Gamma)$ consists of commutation relations, one for each pair of generators connected by an edge in $\Gamma$. 
	
RAAGs have become central in group theory, their study interweaves geometric group theory with other areas of mathematics. This class interpolates between two of the most classical families of groups, free and free abelian groups, and its study provides uniform approaches and proofs, as well as far reaching generalisations of the results for free and free abelian groups. The study of this class from the different perspectives has contributed to the development of new, rich theories such as the theory of CAT(0) cube complexes and has been an essential ingredient in Agol's solution of the virtually fibred Conjecture.
	
The commensurability classification of RAAGs has been previously considered for the following classes of RAAGs: free groups \cite{56, 25, 47}, \cite[1.C]{31}; free Abelian groups, \cite{30, 3}; $F_m \times \mathbb{Z}^n$, \cite{58}; free products of free groups and free Abelian groups, \cite{5}; $F_m \times F_n$ with $m,n \ge 2$, \cite{60, 17}; $\GG(\Gamma)$, where $\Gamma$ is connected, triangle- and square-free graph without any degree one vertices, \cite{KKi}; $\Gamma$ is star-rigid and does not have induced 4-cycles and the outer automorphism of $\GG$ is finite, \cite{Huang}; and $\Gamma$ is a tree of diameter $\le 3$, \cite{BN} and of diameter 4, \cite{CKZ}.

In \cite{CKZ}, we characterise the commensurability classes of RAAGs defined by trees of diameter 4. As a consequence, we deduce the existence of infinitely many different commensurability classes, confirming a conjecture of Behrstock and Neumann, and provide first examples of RAAGs that are quasi-isometric but not commensurable.

The proof of the aforementioned results was performed in three steps. In the first step, we determine a commensurabality invariant for RAAGs defined by trees. More precisely, to a given pair of trees $\Delta$ and $\Gamma$ we associate a linear system of equations $S(\Delta, \Gamma)$ and show that the existence of positive integer solutions is a commensurability invariant of $\GG(\Delta)$ and $\GG(\Gamma)$, i.e. we prove:
\begin{thm*}[see \cite{CKZ}]
Let $\Delta$ and $\Gamma$ be trees. If $\GG(\Delta)$ and $\GG(\Gamma)$ are commensurable, then the system of equations $S(\Delta, \Gamma)$ has positive integer solutions.
\end{thm*}

The proof of this step uses methods from geometric group theory.

In the second step, we center on RAAGs defined by trees of diameter 4 and characterise the trees for which the associated linear system of equations does not have positive integer solutions. This part is the most technical, although the methods required come from linear algebra. The strategy is to locally simplify the structure of the linear system of equations, that is, to determine some subsystems of equations and show that in order for them to have positive integer solutions, they must have an ``easy'' form. This allows us to simplify the entire system of linear equations enough to be able to determine whether or not it has positive integer solutions. This step allows us to characterise RAAGs $\GG(T)$ and $\GG(T')$ defined by trees of diameter 4 for which the system of equations $S(T,T')$ does not have positive integer solutions and hence by step 1, deduce that these groups are not commensurable.

In the last step, we consider RAAGs $\GG(T)$ and $\GG(T')$ defined by trees of diameter 4 for which the system of linear equations $S(T,T')$ does have positive integer solutions. From a minimal solution of the system $S(T,T')$, we build explicit finite index subgroups of $\GG(T)$ and $\GG(T')$, show that they are isomorphic and conclude that $\GG(T)$ and $\GG(T')$ are commensurable. The methods used in this step come mainly from Bass-Serre theory.

\subsection{Results and strategy of the proof}

In this paper, we develop methods introduced in \cite{CKZ} and study the commensurability classes of RAAGs defined by paths. More precisely, let $P_n$, $n \geq 1$, denote the path graph with $n$ edges and $n+1$ vertices and let $T_{k,k+1}$, $k \geq 1$, denote a tree of diameter 4, with the central vertex of degree 2 and such that the two vertices adjacent to the central vertex have  degrees $k+1$ and $k+2$ correspondingly, so that $T_{k,k+1}$ has $2k+1$ leaves, see Figure \ref{fig:T4}. 

We show that different paths of length more than $4$ are not commensurable.

\begin{theorem}\label{thm1}
The groups $\GG(P_n)$ and $\GG(P_m)$, $n>m \ge 0$, are commensurable if and only if $n=4$ and $m=3$.
\end{theorem}

We also compare the commensurability classes between paths and trees of diameter $4$ and prove

\begin{theorem}\label{th2}
Let $n>4$. The group $\GG(P_n)$ is commensurable to $\GG(T)$, where $T$ is a tree of diameter $4$, if and only if $n = 4k+2$, $k\ge 1$, and $\GG(T)$ is commensurable to $\GG(T_{k,k+1})$.
\end{theorem}

Note that in \cite{CKZ} we give a complete commensurability classification of RAAGs defined by trees of diameter 4, in particular, it is described which of them are commensurable to $\GG(T_{k,k+1})$.

The proof follows the same $3$-step structure as in \cite{CKZ}. Note, however, that instead of considering systems of equations as in \cite{CKZ}, in this paper we work with linear systems of equations and inequalities and instead of requiring that the linear system of equations have a positive integer solution (as in \cite{CKZ}), we require that our system of equations and inequalities have an integer solution.

\begin{itemize}
\item In Section \ref{sec:2}, we reduce commensurability between RAAGs defined by trees to the existence of integer solutions of a (disjunction of) linear system of equations and inequalities $S$, see Corollary \ref{cor:reductionLSE}.

\item In Sections \ref{sec:3}, \ref{sec:4} and \ref{sec:6}, we analyse the system $S$, characterise when it has no integer solutions and deduce when two RAAGs from our class are not commensurable. In Section \ref{sec:3}, we study the system $S$ defined by paths of length $3$ and $m\ge 5$ and show that it never has  integer solutions. This is the simplest case and it introduces the techniques and ideas for the other cases. In Section \ref{sec:4}, we study the system $S$ defined by a path of length $n>4$ and a tree of diameter $4$ and show that if $n \not\equiv 2 \, (\modu  4)$ then the system does not have integer solutions.  Finally in Section \ref{sec:6}, we address the system defined by paths of different length greater than $4$ and again show that it never has integer solutions.

\item In Section \ref{sec:5}, we show that when the path is of length $4k+2$ and the tree of diameter $4$ is $T_{k,k+1}$ we can exhibit isomorphic finite index subgroups and conclude that the corresponding RAAGs are commensurable.
\end{itemize}

\subsection{Related problems and further research:}\

As we already mentioned in \cite{CKZ}, it is our belief that the general strategy of the proof can be used to study commensurability classes of RAAGs defined by trees and more generally all  RAAGs.

Corollary \ref{cor:reductionLSE} reduces commensurability between two RAAGs defined by trees to the existence of integer solutions of a linear system of equations and inequalities. This brings up a natural question of whether or not this necessary condition is also sufficient.

\begin{question}\label{q:2}
Let $\GG(\Gamma_1)$ and $\GG(\Gamma_2)$ be RAAGs defined by trees. Is it true that $\GG(\Gamma_1)$ and $\GG(\Gamma_2)$ are commensurable if and only if the system $S_i(\Gamma_1, \Gamma_2)$ defined by the product graph {\rm(}see {\rm Section \ref{sec:prodgraph})} has integer solutions?
\end{question}

In all the cases we studied so far, solutions of the (linear) system (of equations and inequalities) have guided the construction of the  subgroups which witness commensurability.  In essence, Question \ref{q:2} asks whether one can build a finite cover of the Salvetti complex of a RAAG from local covers of the complexes associated to the centralisers of generators. 

As pointed out to us by Henry Wilton, this question may be approached using techniques introduced by Ian Agol in his solution of the virtual Haken conjecture. In \cite{Agol}, the author constructs a finite-sheeted cover which is modelled on some hierarchy. In order to do it, he constructs a measure on the space of colorings of a wall graph and then refines the colors to reflect how each wall is cut up by previous stages of the hierarchy. He then uses the measure to find a solution to certain gluing equations on the colored cubical polyhedra defined by the refined colorings, and uses solutions to these equations to get the base case of the hierarchy and glue up successively each stage of the hierarchy using a gluing theorem to glue at each stage after passing to a finite-sheeted cover. 

Solutions of the linear system are, in some sense, values necessary for the consistent gluing of the local covers. The goal would be to generalise Agol's Gluing Theorem to build the finite index cover from the local ones and so describe the finite index subgroup that witnesses commensurability.

In the same way we speculate that trees and, more generally, $2$-dimensional RAAGs could play the role of the base of induction for a hierarchy. (Recall that a RAAG is 2-dimensional if and only if its commutation graph is triangle-free.) Centralisers in $2$-dimensional RAAGs are of the form $\mathbb{Z} \times F_n$. If the answer to Question \ref{q:2} is positive, then given a solution of the system, one can build the finite index subgroup from local covers of free groups. In the general case, commensurability of RAAGs would imply compatible commensurable centralisers of certain elements and centralisers are again RAAGs of lower complexity. By induction, one then could build finite index subgroups for the centralisers and, using a gluing theorem, extend them to a finite cover of the group. This brings us to the following question

\begin{question}\label{q:3}
Let $\GG(\Gamma_1)$ and $\GG(\Gamma_2)$ be RAAGs. Is it true that  $\GG(\Gamma_1)$ and $\GG(\Gamma_2)$ are commensurable if and only if the system $S_i(\Gamma_1, \Gamma_2)$ defined by the product graph {\rm(}see {\rm Section \ref{sec:prodgraph})} has integer solutions?
\end{question}

This is just a rough strategy to approach the general problem. A good starting point is to understand whether or not the reduction from commensurability to integer solutions of a linear system generalises from trees to $2$-dimensional RAAGs. More precisely, we expect that Corollary \ref{cor:reductionLSE} can be generalised as follows

\begin{conj}
Let $\GG(\Gamma_1)$ and $\GG(\Gamma_2)$ be $2$-dimensional RAAGs. If $\GG(\Gamma_1)$ and $\GG(\Gamma_2)$ are commensurable, then the system $S_i(\Gamma_1, \Gamma_2)$ defined by the product graph  {\rm(}see {\rm Section \ref{sec:prodgraph})} has integer solutions.
\end{conj}

\medskip

The existence of integer solutions of a linear system of equations and inequalities can be interpreted as a syntactic fragment of the Presburger arithmetic (with order) and so in particular, it is a decidable problem. The Presburger arithmetic has quantifier elimination if we add predicates for division. Hence, the existence of integer solutions is equivalent to a boolean combination of atomic formulas in the language $(+,<,0)$ and congruences of integers. This justifies the classification we obtain for trees and paths, where it is required that the length $n$ of the path is congruent to $2$ modulo $4$.

The decidability of the existence of integer solutions of a linear system is a very well-known and long-studied problem in Computer Science. It was intensively studied in the field of mechanical theorem proving and it is most commonly known as (linear) integer-programming. It is actually one of the most important models in management science (capital budgeting, warehouse location, scheduling, etc) and there are many different efficient algorithms to address it. 

As a consequence, given two trees, one can describe the linear system associated to them and use an algorithm to decide whether or not the system has an integer solution. If the answer is negative, that is, there is no integer solutions, then we can conclude that the corresponding groups are not commensurable. Furthermore if the answer to Question \ref{q:2} (and Question \ref{q:3}) were positive, we could conclude that the commensurability problem between tree (and general) RAAGs is decidable as well as have a good understanding of its complexity, see \cite{CH} and references there.

\section{Systems of equations associated to tree RAAGs}\label{sec:2}

The main goal of this section is, given two RAAGs $\GG(\Gamma_1)$ and $\GG(\Gamma_2)$ defined by trees, to construct a linear system of equations and inequalities $S(\Gamma_1, \Gamma_2)$ such that if $\GG(\Gamma_1)$ and $\GG(\Gamma_2)$ are commensurable, then the system $S(\Gamma_1,\Gamma_2)$ has  integer solutions, see Corollary \ref{cor:reductionLSE}.

In order to construct the system of equations, one needs to introduce several commensurability invariants, namely the (reduced) extension graph and the quotient graph. We assign certain labels to the quotient graph, deduce a system of linear equations and show that if two tree RAAGs are commensurable, then the exponents are positive integer solutions of the system of equations.

Since the quotient graph depends on the subgroup witnessing commensurability, so does the system of equations. In order to overcome this dependence, we introduce a new graph, the product graph, which only depends on the trees $\Gamma_1$ and $\Gamma_2$, we associate certain labels to the vertices and edges, describe a linear system of equations and inequalities and show again that if the groups are commensurable, the labels are integer solutions of the linear system.

This sections follows the ideas introduced in \cite{CKZ}. For completeness, we recall the definitions and results needed in this paper.

\subsection{Reduced centralizer splitting}

Observe that tree RAAGs split as fundamental groups of graphs of groups, whose vertex groups are centralisers of vertex generators. We recall the notion of (reduced) centralizer splitting, as in \cite{CKZ}, which we will use in Section \ref{sec:5}.
\begin{defn}[(Reduced) Centraliser splitting]
Let $\Delta$ be a tree and let $\GG(\Delta)$ be the RAAG with the underlying graph $\Delta$. The {\it centraliser splitting} of $\GG(\Delta)$ is defined as follows. The graph of the splitting is isomorphic to $\Delta$ and the vertex group at every vertex is defined to be the centralizer of the corresponding vertex generator. Note that if $v$ is some vertex of $\Delta$, and $u_1, \ldots, u_s$ are all vertices of $\Delta$ adjacent to $v$, then $C(v)=\langle v,u_1,\ldots,u_s \rangle \cong \mathbb{Z} \times F_s$, where $F_s$ is the free group of rank $s$, see \cite{CKZ} for more details on centralizers in RAAGs. In particular, $C(v)$ is abelian if and only if $v$ has degree 1, and in this case $C(v) \cong \mathbb{Z}^2$ is contained in the centralizer of the vertex adjacent to $v$. For an edge $e$ connecting vertices $u$ and $v$ the edge group at $e$ is $C(u) \cap C(v)=\langle u,v \rangle  \cong \mathbb{Z}^2$.

Note that the centralizer splitting is neither reduced nor minimal, since for every vertex of degree 1 in $\Delta$ the vertex group is equal to the incident edge group. Thus it makes sense to consider the {\it reduced centraliser splitting} of $\mathbb{G}(\Delta)$ (for a tree $\Delta$), which is obtained from the centralizer splitting by removing all vertices of degree 1. In this splitting all the vertex groups are non-abelian, and all the edge groups are isomorphic to $\mathbb{Z}^2$, in particular, this splitting is already minimal and reduced. 
\end{defn}	

\subsection{Reduced extension graph and quotient graph}
In this section, we recall the notions of (reduced) extension and quotient graphs, see \cite{CKZ} for further details.

\begin{defn}[Extension graph, see \cite{KK}]
Let $\GG(\Gamma)$ be a RAAG with the underlying commutation graph $\Gamma$, then the {\it extension graph} $\Gamma^e$ is defined as follows. The vertex set of $\Gamma^e$ is the set of all elements of $\mathbb{G}(\Gamma)$ which are conjugate to the canonical generators (vertices of $\Gamma$). Two vertices are joined by an edge if and only if the corresponding group elements commute. The group $\mathbb{G}(\Gamma)$ acts on $\Gamma^e$ by conjugation.
\end{defn}
	\begin{defn}[Reduced extension graph]
For a tree $\Gamma$, we define the {\it reduced extension graph} of $\Gamma$, and denote it by $\widetilde{\Gamma}^e$,  to be the full subgraph of the extension graph $\Gamma^e$, whose vertex set is the set of all elements of $\mathbb{G}(\Gamma)$ which are conjugate to the canonical generators corresponding to vertices of $\Gamma$ of degree more than 1 (which are exactly those which have non-abelian centralizers). 
\end{defn}	
   Suppose that $\Gamma$ is a finite tree of diameter at least $3$ and let $G=\GG(\Gamma)$. Let $\ti{\Gamma}$ be the tree obtained from $\Gamma$ by deleting all degree $1$ vertices together with the incident edges.
  Then $G$ acts on $\Gamma^e$ and on $\ti{\Gamma}^e$ by conjugation, so that $G \backslash \Gamma^e \cong \Gamma$ and $G \backslash \ti{\Gamma}^e \cong \ti{\Gamma}$; these actions can be thought of as the natural actions on the Bass-Serre trees of the centralizer splitting and the reduced centralizer splitting of $G$ respectively, see Lemma 3.4 in \cite{CKZ}.  
   
  Suppose that $H$ is a finite index subgroup of $G$. Let $\Psi(H) = H \backslash \ti{\Gamma}^e$, then $\Psi(H)$ is a finite graph, and there are natural graph morphisms $\gamma = \gamma_H:   \ti{\Gamma}^e \rightarrow \Psi(H)$ and $\delta=\delta_H: \Psi(H) \rightarrow \Gamma$. Note that the image of $\delta$ lies in $\ti{\Gamma}$, so we can also think of $\delta$ as a morphism  $\delta: \Psi(H) \rightarrow \ti{\Gamma}$.
   
  \subsection{Labels in the reduced extension graph and the quotient graph}
 
Suppose $H$ is a finite index subgroup of $G=\GG(\Gamma)$. We can then associate some labels to the quotient graph as follows.
 
\begin{defn}[Label of a vertex/edge]  
Let $w$ be a vertex of $\widetilde{\Gamma}^e$, thus $w$ is also an element of $G$. Define the {\it label of the vertex $w$}, denoted by $\overline{L}(w)$, to be the minimal positive integer $k$ such that $w^k \in H$. Such number exists, since $H$ has finite index in $G$. 

For an edge $f$ of $\widetilde{\Gamma}^e$ connecting vertices $w_1$ and $w_2$ define the {\it label of the edge $f$ at the vertex $w_1$}, denoted by $\overline{l}_{w_1}(f)$, to be the minimal positive integer $k$ such that there exists an integer $l$ such that $w_1^k w_2^l \in H$. Without loss of generality, we can assume that $l$ is non-negative. The label of $f$ at $w_2$ is defined analogously. Note that, by definition, $\overline L(w_1) \geq \overline{l}_{w_1}(f)$, for all edges $f$. 
  
Observe that the labels of vertices and edges are invariant under the action of $H$ on $\widetilde{\Gamma}^e$ (by conjugation). Indeed, for $h \in H$ we have $w^k \in H$ iff $(w^h)^k \in H$, and $w_1^k w_2^l \in H$ iff $(w_1^h)^k (w_2^h)^l \in H$.  This means that we can define labels for the quotient graph $\Psi(H)$. If $v$ is a vertex of $\Psi(H)$, then define the {\it label of the vertex $v$}, denoted by $L(v)$, to be the label $\overline{L}(w)$, where $w$ is some vertex of $\widetilde{\Gamma}^e$ such that $\gamma(w)=v$. Analogously, if $p$ is an edge of $\Psi(H)$ connecting vertices $v_1$ and $v_2$, then define the {\it label of the edge $p$ at the vertex $v_1$}, denoted by $l_{v_1}(p)$, to be the label $\overline{l}_{w_1}(f)$, where $f$ is some edge of   $\widetilde{\Gamma}^e$ such that $\gamma(f)=p$, and $w_1$ is the end of $p$ such that $\gamma(w_1)=v_1$. It follows that these labels are well-defined. Note that the labels of vertices and edges of $\Psi(H)$ are positive integers.
\end{defn}

\subsection{System of equations for the quotient graph}
  
  Suppose now that $\Gamma_1$ and $\Gamma_2$ are two finite trees of diameters at least 3 such that $\GG({\Gamma_1}) $ and $\GG({\Gamma_2})$ are commensurable, and $H_1 \leq G_1$, $H_2 \leq G_2$ are finite index isomorphic subgroups, and $\varphi: H_1 \rightarrow H_2$ is the isomorphism. All the definitions above apply to both $H_1$ in $G_1$ and $H_2$ in $G_2$. In \cite{CKZ} we show that $\varphi$ induces graph isomorphisms $\overline{\varphi}$: $\ti{\Gamma}_1^e \rightarrow \ti{\Gamma}_2^e$ and $\varphi_*: \Psi(H_1) \rightarrow \Psi(H_2)$, see Lemma 3.6 in \cite{CKZ}. 
  
  We associate a system of equations to the quotient graph and show that the edge and vertex labels of $\Psi(H_1)$ and $\Psi(H_2)$ are positive integer solutions of the system. 
   
  As above, we have surjective graph morphisms 
  $$
  \gamma_1: \ti{\Gamma}_1^e \rightarrow \Psi(H_1),  \: \gamma_2: \ti{\Gamma}_2^e \rightarrow \Psi(H_2), \: \delta_1: \Psi(H_1) \rightarrow \ti{\Gamma}_1, \: \delta_2: \Psi(H_2) \rightarrow \ti{\Gamma}_2. 
  $$
   As in \cite{CKZ}, for an edge $e$ of $\Psi(H_1)$ beginning in a vertex $u$ we denote by $L(u)$ the vertex label of $u$ in $\Psi(H_1)$, by $L'(u)$ the vertex label of $\varphi_*(u)$ in $\Psi(H_2)$, by $l_{u}(e)$ the edge label of $e$ at the vertex $u$ in $\Psi(H_1)$, and by $l_{u}'(e)$ the edge label of $\varphi_*(e)$ at the vertex $\varphi_*(u)$ in $\Psi(H_2)$. All these labels are positive integers by definition. 
   
  We summarize the equations obtained in \cite{CKZ} in the following two lemmas.
  
\begin{lemma}[see Lemma 3.16 in \cite{CKZ}] \label{proportions}
		Let $e$ be an edge of $\Psi(H_1)$ connecting vertices $v_1$ and $v_2$. Then the following equations hold: 
		\begin{equation}\label{eq3}
		\frac{L(v_1)}{l_{v_1}(e)} = \frac{L'(v_1)}{l'_{v_1}(e)} = \frac{L(v_2)}{l_{v_2}(e)} = \frac{L'(v_2)}{l'_{v_2}(e)}=q,
		\end{equation}
		where $q$ is some positive integer.
	\end{lemma}
 
 Let $v$ be a vertex of $\Psi(H_1)$, and $u_1= \delta_1(v) \in V(\ti{\Gamma}_1)$, $u_2=\delta_2(\varphi_*(v)) \in V(\ti{\Gamma}_2)$. Let $p_1, \ldots, p_m$ be all vertices of $\ti{\Gamma}_1$ adjacent to $u_1$, and $q_1, \ldots, q_n$ be all vertices of $\ti{\Gamma}_2$ adjacent to $u_2$. Suppose that the edge $e_i$ connects $u_1$ with $p_i$, $i=1, \ldots, m$, and the edge $f_j$ connects $u_2$ with $q_j$, $j=1, \ldots, n$.  Let $e_i^1, \ldots, e_i^{r_i}$ be all the edges of $\Psi(H_1)$ beginning in $v$ which project into $e_i$ under $\delta_1$, for each $i=1, \ldots, m$; note that $e_1^1, \ldots, e_1^{r_1}, e_2^1, \ldots, e_2^{r_2}, \ldots, e_m^1, e_m^2, \ldots, e_m^{r_m}$ are all the edges of $\Psi(H_1)$ beginning in $v$.
 Analogously, let $f_j^1, \ldots, f_j^{s_j}$ be all the edges of $\Psi(H_2)$ beginning in $\varphi_*(v)$ which project into $f_j$  under $\delta_2$, for each $j=1, \ldots, n$; note that $f_1^1, \ldots, f_1^{s_1}, f_2^1, \ldots, f_2^{s_2}, \ldots, f_n^1, f_n^2, \ldots, f_n^{s_n}$ are all the edges of $\Psi(H_2)$ beginning in $\varphi_*(v)$, in particular, $\varphi_*$ induces a bijection between them and the above edges in $\Psi(H_1)$ beginning in $v$.  
 
 Let also $D_1$ be the degree of $u_1$ in $\Gamma_1$ minus 1, and $D_2$ be the degree of $u_2$ in $\Gamma_2$ minus 1. Note that we take degrees in $\Gamma_1$, $\Gamma_2$, not in $\ti{\Gamma}_1$, $\ti{\Gamma}_2$. Below all the edge labels are taken at the end vertices, i.e., at the vertices which are not $u_1$ or $u_2$; we omit the lower index notation.
\begin{lemma}\label{equations}
	For every vertex $v$ of $\Psi(H_1)$, in the above notation the following equations hold:
	\begin{equation}\label{eq}
	\begin{split}
		D_1 \sum_{i=1}^{r_1}l(e_1^i) = D_1 \sum_{i=1}^{r_2}l(e_2^i) = \ldots = D_1 \sum_{i=1}^{r_m}l(e_m^i) = \\
		= D_2 \sum_{j=1}^{s_1}l'(f_1^j) = D_2 \sum_{j=1}^{s_2}l'(f_2^j) = \ldots = D_2 \sum_{j=1}^{s_n}l'(f_n^j). 
	\end{split}	
	\end{equation}
\end{lemma}
\begin{proof}
	The statement follows immediately from Lemmas 3.12 and 3.13 in \cite{CKZ}. 
\end{proof}
 
\subsection{Product graph}   \label{sec:prodgraph}

In this section, we introduce the product graph, which only depends on the graphs $\Gamma_1$ and $\Gamma_2$, and describe its relation with the quotient graph.
 
Recall that for two graphs $\Delta_1, \Delta_2$ their {\it direct product} (also called tensor product) is the graph $\Delta_1 \times \Delta_2$ defined as follows. The vertex set of $\Delta_1 \times \Delta_2$ is the (Cartesian) product of the vertex sets of $\Delta_1$ and $\Delta_2$. If $u_1, v_1$ are vertices of $\Delta_1$, and $u_2, v_2$ are vertices of $\Delta_2$, then we define two vertices $(u_1, u_2)$ and $(v_1, v_2)$ to be adjacent in $\Delta_1 \times \Delta_2$ if and only if $u_1$ is adjacent to $v_1$ in $\Delta_1$ and $u_2$ is adjacent to $v_2$ in $\Delta_2$. Note that this is a category-theoretic product, which means that there exist naturally defined projection morphisms 
$$
\pi_1^0:  \Delta_1 \times \Delta_2 \rightarrow \Delta_1, \: \pi_1^0((u,v))=u; \quad  \pi_2^0:  \Delta_1 \times \Delta_2 \rightarrow \Delta_2, \: \pi_2^0((u,v))=v,
$$
such that if $\Delta$ is a graph and $\delta_1^0: \Delta \rightarrow \Delta_1$, $\delta_2^0: \Delta \rightarrow \Delta_2$ are graph morphisms, then there exists a unique graph morphism $\delta^0: 
\Delta \rightarrow \Delta_1 \times \Delta_2$ such that $\delta_1^0=\pi_1^0 \delta^0$ and $\delta_2^0 = \pi_2^0 \delta^0$. Namely, $\delta^0(x)=(\delta_1^0(x), \delta_2^0(x))$ for every vertex  $x$ of $\Delta$, and it is extended to the edges of $\Delta$ in a natural way.

 \begin{defn}[Product graph and type morphism] \
	 Recall $\ti{\Gamma}_1$ is the subgraph obtained from $\Gamma_1$ by deleting all degree 1 vertices and incident edges, and similar for $\Gamma_2$. We define the \emph{product graph} to be the direct product $\mathfrak{D}=\ti{\Gamma}_1 \times \ti{\Gamma}_2$ of the subgraphs $\ti{\Gamma}_1$ and $\ti{\Gamma}_2$.

	To abbreviate the notation, we will denote $\Psi(H_1)$ by $\Psi$. Let $\delta: \Psi \rightarrow \mathfrak{D}$ be the graph morphism defined by $\delta(x)=(\delta_1(x), \delta_2(\varphi_*(x)))$ for every vertex $x$ of $\Psi$, extended to the edges in a natural way.  Let $\pi_1:  \mathfrak{D} \rightarrow \ti{\Gamma}_1 $, 	$\pi_2:  \mathfrak{D} \rightarrow \ti{\Gamma}_2$ be the projection morphisms as above. We call $\delta$ the \emph{type morphism}. Recall that in \cite{CKZ} we defined the pair of vertices $(\delta_1(u), \delta_2(\varphi_*(u)))$ to be the type of a vertex $u$ of $\Psi(H_1)$. Thus, the type of a vertex $u$ is now just $\delta(u)$, which is a vertex of $\mathfrak{D}$, see Figure \ref{fig:0}.	
\end{defn}

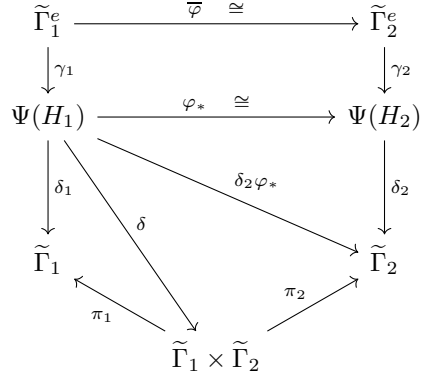
\begin{figure}[!h]
	\begin{tikzcd}
		\ti{\Gamma}_1^e \arrow[rr, "\overline{\varphi} \quad \cong"] \arrow[d, "\gamma_1"]
		&&   \ti{\Gamma}_2^e \arrow[d, "\gamma_2"] \\
		\Psi(H_1) \arrow[rr,"\varphi_* \quad \cong"] \arrow[dd, "\delta_1"] 
		\arrow[ddrr, "\delta_2\varphi_*"] \arrow[dddr, "\delta"]
		&& \Psi(H_2) \arrow[dd,"\delta_2"] \\ \\
		\ti{\Gamma}_1 && \ti{\Gamma}_2 \\
		& \ti{\Gamma}_1 \times \ti{\Gamma}_2 \arrow[ul, "\pi_1"] \arrow[ur, "\pi_2"]
	\end{tikzcd}
	\caption{\small Commutative diagram for the reduced extension graphs and quotient graphs of commensurable tree RAAGs} \label{fig:0}
\end{figure}

We denote the image of $\Psi$ in $\DD$ under $\delta$ by $\CC$. 
 Thus the vertex set of $\CC$ is the set of all possible types of vertices of $\Psi$. 
  \begin{lemma}\label{DD}
  		In the above notation the following statements hold:
  		\begin{enumerate}
  			\item The product graph $\mathfrak{D}$ has two connected components $\DD_1$ and $\DD_2$, and the graph $\mathfrak{C}$ is a connected subgraph of one of these components.
  			\item The restrictions of the projections $\pi_1:  \mathfrak{D} \rightarrow \ti{\Gamma}_1$ and $\pi_2:  \mathfrak{D} \rightarrow \ti{\Gamma}_2$ to $\CC$ are surjective.
  			\item Moreover, the restrictions of $\pi_1$ and $\pi_2$ to $\CC$ are locally surjective, i.e. if $v \in V(\CC)$, and $\pi_1(v) \in V(\ti{\Gamma}_1)$ is adjacent to some vertex $u \in V(\ti{\Gamma}_1)$, then there exists a vertex $w \in V(\CC)$ adjacent to $v$ in $\CC$ such that $\pi_1(w)=u$; the same is true for $\pi_2$.
  		\end{enumerate}
  \end{lemma}
\begin{proof}
	The first claim follows from the general graph-theoretic fact that the direct product of two connected graphs without cycles of odd length has two connected components \cite{Weichsel}, in particular this is true for trees. Obviously, $\CC$ is connected, so it should lie inside one of the two connected components. 
	
	The second claim follows immediately from the fact that $\delta_1=\pi_1\delta$ and $\delta_2=\pi_2\delta \varphi_*^{-1}$ are surjective. 
	
	Note that $\delta_1\gamma_1$ is locally surjective in the sense above, so $\delta_1$ is locally surjective as well, and then $\pi_1$ restricted to $\CC$ is also locally surjective; the same holds for $\pi_2$ restricted to $\CC$. 
\end{proof}

\subsection{Edge labels for the product graph}
  
   As we did with the quotient graph, in this section we define certain labels and assign them to edges of the product graph.

   We assume that all our graphs are directed, so that each edge $e$ has initial vertex $\alpha(e)$, terminal vertex $\omega(e)$ and inverse edge $e^{-1}$. Note that  the edge labels on $\Psi$ defined above are defined for unoriented edges, i.e. they are the same for $e$ and $e^{-1}$ (but depend on the choice of the endpoint). We now define edge labels for the product graph $\DD$. Each oriented edge $e$ will have $4$ labels, denoted by $M_{11}(e), M_{12}(e), M_{21}(e)$ and $M_{22}(e)$.

   Let $e$ be an edge of $\DD$. If $e \notin E(\CC)$, define $M_{11}(e)=M_{12}(e)=M_{21}(e)=M_{22}(e)=0$. 
 Otherwise, if $e \in E(\CC)$, let $\delta^{-1}(e)=\{ f_1, \ldots, f_k \}$ be the edges of $\Psi$ which project into $e$ under $\delta$, $k \geq 1$. Let $\alpha(e)=u$, $\omega(e)=v$, and $\alpha(e_i)=u_i$, $\omega(e_i)=v_i$, for $i=1, \ldots, k$, so that $\delta(u_i)=u$, $\delta(v_i)=v$, $i=1, \ldots, k$. Now define the edge labels of $e$ as follows:
 	\begin{equation}\label{M}
 	\begin{split} 
		M_{11}(e)=\sum_{i=1}^k L(u_i)l_{v_i}(e_i);  \quad M_{12}(e)=\sum_{i=1}^k L(u_i)l'_{v_i}(e_i); \\ \quad
	    M_{21}(e)=\sum_{i=1}^k L'(u_i)l_{v_i}(e_i);  \quad M_{22}(e)=\sum_{i=1}^k L'(u_i)l'_{v_i}(e_i).
	    \end{split}
	\end{equation}

	\begin{lemma}\label{CC}
	For every $e \in E(\DD)$, all the labels of $e$ are non-negative integers, and 
	\begin{equation}\label{M0}
		e \in E(\CC) \Leftrightarrow M_{11}(e)>0 \Leftrightarrow M_{12}(e)>0 \Leftrightarrow M_{21}(e)>0 \Leftrightarrow M_{22}(e)>0.		 
	\end{equation} 
	\end{lemma}
	\begin{proof}
		Follows immediately from the definition of labels in $\DD$ and the fact that  all edge and vertex labels of $\Psi$ are positive integers. 
	\end{proof}	
	
	\begin{rem}\label{rem:localsurject}
	Note that the conditions on {\rm(}local{\rm)} surjectivity from {\rm Lemma \ref{DD}} can be expressed as a union of linear equations and inequalities. Indeed, if $e' \in E(\ti\Gamma_j)$ and $e_1, \dots, e_k \in E(\DD)$ is the preimage $\pi_j^{-1}(e')$ in $\DD$, then, by {\rm Lemma \ref{CC}}, $\pi_j$ is surjective on the edge $e'$ if and only if 
	$$
	\sum\limits_{i=1}^{k}M_{lm}(e_i) >0,\ j,l,m\in \{1,2\}.
	$$ 
	Similarly, if $v \in V(\CC)$, $\pi_j(v) \in V(\ti{\Gamma}_j)$ is adjacent to some vertex $u \in V(\ti{\Gamma}_j)$, $e'=(\pi_j(v),u) \in E(\ti{\Gamma}_j)$ and $e_1, \dots, e_k \in E(\DD)$ are all the edges in the preimage $\pi_j^{-1}(e')$ in $\DD$ which begin in $v$, then $\sum\limits_{i=1}^{k}M_{lm}(e_i) >0$, $j,l,m\in \{1,2\}$.
	\end{rem}
	
\subsection{Linear system for the product graph}
	In this section, we define a linear system of equations and inequalities associated to the product graph and show that the labels of the edges are positive integer solutions of the system.
	
	We first show that the following equations are satisfied for each edge. 
	
	\begin{lemma}\label{EE}
		For every edge $e \in E(\DD)$ the following holds:
		\begin{equation}\label{ee}
			M_{11}(e)=M_{11}(e^{-1}), \: M_{12}(e)=M_{21}(e^{-1}), \: M_{21}(e)=M_{12}(e^{-1}), \: M_{22}(e)=M_{22}(e^{-1}). 
		\end{equation}
	\end{lemma}
	\begin{proof}
		If $e \notin E(\CC)$, then all the labels of $e$ and $e^{-1}$ are 0, and so the claim holds automatically. Suppose now $e \in E(\CC)$. Note that, in the above notation, by Lemma \ref{proportions}, we have that $L(v_i)l_{u_i}(e_i)=L(u_i)l_{v_i}(e_i)$, for all $i=1, \ldots, k$, hence 
		$$
		M_{11}(e^{-1})= \sum_{i=1}^k L(v_i)l_{u_i}(e_i)= \sum_{i=1}^k L(u_i)l_{v_i}(e_i) = M_{11}(e).
		$$
Analogously, by Lemma \ref{proportions}, we have that $L(v_i)l'_{u_i}(e_i)=L'(u_i)l_{v_i}(e_i)$, for all $i=1, \ldots, k$ and so
		$$
		M_{12}(e^{-1})= \sum_{i=1}^k L(v_i)l'_{u_i}(e_i)= \sum_{i=1}^k L'(u_i)l_{v_i}(e_i) = M_{21}(e).
		$$
The proof of other two equalities in (\ref{ee}) is analogous.		
	\end{proof}
	
	We now describe the equations that we associate to each vertex of $\DD$. Let $w=(w_1,w_2) \in V(\DD)$, where $w_1=\pi_1(w) \in V(\ti{\Gamma}_1)$, $w_2=\pi_2(w) \in V(\ti{\Gamma}_2)$. Let $p_1, \ldots, p_m$ be all vertices of $\ti{\Gamma}_1$ adjacent to $w_1$, and $q_1, \ldots, q_n$ be all vertices of $\ti{\Gamma}_2$ adjacent to $w_2$. Suppose that the edge $e_i$ connects $u_1$ with $p_i$, $i=1, \ldots, m$, and the edge $f_j$ connects $u_2$ with $q_j$, $j=1, \ldots, n$. Then $w_{ij}=(p_i, q_j)$, $i=1, \ldots, m$, $j=1, \ldots, n$, are all the vertices of $\DD$ adjacent to $w$. Let $e_{ij}$ be the (oriented) edge of $\DD$ beginning in $w$ and ending in $w_{ij}$, $i=1, \ldots, m$, $j=1, \ldots, n$. Thus, the edges $e_{i1}, e_{i2}, \ldots, e_{in}$ are all the edges of $\DD$ beginning in $w$ which project into $e_i$ under $\pi_1$, for $1\le i\le m$ and the edges $e_{1j}, e_{2j}, \ldots, e_{mj}$ are all the edges of $\DD$ beginning in $w$ which project into $f_j$ under $\pi_2$, for $1\le j\le n$.

		Recall that $D_1$ is the degree of the vertex $w_1$ considered as a vertex of $\Gamma_1$ minus $1$, and $D_2$ is the degree of the vertex $w_2$ considered as a vertex of $\Gamma_2$ minus $1$. Note that we take degrees in $\Gamma_1$, $\Gamma_2$, not in $\ti{\Gamma}_1$, $\ti{\Gamma}_2$.
	\begin{lemma}\label{VE}
		For every vertex $w \in V(\DD)$, in the above notation the following equations hold:
		\begin{equation}\label{ve1}
		\begin{split}
		 		D_1 \sum_{j=1}^n M_{11}(e_{1j}) =  D_1 \sum_{j=1}^n M_{11}(e_{2j}) = \ldots = 	D_1 \sum_{j=1}^n M_{11}(e_{mj})	= \\
		 		= D_2 \sum_{i=1}^m M_{12}(e_{i1}) = D_2 \sum_{i=1}^m M_{12}(e_{i2}) = \ldots = D_2 \sum_{i=1}^m M_{12}(e_{in}),
		 \end{split}		
		\end{equation}
		and
		\begin{equation}\label{ve2}
		\begin{split}
		 		D_1 \sum_{j=1}^n M_{21}(e_{1j}) =  D_1 \sum_{j=1}^n M_{21}(e_{2j}) = \ldots = 	D_1 \sum_{j=1}^n M_{21}(e_{mj})	= \\
		 		= D_2 \sum_{i=1}^m M_{22}(e_{i1}) = D_2 \sum_{i=1}^m M_{22}(e_{i2}) = \ldots = D_2 \sum_{i=1}^m M_{22}(e_{in}).
		 \end{split}
		\end{equation}
	\end{lemma}	
	\begin{proof}
		If $w \notin V(\CC)$, then the claim follows since in this case $M_{ij}$'s are all equal to $0$. Hence, we can assume that $w \in V(\CC)$. Let $\delta^{-1}(w)= \{w_1, \ldots, w_N\}$ be all the vertices of $\Psi(H_1)$ which project to $w$ under $\delta$. Now for each vertex $w_k$, $k=1, \ldots, N$, write the equations (\ref{eq}) from Lemma \ref{equations}, multiply each side by $L(w_k)$ and sum over $k=1, \ldots, N$. Equation (\ref{ve1}) follows now from the definitions of the labels on $\DD$.  Analogously, for each vertex $w_k$, $k=1, \ldots, N$, write the equations (\ref{eq}) from Lemma \ref{equations}, multiply each side by $L'(w_k)$ and sum over $k=1, \ldots, N$. Hence Equation (\ref{ve2}).  
		\end{proof}
	
	In the above notation, for every vertex $w \in V(\DD)$, define two labels: 
\begin{equation}\label{R}	
	R_1(w)=D_1 \sum_{j=1}^n M_{11}(e_{1j}), \quad R_2(w)=D_2 \sum_{i=1}^m M_{22}(e_{in}).
\end{equation}	
	By Lemma \ref{VE}, we can rewrite these labels in several different ways.

\begin{defn}[Linear system $S(\Gamma_1, \Gamma_2)$]
Let $\GG(\Gamma_1)$ and $\GG(\Gamma_2)$ be two RAAGs defined by trees $\Gamma_1$ and $\Gamma_2$ and let $\ti{\Gamma}_i$ be the induced subgraph of $\Gamma_i$ defined by all non-leaf vertices, $i=1,2$. We denote by $S_i(\Gamma_1,\Gamma_2)$ the system of linear equations defined by the $i$-th connected component $\DD_i$ of the product graph $\DD$ of $\ti{\Gamma}_1$ and $\ti{\Gamma}_2$ in variables $M_{kl}(e)$, $k,l=1,2$, $e \in E(\DD_i)$, that is $S_i(\Gamma_1,\Gamma_2)$ is the union of Equations (\ref{ee}), (\ref{ve1}) and (\ref{ve2}) from Lemmas \ref{EE} and \ref{VE}, for all edges and vertices of the $i$-th connected component $\DD_i$. 

Let $P$ be the set of inequalities $M_{kl}(e) \ge 0$, for all $e\in E(\DD)$. Let $E$ be the following disjunction of equations and inequalities encoding surjectivity:
$$
\sum\limits_{i=1}^{k}M_{lm}(e_i) >0,
$$ 
for all $e' \in E(\ti\Gamma_j)$, where $e_1, \dots, e_k \in E(\DD)$ is the preimage $\pi_j^{-1}(e')$ in $\DD$, and $j,l,m\in \{1,2\}$; and local surjectivity:
$$
M_{kl}(e) >0 \implies \sum\limits_{i=1}^{k}M_{lm}(e_i) >0,
$$
or, equivalently,
$$
M_{kl}(e) = 0 \  \vee \  \sum\limits_{i=1}^{k}M_{lm}(e_i) > 0,
$$
for all $v \in V(\DD)$, for all $e=(v,v') \in E(\DD)$, for all vertices $u \in V(\ti{\Gamma}_j)$ adjacent to $\pi_j(v) \in V(\ti{\Gamma}_j)$, $e'=(\pi_j(v),u) \in E(\ti{\Gamma_j})$, where $e_1, \dots, e_k \in E(\DD)$ are all the edges in the preimage $\pi_j^{-1}(e')$ in $\DD$ which begin in $v$, and $j,l,m\in \{1,2\}$.

Note that
$$
\left(S_1(\Gamma_1, \Gamma_2)\  \vee\ S_2(\Gamma_1,\Gamma_2)\right)\ \wedge \  P \ \wedge \  E
$$
is a disjunction of linear system of equations and inequalities in variables $M_{kl}(e)$, which we denote by $S(\Gamma_1,\Gamma_2)$.
\end{defn}

\begin{corollary}\label{cor:reductionLSE}
In the above notation, if $\GG(\Gamma_1)$ and $\GG(\Gamma_2)$ are commensurable, then the {\rm(}disjunction of{\rm)} linear system of equations and inequalities $S(\Gamma_1, \Gamma_2)$ has an integer solution. 
\end{corollary}
\begin{proof}
It follows from Lemmas \ref{EE} and \ref{VE} that the labels assigned to vertices and edges of $\Psi(H)$ give rise to a solution of the linear system of equations $S_i$ for some $i=1,2$. By Lemma \ref{CC} the solutions are non-negative and so satisfy the inequalities of the system $P$. Furthermore, by Lemma \ref{DD}, the labels satisfy the disjunctions of inequalities encoding local surjectivity and so are a solution of the system $E$.
\end{proof}

\section{The RAAGs $\GG(P_3)$ and $\GG(P_m)$ are not commensurable for $m \geq 5$}\label{sec:3}
	
	In this section, we prove a special case of Theorem \ref{thm1}, namely we show that $\GG(P_3)$ is not commensurable to $\GG(P_n)$, for all $n>4$. As we already mentioned, $\GG(P_3)$ and $\GG(P_4)$ are commensurable, see \cite{CKZ} for more details on that.
	
	The proof in this special case is easier and introduces the reader to the techniques and ideas behind the proof of Theorem \ref{thm1}. The proof of Theorem \ref{thm1} is the most technically demanding in this paper and we will address it in the last section, Section \ref{sec:6}.
	
	\begin{theorem}\label{P3}
		$\GG(P_3)$ is not commensurable to $\GG(P_m)$ for $m \geq 5$.
	\end{theorem}
	\begin{proof}
			 Let $a_0, a_1, \ldots, a_m$ be the vertices of $P_m$, considered as canonical generators of $\GG(P_m)$, and $b_0, b_1, b_2, b_3$ be the vertices of $P_3$, considered as canonical generators of $\GG(P_3)$. Then, in the above notation, $\Gamma_1=P_m$, $\Gamma_2=P_3$, and $\ti{\Gamma}_1=P_{m-2}$, with vertices $a_1, \ldots, a_{m-1}$, $\ti{\Gamma}_2=P_{1}$, with vertices $b_1, b_{2}$. 
			  
			Suppose that $\GG(P_m)$ and $\GG(P_3)$ are commensurable. 
			Note that in our case $\DD = \ti{\Gamma}_1 \times \ti{\Gamma}_2=P_{m-2} \times P_{1}$ is the following graph: its set of vertices is $\{ (a_i, b_j), \: i=1, \ldots, m-1; \: j=1, \ldots, 2 \}$, and two vertices $(a_{i_1}, b_{j_1})$ and $(a_{i_2}, b_{j_2})$ are connected by an edge in $\DD$ if and only if $|i_1-i_2| = 1$ and $|j_1-j_2|=1$, for $i_1, i_2=1, \ldots, m-1; \: j_1, j_2=1, 2$.	To abbreviate the notation, we will denote the vertex $(a_i,b_j)$ of $\DD$ by $(i,j)$, where $i=1, \ldots, m-1; \: j=1, 2$, see Figure \ref{fig:1}. 	
			 
			 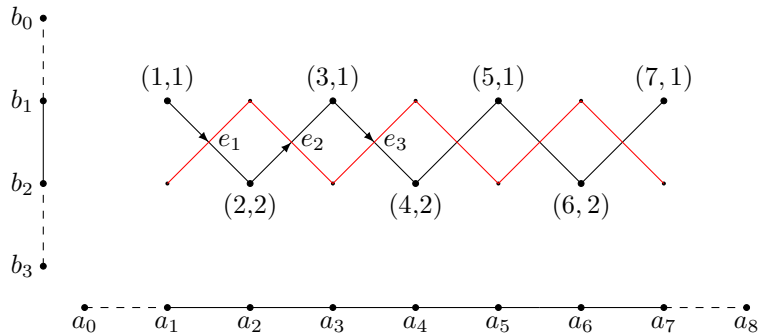
\begin{figure}[h!]
			 	\begin{tikzpicture}[scale=1.1]
			 	\tikzset{near start abs/.style={xshift=1cm}}
			 	\draw [dashed] (0.5,0)--(1.5,0);
			 	\draw (1.5,0) -- (6,0);
			 	\draw (6, 0) -- (7.5, 0);
			 	\draw[dashed] (7.5,0)--(8.5,0);
			 	\filldraw (0.5,0) circle (1pt)  node[align=left, below] {$a_0$};
			 	\filldraw (1.5,0) circle (1pt)  node[align=left, below] {$a_1$};
			 	\filldraw (2.5,0) circle (1pt)  node[align=left, below] {$a_2$};
			 	\filldraw (3.5,0) circle (1pt)  node[align=left, below] {$a_3$};
			 	\filldraw (4.5,0) circle (1pt)  node[align=left, below] {$a_4$};
			 	\filldraw (5.5,0) circle (1pt)  node[align=left, below] {$a_5$};
			 	\filldraw (6.5,0) circle (1pt)  node[align=left, below] {$a_{6}$};
			 	\filldraw (7.5,0) circle (1pt)  node[align=left, below] {$a_{7}$};
			 	\filldraw (8.5,0) circle (1pt)  node[align=left, below] {$a_{8}$};
			 	
			 	\draw  (0,1.5) -- (0,2.5);
			 	\draw[dashed] (0,0.5)--(0,1.5);
			 	\draw[dashed] (0,2.5)--(0,3.5);
			 	\filldraw (0,0.5) circle (1pt)  node[align=left, left] {$b_3$};
			 	\filldraw (0,1.5) circle (1pt)  node[align=left, left]  {$b_2$};
			 	\filldraw (0,2.5) circle (1pt)  node[align=left, left] {$b_1$};
			 	\filldraw (0,3.5) circle (1pt)  node[align=left, left] {$b_0$};

			 	\filldraw (1.5,2.5) circle (1pt) node[align=left, above] {(1,1)};
			 	\filldraw (2.5,1.5) circle (1pt) node[align=left, below] {(2,2)};
			 	\filldraw (3.5,2.5) circle (1pt) node[align=left, above] {(3,1)};
			 	\filldraw (4.5,1.5) circle (1pt) node[align=left, below] {(4,2)};
			 	\filldraw (5.5,2.5) circle (1pt) node[align=left, above] {(5,1)};
			 	\filldraw (6.5,1.5) circle (1pt) node[align=left, below] {($6,2$)};
			 	\filldraw (7.5,2.5) circle (1pt) node[align=left, above] {($7,1$)};

			 	\draw[middlearrow={latex}] (1.5,2.5) -- (2.5,1.5)  node [midway, right] {$e_1$};
			 	\draw[middlearrow={latex}] (2.5,1.5) -- (3.5,2.5)  node [midway, right] {$e_2$};
			 	\draw[middlearrow={latex}] (3.5,2.5) -- (4.5,1.5)  node [midway, right] {$e_3$};
			 	\draw (4.5,1.5) -- (5.5,2.5);
			 	\draw (5.5,2.5) -- (6,2);
			 	\draw (6,2) -- (6.5,1.5);
			 	\draw (6.5,1.5) -- (7.5,2.5);
			 	
			 	\filldraw (1.5,1.5) circle (0.5pt);
			 	\filldraw (2.5,2.5) circle (0.5pt);
			 	\filldraw (3.5,1.5) circle (0.5pt);
			 	\filldraw (4.5,2.5) circle (0.5pt);
			 	\filldraw (5.5,1.5) circle (0.5pt);
			 	\filldraw (6.5,2.5) circle (0.5pt);
			 	\filldraw (7.5,1.5) circle (0.5pt);
			 	
			 	\draw[color=red] (1.5,1.5) -- (2.5,2.5) (2.5,2.5) -- (3.5,1.5) (3.5,1.5) -- (4.5,2.5) (4.5,2.5) -- (5.5,1.5) (5.5,1.5) -- (6,2) (6,2) -- (6.5,2.5) (6.5,2.5) -- (7.5,1.5);
			 	
			 	\end{tikzpicture}
			 	\caption{\small Graph $\mathfrak{D}$ in the proof of Theorem \ref{P3} for $m=8$. Edges of $\mathfrak{D}_1$ are black and edges of $\mathfrak{D}_2$ are red.} \label{fig:1}
			 \end{figure}

 Note that, as in Lemma \ref{DD}, $\DD$ has two connected components, one of them, denoted by $\DD_1$, consisting of vertices $(i,j)$, where $i+j$ is even, and the other one, denoted by $\DD_2$, where $i+j$ is odd, and $\CC$ lies in one of them. The automorphism of $P_3$ which reverses the order of its vertices (this also induces an automorphism of $\GG(P_3)$) switches these components, which are isomorphic graphs, so, after applying this automorphism of $\GG(P_3)$ if necessary, without loss of generality, we assume that $\CC$ lies in a particular component of $\DD$. 
	
    So we assume that $\CC$ lies in the component $\DD_1$ containing the vertex $(1,1)$. Then $\DD_1$ contains vertices $(1,1), (2,2), (3,1), (4,2), \ldots, (m-1, n_0)$, where $n_0 = 1$ if $m$ is even, and $n_0=2$ if $m$ is odd, and all the connecting edges (as a graph, $\DD_1$ is isomorphic to the path of length $m-2$). Note also that, by Lemma \ref{DD}, $\CC$ should project surjectively on $\ti{\Gamma}_1$, so in fact $\CC= \DD_1$ in this case. 
     
    Denote by $e_1$ the (oriented) edge of $\DD$ beginning in $(1,1)$ and ending in $(2,2)$, by $e_2$ the edge beginning in $(2,2)$ and ending in $(3,1)$, and by $e_3$ the edge beginning in $(3,1)$ and ending in $(4,2)$ (note that such edges always exist, since $m \geq 5$, so $m-1 \geq 4$), see Figure \ref{fig:1}.

    Note that we have the following equations on the labels of $e_1, e_2, e_3$: edge equations as in Lemma \ref{EE}, and vertex equations which follow from Lemma \ref{VE} applied to the case under consideration. 
     
    From the $(1,1)$ vertex we get 
  	$$
  	M_{11}(e_1)=M_{12}(e_1), \: M_{21}(e_1)=M_{22}(e_1);
  	$$
  	from the $(2,2)$ vertex we get
  	$$
  	M_{11}(e_1^{-1})=M_{11}(e_2)=M_{12}(e_1^{-1})+M_{12}(e_2), \: M_{21}(e_1^{-1})=M_{21}(e_2)=M_{22}(e_1^{-1})+M_{22}(e_2);
  	$$
  	and from the $(3,1)$ vertex we get
  		$$
  		M_{11}(e_2^{-1})=M_{11}(e_3)=M_{12}(e_2^{-1})+M_{12}(e_3), \: M_{21}(e_2^{-1})=M_{21}(e_3)=M_{22}(e_2^{-1})+M_{22}(e_3). 
  		$$
  		Using these equations, on the one hand we can write
  		\begin{equation}\label{r1} 
  		M_{11}(e_1)= M_{12}(e_1)=M_{21}(e_1)^{-1}=M_{22}(e_1^{-1})+M_{22}(e_2)=M_{22}(e_1)+M_{22}(e_2^{-1}), \end{equation}
  		 and on the other hand
		\begin{equation}\label{r2} 
		\begin{split} M_{11}(e_1)=M_{11}(e_1^{-1})= M_{12}(e_1^{-1})+M_{12}(e_2) = M_{21}(e_1)+M_{21}(e_2^{-1})= \\ =M_{22}(e_1)+M_{21}(e_2^{-1})= M_{22}(e_1)+ M_{22}(e_2^{-1})+M_{22}(e_3).
		\end{split} 
		\end{equation}

		Comparing (\ref{r1}) and (\ref{r2}), we see that $M_{22}(e_3)=0$, so by $(\ref{M0})$ we have $e_3 \notin \CC$, which is a contradiction, since $\CC=\DD_1$ as mentioned above. This shows that $\GG(P_3)$  is not commensurable to $\GG(P_m)$ for $m \geq 5$.  
	\end{proof}

	\begin{corollary}\label{P4}
		$\GG(P_4)$ is not commensurable to $\GG(P_m)$ for $m \geq 5$.
	\end{corollary}
	\begin{proof}
		Follows immediately from Theorem \ref{P3} and the fact that $\GG(P_3)$ and $\GG(P_4)$ are commensurable, see \cite[Proposition 4.4]{CKZ}.
	\end{proof}

	\section{Non-commensurability of RAAGs defined by paths and trees of diameter $4$}\label{sec:4}	
		
		In this section, we address the commensurability relations between the classes of RAAGs defined by paths and trees of diameter $4$. More precisely, we show that if $n>5$ is not congruent to $2$ modulo $4$, then $\GG(P_n)$ is not commensurable to any RAAG defined by a tree of diameter $4$.

		The general strategy to prove that two RAAGs defined by trees are not commensurable is common. First, we specialise the linear system of equations and inequalities given in Corollary \ref{cor:reductionLSE} to the case under consideration. We then find a (local) pattern in the product graph and prove that the absence of certain set of edges $S$ at a vertex implies the absence of other edges, see Lemma \ref{l2} and Figure \ref{fig:6}. We then determine a vertex $v$ for which the set of edges $S$ is missing, see Lemma \ref{tr}. This allows us to recursively remove edges using the identified local pattern and starting at $v$ until we remove enough edges to contradict the local surjectivity at a vertex (assured by Lemma \ref{DD}). This allows us to conclude that the system $S$ does not have integer solutions and hence the RAAGs are not commensurable.

		For our purposes, it will be convenient to encode finite trees of diameter four as follows. Let $T$ be any finite tree of diameter four.  Let $f$ be a path (without backtracking) of length four from one leaf of $T$ to another. By definition $f$ contains $5$ vertices and let $c_f\in V(T)$ be the middle vertex in $f$. It is immediate to see that the choice of the vertex $c=c_f$ does not depend on the choice of the path $f$ of length four. We call $c$ the \emph{center} of $T$. 
		
		Any leaf of $T$ connected to $c$ by an edge is called a \emph{hair vertex}. Vertices connected to $c$ by an edge which are not hair are called \emph{pivots}.  Any finite tree $T$ of diameter $4$ is uniquely defined by the number $q$ of hair vertices and by the number $k_i$ of pivots of a given degree $d_i+1$.  Hence we encode any finite tree of diameter $4$ as $T((d_1,k_1), \dots, (d_l,k_l);q)$. Here all $d_i$ and $k_i$ and $l$ are positive integers, $d_1 < d_2 < \ldots < d_l$, and $q$ is a non-negative integer; moreover, either $l \geq 2$ or $l=1$ and $k_1 \geq 2$, so that $T$ indeed has diameter $4$. See \cite{CKZ} for more details.

		\begin{theorem}\label{non-com}
			Let $m \geq 5$ and suppose that $m$ is not $2$ modulo $4$. Then $\GG(P_m)$ is not commensurable to any RAAG defined by a tree of diameter $4$. 
		\end{theorem}
		\begin{proof}
			Let $T$ be a tree of diameter $4$, and $m \geq 5$, $m$ is not $2$ modulo $4$. We need to prove that $\GG(P_m)$ is not commensurable to $\GG(T)$. If $T=P_4$, then the claim follows from Theorem \ref{P3}. Hence, without loss of generality, we can assume that $\GG(T)$ is not commensurable to $\GG(P_4)$. By \cite{CKZ}, this implies that $\GG(T)$ is commensurable to $\GG(T((m_k,1),(m_{k-1},1),\ldots,(m_1,1);0))$ for some $m_1 > m_2 > \ldots > m_k \geq 1$. Here  $T((m_k,1),(m_{k-1},1),\ldots,(m_1,1);0)$ is the tree with central vertex $c$ of degree $k$, connected to pivot vertices $b_1, \ldots,b_k$, and each $b_i$ has degree $m_i+1$ and is connected to $m_i$ degree one vertices, see Figure \ref{fig:30}.
			
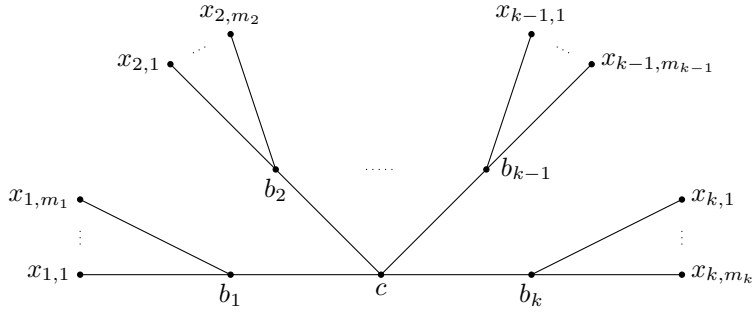
\begin{figure}[!h]	
\begin{tikzpicture}

\draw (1,0) -- (9,0);

\filldraw (1,0) circle (1pt)  node[align=left, left] {$x_{1,1}$};
\filldraw (3,0) circle (1pt)  node[align=left, below] {$b_1$};
\filldraw (5,0) circle (1pt)  node[align=left, below] {$c$};
\filldraw (7,0) circle (1pt)  node[align=left, below] {$b_k$};
\filldraw (9,0) circle (1pt)  node[align=right,right] {$x_{k,m_k}$};
\filldraw (1,1) circle (1pt)  node[align=left, left] {$x_{1,m_1}$};
\filldraw (9,1) circle (1pt)  node[align=right, right] {$x_{k,1}$};

\draw[dotted] (1,0.4)--(1,0.6) (9,0.4)--(9,0.6);

\draw (3,0) -- (1,1) (7,0) -- (9,1);

\draw (5,0) -- (2.2,2.8);
\filldraw (3.6,1.4) circle (1pt)  node[align=left, below] {$b_2$};
\filldraw (2.2,2.8) circle (1pt)  node[align=left, left] {$x_{2,1}$};
\filldraw (3,3.2) circle (1pt)  node[align=left, above] {$x_{2,m_2}$};

\draw (3.6,1.4)--(3,3.2);
\draw[dotted] (2.5,2.95)--(2.7,3.05);

\draw (5,0) -- (7.8,2.8);
\filldraw (6.4,1.4) circle (1pt)  node[align=right, below, right=2pt] {$b_{k-1}$};
\filldraw (7.8,2.8) circle (1pt)  node[align=right,right] {$x_{k-1,m_{k-1}}$};
\filldraw (7,3.2) circle (1pt)  node[align=right, above] {$x_{k-1,1}$};

\draw (6.4,1.4)--(7,3.2);
\draw[dotted] (7.5,2.95)--(7.3,3.05);

\draw[dotted] (4.8,1.4)--(5.2,1.4);

\end{tikzpicture}
	  \caption{\small Tree $T((m_k,1),(m_{k-1},1),\ldots,(m_1,1); 0)$ for $m_1>m_2>\ldots >m_{k-1}>m_k>0$.} \label{fig:30}
	 \end{figure}

Thus we can assume that $T=T((m_k,1),(m_{k-1},1),\ldots,(m_1,1);0)$ for some $m_1 > m_2 > \ldots > m_k \geq 1$, $k \geq 2$. So, in the above notation, we have that $\Gamma_1=P_m$, $m \geq 5$, with the vertices $a_0, a_1, \ldots, a_{m-1}, a_m$, and $\Gamma_2=T$. It follows that $\ti{\Gamma}_1=P_{m-2}$, with the vertices $a_1, \ldots, a_{m-1}$, and $\ti{\Gamma}_2=\Delta$, where $\Delta$ is a tree of diameter $2$ with central vertex $c$ of degree $k$, connected to the degree one (in $\Delta$) vertices $b_1, \ldots, b_k$. So $a_1, \ldots, a_{m-1}$ are those canonical generators of $\GG(P_m)$ which have non-abelian centralizers, and $c, b_1, \ldots, b_k$ are those canonical generators of $\GG(T)$ which have non-abelian centralizers.

We have that $\DD=P_{m-2} \times \Delta$ is a graph with vertices of the form $(a_i, c)$ and $(a_i, b_j)$, where $i=1, \ldots, m-1$, $j=1,\ldots, k$, and the following edges: for $i=1,\ldots,m-2$ the vertex $(a_i, c)$ is connected to the vertices $(a_{i+1},b_j)$, for all $j=1, \ldots, k$, and for $i=2,\ldots,m-1$ the vertex $(a_i,c)$  is connected to the vertices $(a_{i-1},b_j)$, for all $j=1, \ldots, k$. We will denote the vertex $(a_i, c)$ by $(i,c)$, and the vertex $(a_i,b_j)$ by $(i,j)$ for shortness, for $i=1,\ldots,m-1$, $j=1,\ldots,k$, see Figure \ref{fig:3}.

\begin{figure}[!h]
	\small 
	\begin{tikzpicture}[scale=0.8]
	
	\draw (1,0) -- (7,0);
	
	\filldraw (1,0) circle (1pt)  node[align=left, below] {$a_1$};
	\filldraw (2,0) circle (1pt)  node[align=left, below] {$a_2$};
	\filldraw (3,0) circle (1pt)  node[align=left, below] {$a_3$};
	\filldraw (4,0) circle (1pt)  node[align=left, below] {$a_4$};
	\filldraw (5,0) circle (1pt)  node[align=left, below] {$a_5$};
	\filldraw (6,0) circle (1pt)  node[align=left, below] {$a_{6}$};
	\filldraw (7,0) circle (1pt)  node[align=left, below] {$a_{7}$};

	\filldraw (0,1) circle (1pt)  node[align=right, right] {$b_k$};
	\filldraw (0,2) circle (1pt)  node[align=right, right] {$b_3$};
	\filldraw (0,3) circle (1pt)  node[align=right, right] {$b_2$};
	\filldraw (0,4) circle (1pt)  node[align=right, right] {$b_1$};
	\draw[dotted] (0,1.4) -- (0, 1.6);
	\filldraw (-1,5) circle (1pt)  node[align=left, left] {$c$};
	\draw (-1,5) -- (0,1) (-1,5) -- (0,2) (-1,5) -- (0,3) (-1,5) -- (0,4);

	\filldraw (1,1) circle (0.5pt);
	\filldraw (1,2) circle (0.5pt); 
	\filldraw (1,3) circle (0.5pt);
	\filldraw (1,4) circle (0.5pt);
	\filldraw (1,5) circle (1pt);
	\draw[dotted] (1,1.4) -- (1, 1.6);
	
	\filldraw (2,1) circle (1pt);
	\filldraw (2,2) circle (1pt);
	\filldraw (2,3) circle (1pt);
	\filldraw (2,4) circle (1pt);
	\filldraw (2,5) circle (0.5pt);
	\draw[dotted] (2,1.4) -- (2, 1.6);
	
	\filldraw (3,1) circle (0.5pt);
	\filldraw (3,2) circle (0.5pt);
	\filldraw (3,3) circle (0.5pt);
	\filldraw (3,4) circle (0.5pt);
	\filldraw (3,5) circle (1pt);
	\draw[dotted] (3,1.4) -- (3, 1.6);
	
	\filldraw (4,1) circle (1pt);
	\filldraw (4,2) circle (1pt);
	\filldraw (4,3) circle (1pt);
	\filldraw (4,4) circle (1pt);
	\filldraw (4,5) circle (0.5pt);
	\draw[dotted] (4,1.4) -- (4, 1.6);
	
	\filldraw (5,1) circle (0.5pt);
	\filldraw (5,2) circle (0.5pt);
	\filldraw (5,3) circle (0.5pt);
	\filldraw (5,4) circle (0.5pt);
	\filldraw (5,5) circle (1pt);
	\draw[dotted] (5,1.4) -- (5, 1.6);
	
	\draw (1,5)--(2,4) (1,5)--(2,3) (1,5)--(2,2) (1,5)--(2,1);
	\draw (2,4)--(3,5) (2,3)--(3,5) (2,2)--(3,5) (2,1)--(3,5);
	\draw (3,5)--(4,4) (3,5)--(4,3) (3,5)--(4,2) (3,5)--(4,1);
	\draw (4,4)--(5,5) (4,3)--(5,5) (4,2)--(5,5) (4,1)--(5,5);
	
	\draw[dotted](5.5, 3)--(5.5,3.2);
	
	\draw[dotted] (6,1.4) -- (6, 1.6);

	\filldraw (6,1) circle (1pt); 
	\filldraw (6,2) circle (1pt); 
	\filldraw (6,3) circle (1pt); 
	\filldraw (6,4) circle (1pt);
	\filldraw (6,5) circle (0.5pt);
	\draw[dotted] (7,1.4) -- (7, 1.6);

	\filldraw (7,1) circle (0.5pt); 
	\filldraw (7,2) circle (0.5pt);
	\filldraw (7,3) circle (0.5pt);
	\filldraw (7,4) circle (0.5pt);
	\filldraw (7,5) circle (1pt);
	\draw (5,5)--(6,4) (5,5)--(6,3) (5,5)--(6,2) (5,5)--(6,1);
	\draw (6,4)--(7,5) (6,3)--(7,5) (6,2)--(7,5) (6,1)--(7,5);

	\draw[color=blue] (2,5)--(3,4) (2,5)--(3,3) (2,5)--(3,2) (2,5)--(3,1);
	\draw[color=blue] (1,4)--(2,5) (1,3)--(2,5) (1,2)--(2,5) (1,1)--(2,5);
	\draw[color=blue] (4,5)--(5,4) (4,5)--(5,3) (4,5)--(5,2) (4,5)--(5,1);
	\draw[color=blue] (3,4)--(4,5) (3,3)--(4,5) (3,2)--(4,5) (3,1)--(4,5);
	\draw[color=blue] (6,5)--(7,4) (6,5)--(7,3) (6,5)--(7,2) (6,5)--(7,1);
	\draw[color=blue] (5,4)--(6,5) (5,3)--(6,5) (5,2)--(6,5) (5,1)--(6,5);

	\end{tikzpicture}
	\hspace{0.2cm}
	\begin{tikzpicture}[scale=0.8]
	
	\draw (1,0) -- (6,0);
	
	\filldraw (1,0) circle (1pt)  node[align=left, below] {$a_1$};
	\filldraw (2,0) circle (1pt)  node[align=left, below] {$a_2$};
	\filldraw (3,0) circle (1pt)  node[align=left, below] {$a_3$};
	\filldraw (4,0) circle (1pt)  node[align=left, below] {$a_4$};
	\filldraw (5,0) circle (1pt)  node[align=left, below] {$a_{5}$};
	\filldraw (6,0) circle (1pt)  node[align=left, below] {$a_{6}$};

	\filldraw (0,1) circle (1pt)  node[align=right, right] {$b_k$};
	\filldraw (0,2) circle (1pt)  node[align=right, right] {$b_3$};
	\filldraw (0,3) circle (1pt)  node[align=right, right] {$b_2$};
	\filldraw (0,4) circle (1pt)  node[align=right, right] {$b_1$};
	\draw[dotted] (0,1.4) -- (0, 1.6);
	\filldraw (-1,5) circle (1pt)  node[align=left, left] {$c$};
	\draw (-1,5) -- (0,1) (-1,5) -- (0,2) (-1,5) -- (0,3) (-1,5) -- (0,4);

	\filldraw (1,1) circle (0.5pt); 
	\filldraw (1,2) circle (0.5pt); 
	\filldraw (1,3) circle (0.5pt); 
	\filldraw (1,4) circle (0.5pt); 
	\filldraw (1,5) circle (1pt);
	\draw[dotted] (1,1.4) -- (1, 1.6);
	
	\filldraw (2,1) circle (1pt);
	\filldraw (2,2) circle (1pt);
	\filldraw (2,3) circle (1pt);
	\filldraw (2,4) circle (1pt);
	\filldraw (2,5) circle (0.5pt);
	\draw[dotted] (2,1.4) -- (2, 1.6);
	
	\filldraw (3,1) circle (0.5pt);
	\filldraw (3,2) circle (0.5pt);
	\filldraw (3,3) circle (0.5pt);
	\filldraw (3,4) circle (0.5pt);
	\filldraw (3,5) circle (1pt);
	\draw[dotted] (3,1.4) -- (3, 1.6);
	
	\filldraw (4,1) circle (1pt); 
	\filldraw (4,2) circle (1pt);
	\filldraw (4,3) circle (1pt);
	\filldraw (4,4) circle (1pt); 
	\filldraw (4,5) circle (0.5pt);
	\draw[dotted] (4,1.4) -- (4, 1.6);

	\draw (1,5)--(2,4) (1,5)--(2,3) (1,5)--(2,2) (1,5)--(2,1);
	\draw (2,4)--(3,5) (2,3)--(3,5) (2,2)--(3,5) (2,1)--(3,5);
	\draw (3,5)--(4,4) (3,5)--(4,3) (3,5)--(4,2) (3,5)--(4,1);
	\draw (4,4)--(5,5) (4,3)--(5,5) (4,2)--(5,5) (4,1)--(5,5);
	
	\draw[dotted](4.5, 3)--(4.5,3.2);
	
	\filldraw (5,1) circle (0.5pt); 
	\filldraw (5,2) circle (0.5pt); 
	\filldraw (5,3) circle (0.5pt); 
	\filldraw (5,4) circle (0.5pt); 
	\filldraw (5,5) circle (1pt);
	\draw[dotted] (5,1.4) -- (5, 1.6);

	\filldraw (6,1) circle (1pt); 
	\filldraw (6,2) circle (1pt); 
	\filldraw (6,3) circle (1pt); 
	\filldraw (6,4) circle (1pt); 
	\filldraw (6,5) circle (0.5pt);
	\draw[dotted] (6,1.4) -- (6, 1.6);

	\draw (5,5)--(6,4) (5,5)--(6,3) (5,5)--(6,2) (5,5)--(6,1);

	\draw[color=blue] (2,5)--(3,4) (2,5)--(3,3) (2,5)--(3,2) (2,5)--(3,1);
	\draw[color=blue] (1,4)--(2,5) (1,3)--(2,5) (1,2)--(2,5) (1,1)--(2,5);
	\draw[color=blue] (3,4)--(4,5) (3,3)--(4,5) (3,2)--(4,5) (3,1)--(4,5);
	\draw[color=blue] (5,4)--(6,5) (5,3)--(6,5) (5,2)--(6,5) (5,1)--(6,5);
	
	\draw[color=blue] (4,5)--(5,4) (4,5)--(5,3) (4,5)--(5,2) (4,5)--(5,1);
	\end{tikzpicture}
	
	\caption{\small Graph $\mathfrak{D}$ in the proof of Theorem \ref{non-com} for $m=8$ (on the left) and $m=7$ (on the right).  Edges of $\mathfrak{D}_1$ are black and edges of $\mathfrak{D}_2$ are blue.} \label{fig:3}
\end{figure}
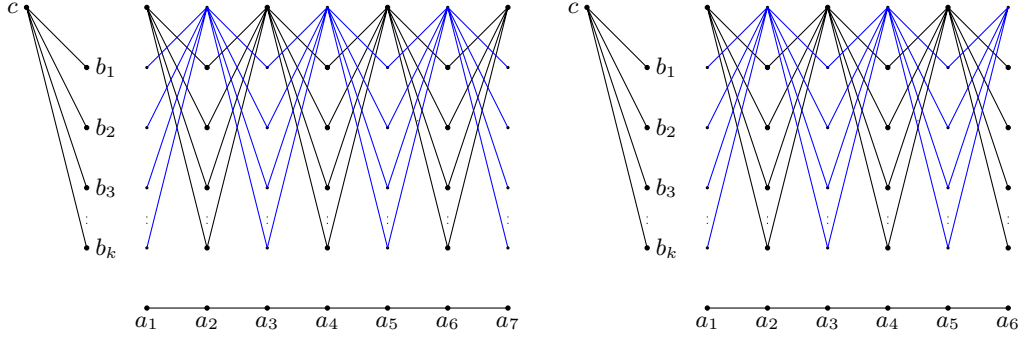

By Lemma \ref{DD}, the graph $\DD$ has two connected components, and $\CC$ is a connected subgraph of one of them. Denote the one which contains the vertex $(1,c)$ by $\DD_1$ and the other one by $\DD_2$. Then $\DD_1$ contains all the vertices of the form $(i,c)$, where $i$ is odd, $1 \leq i \leq m-1$, and $(i',j)$, where $i'$ is even, $1 \leq i' \leq m-1$, $j=1, \ldots, k$, as well as all the incident edges, and $\DD_2$ contains all the vertices of the form $(i,c)$, where $i$ is even, $1 \leq i \leq m-1$, and $(i',j)$, where $i'$ is odd, $1 \leq i' \leq m-1$, $j=1, \ldots, k$, as well as all the incident edges. 

\begin{rem} \label{rem:f3}
Note that if $m$ is odd, then the automorphism of $P_m$ which reverses the order of its vertices {\rm(}it also induces an automorphism of $\GG(P_m)${\rm)} switches the components $\DD_1$ and $\DD_2$, which are in this case isomorphic graphs, so, after applying this automorphism of $\GG(P_m)$ if necessary, without loss of generality, we can assume that $\CC$ lies in a particular component of $\DD$.  However, if $m$ is even {\rm(}i.e., $m$ is $0$ modulo $4$ in our case{\rm)}, then $\DD_1$ and $\DD_2$ are not isomorphic graphs, and we should consider two cases, depending on whether $\CC$ lies in $\DD_1$ or $\DD_2$, see {\rm Figure \ref{fig:3}}.
\end{rem}

	Denote by $e_{l,c}^{i,j}$ the (oriented) edge of $\DD$ beginning in the vertex $(l, c)$ and ending in the vertex $(i,j)$, for all  $i,l=1,\ldots,m$, such that $|i-l|=1$, and $j=1,\ldots,k$. This is an edge of $\DD_1$ if $l$ is odd and $i$ is even, and of $\DD_2$ if $l$ is even and $i$ is odd. Denote also by $e_{i,j}^{l,c}=(e_{l,c}^{i,j})^{-1}$ the inverse edge, beginning in $(i,j)$ and ending in $(l,c)$.
	
	Note that, in the notation of Lemma \ref{VE}, we have $D_1=1$ for all vertices of $\DD$, and $D_2=k-1$ for the vertices $(i,c)$ of $\DD$, $i=1,\ldots,m-1$, and $D_2=m_j-1$ for the vertices $(i,j)$ of $\DD$, $i=1,\ldots,m-1$, $j=1, \ldots,k$.

	In our case the equations of Lemma \ref{VE} and Equation (\ref{R}) have the following form. For a vertex $w=(1,j)$,  where $j=1,\ldots,k$, which has degree $1$, we have 
	$$
	R_1(w)=M_{11}(e_{1,j}^{2,c})=m_jM_{12}(e_{1,j}^{2,c}), \quad  R_2(w)=M_{21}(e_{1,j}^{2,c})=m_jM_{22}(e_{1,j}^{2,c}). 
	$$
	For a vertex $w=(i,j)$ of $\DD$, where $1<i<m-1$, $j=1, \ldots, k$, which has degree 2, we have
	$$
	\begin{array}{l}
	R_1(w)=M_{11}(e_{i,j}^{i-1,c})=M_{11}(e_{i,j}^{i+1,c})=m_j(M_{12}(e_{i,j}^{i-1,c})+M_{12}(e_{i,j}^{i+1,c})),\\
	R_2(w)=M_{21}(e_{i,j}^{i-1,c})=M_{21}(e_{i,j}^{i+1,c})=m_j(M_{22}(e_{i,j}^{i-1,c})+M_{22}(e_{i,j}^{i+1,c})).
	\end{array}
	$$
	For a vertex $w=(m-1,j)$ of $\DD$, where $j=1, \ldots, k$, which has degree 1, we have
	$$
	R_1(w)=M_{11}(e_{m-1,j}^{m-2,c})=m_jM_{12}(e_{m-1,j}^{m-2,c}), \quad  R_2(w)=M_{21}(e_{m-1,j}^{m-2,c})=m_jM_{22}(e_{m-1,j}^{m-2,c}). 
	$$
	We refer the reader to Figure \ref{fig:4} for notation.
	
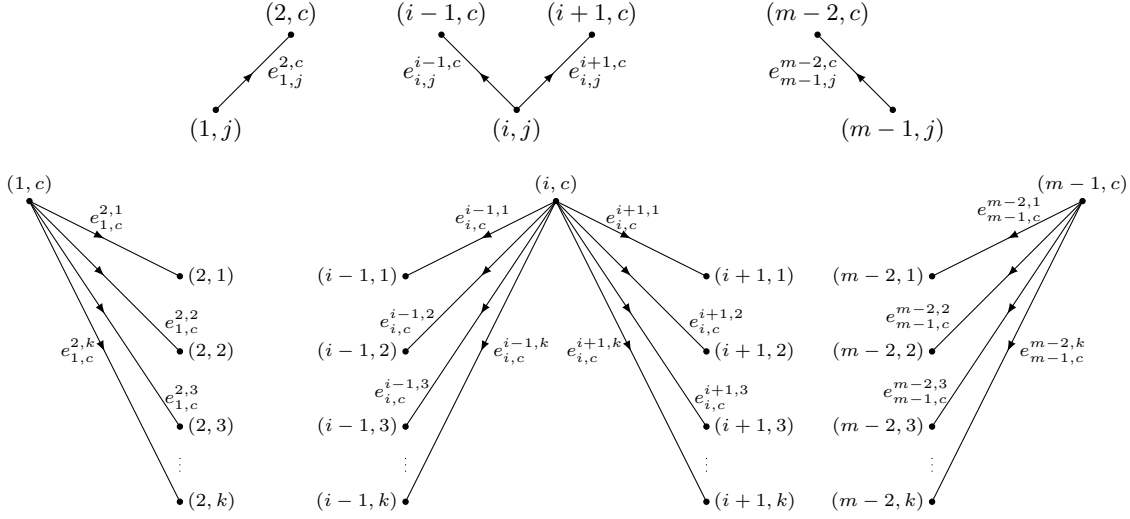
\begin{figure}[!h]
\begin{tikzpicture}

\small

\filldraw (1,4) circle (1pt) node[align=left, below] {$(1,j)$};
\filldraw (2,5) circle (1pt) node[align=left, above] {$(2,c)$};
\draw[middlearrow={latex}] (1,4) -- (2,5)  node [midway, right] {$\: e_{1,j}^{2,c}$};

\filldraw (5,4) circle (1pt) node[align=left, below] {$(i,j)$};
\filldraw (6,5) circle (1pt) node[align=left, above] {$(i+1,c)$};
\filldraw (4,5) circle (1pt) node[align=left, above] {$(i-1,c)$};
\draw[middlearrow={latex}] (5,4) -- (4,5)  node [midway,left=3pt] {$\: e_{i,j}^{i-1,c}$};
\draw[middlearrow={latex}] (5,4) -- (6,5)  node [midway, right] {$\: e_{i,j}^{i+1,c}$};

\filldraw (10,4) circle (1pt) node[align=left, below] {$(m-1,j)$};
\filldraw (9,5) circle (1pt) node[align=left, above] {$(m-2,c)$};
\draw[middlearrow={latex}] (10,4) -- (9,5)  node [midway,left=3pt] {$\: e_{m-1,j}^{m-2,c}$};

\end{tikzpicture}

\vspace{0.2 cm}

\begin{tikzpicture}

\scriptsize
\filldraw (0,5) circle (1pt) node[align=right, above] {$(1,c)$};
\filldraw (2,1) circle (1pt) node[align=right, right] {$(2,k)$};
\filldraw (2,2) circle (1pt) node[align=right, right] {$(2,3)$};
\filldraw (2,3) circle (1pt) node[align=right, right] {$(2,2)$};
\filldraw (2,4) circle (1pt) node[align=right, right] {$(2,1)$};
\draw[dotted] (2,1.4) -- (2, 1.6);
\draw[middlearrow={latex}] (0,5) -- (2,1)  node [midway, left] {$\: e_{1,c}^{2,k}$};
\draw[middlearrow={latex}] (0,5) -- (2,2)  node [above=3pt] {$\: e_{1,c}^{2,3}$};
\draw[middlearrow={latex}] (0,5) -- (2,3)  node [above=3pt] {$\: e_{1,c}^{2,2}$};
\draw[middlearrow={latex}] (0,5) -- (2,4)  node [midway, right, above] {$\: e_{1,c}^{2,1}$};

\filldraw (7,5) circle (1pt) node[align=right, above] {$(i,c)$};

\filldraw (9,1) circle (1pt) node[align=right, right] {$(i+1,k)$};
\filldraw (9,2) circle (1pt) node[align=right, right] {$(i+1,3)$};
\filldraw (9,3) circle (1pt) node[align=right, right] {$(i+1,2)$};
\filldraw (9,4) circle (1pt) node[align=right, right] {$(i+1,1)$};
\draw[dotted] (9,1.4) -- (9, 1.6);

\filldraw (5,1) circle (1pt) node[align=left, left] {$(i-1,k)$};
\filldraw (5,2) circle (1pt) node[align=left, left] {$(i-1,3)$};
\filldraw (5,3) circle (1pt) node[align=left, left] {$(i-1,2)$};
\filldraw (5,4) circle (1pt) node[align=left, left] {$(i-1,1)$};
\draw[dotted] (5,1.4) -- (5, 1.6);

\draw[middlearrow={latex}] (7,5) -- (9,1)  node [midway, left] {$\: e_{i,c}^{i+1,k}$};
\draw[middlearrow={latex}] (7,5) -- (9,2)  node [right=5pt, above=3pt] {$\: e_{i,c}^{i+1,3}$};
\draw[middlearrow={latex}] (7,5) -- (9,3)  node [right=3pt, above=3pt] {$\: e_{i,c}^{i+1,2}$};
\draw[middlearrow={latex}] (7,5) -- (9,4)  node [midway, right, above] {$\: e_{i,c}^{i+1,1}$};

\draw[middlearrow={latex}] (7,5) -- (5,1)  node [midway, right] {$\: e_{i,c}^{i-1,k}$};
\draw[middlearrow={latex}] (7,5) -- (5,2)  node [left=2pt, above=5pt] {$\: e_{i,c}^{i-1,3}$};
\draw[middlearrow={latex}] (7,5) -- (5,3)  node [above=3pt] {$\: e_{i,c}^{i-1,2}$};
\draw[middlearrow={latex}] (7,5) -- (5,4)  node [midway, right, above] {$\: e_{i,c}^{i-1,1}$};

\filldraw (14,5) circle (1pt) node[align=right, above] {$(m-1,c)$};

\filldraw (12,1) circle (1pt) node[align=left, left] {$(m-2,k)$};
\filldraw (12,2) circle (1pt) node[align=left, left] {$(m-2,3)$};
\filldraw (12,3) circle (1pt) node[align=left, left] {$(m-2,2)$};
\filldraw (12,4) circle (1pt) node[align=left, left] {$(m-2,1)$};
\draw[dotted] (12,1.4) -- (12, 1.6);

\draw[middlearrow={latex}] (14,5) -- (12,1)  node [midway, right] {$\: e_{m-1,c}^{m-2,k}$};
\draw[middlearrow={latex}] (14,5) -- (12,2)  node [left=7pt, above=5pt] {$\: e_{m-1,c}^{m-2,3}$};
\draw[middlearrow={latex}] (14,5) -- (12,3)  node [left=6pt, above=5pt] {$\: e_{m-1,c}^{m-2,2}$};
\draw[middlearrow={latex}] (14,5) -- (12,4)  node [midway, right, above=3pt] {$\: e_{m-1,c}^{m-2,1}$};
\end{tikzpicture}
	  \caption{\small Vertex stars of the graph $\mathfrak{D}$ in the proof of Theorem \ref{non-com}.} \label{fig:4}
	\end{figure}

	For a vertex $w=(1,c)$ of $\DD$, which has degree $k$, we have
$$
	\begin{array}{c}
	R_1(w)=\sum_{j=1}^k M_{11}(e_{1,c}^{2,j})=(k-1)M_{12}(e_{1,c}^{2,j'}), \quad j'=1,\ldots,k, 
	\\
R_2(w)=\sum_{j=1}^k M_{21}(e_{1,c}^{2,j})=(k-1)M_{22}(e_{1,c}^{2,j'}), \quad j'=1,\ldots,k. 
\end{array}
$$
	For a vertex $w=(i,c)$ of $\DD$, where $1 < i < m-1$, which has degree $2k$, we have, for all $j'=1, \ldots, k$, 
	\begin{equation*}
		R_1(w)=\sum_{j=1}^k M_{11}(e_{i,c}^{i-1,j})=\sum_{j=1}^k M_{11}(e_{i,c}^{i+1,j})=(k-1)(M_{12}(e_{i,c}^{i-1,j'}) + M_{12}(e_{i,c}^{i+1,j'})), 
	\end{equation*}
	$$
	R_2(w)= \sum_{j=1}^k M_{21}(e_{i,c}^{i-1,j})=\sum_{j=1}^k M_{21}(e_{i,c}^{i+1,j})=(k-1)(M_{22}(e_{i,c}^{i-1,j'}) + M_{22}(e_{i,c}^{i+1,j'})).
	$$
	Finally, for a vertex $w=(m-1,c)$ of $\DD$, which has degree $k$, we have 
	$$
	\begin{array}{c}
	R_1(w)=\sum_{j=1}^k M_{11}(e_{m-1,c}^{m-2,j})=(k-1)M_{12}(e_{m-1,c}^{m-2,j'}), \quad j'=1,\ldots,k, \\
	R_2(w)=\sum_{j=1}^k M_{21}(e_{m-1,c}^{m-2,j})=(k-1)M_{22}(e_{m-1,c}^{m-2,j'}), \quad j'=1,\ldots,k. 
	\end{array}
	$$

	Recall that $m_1>m_2> \ldots >m_k \geq 1$, $k \geq 2$ and $m \geq 5$. We will suppose that $s=1,2$ is such that $\CC$ lies in $\DD_s$.
	
	\begin{lemma}\label{l2}
		Let $v=(i,c)$ be a vertex of $\DD_s$, where $1 \leq i \leq m-2$. Suppose that the edge $e_{i,c}^{i-1,k}$ is not in $\CC$, where $i > 1$. Then the edges $e_{i,c}^{i+1,j'}$, $j'=1, \ldots, k-1$, are not in $\CC$. 
		
		If, in addition, $i \leq m-3$, then the edges $e_{i+1,j'}^{i+2,c}$ are not in $\CC$, where $j'=1, \ldots, k-1$. Moreover, if $i \leq m-4$, then also the edge $e_{i+2,c}^{i+3,k}$ is not in $\CC$. 
	\end{lemma}
	\begin{proof}
	We will prove the lemma in the case $1 \leq i \leq m-4$; the cases $i=m-3$ and $i=m-2$ are analogous. To abbreviate the notation, denote the edges $e_{i,c}^{i-1,j}$ by $e_j$ (if $i \geq 2$), the edges $e_{i,c}^{i+1,j}$ by $f_j$, the edges $e_{i+2,c}^{i+1,j}$ by $h_j$, and the edges $e_{i+2,c}^{i+3,j}$ by $p_j$, for all $j=1, \ldots, k$. So we know that $e_k \notin \CC$, see Figure \ref{fig:6}.
	
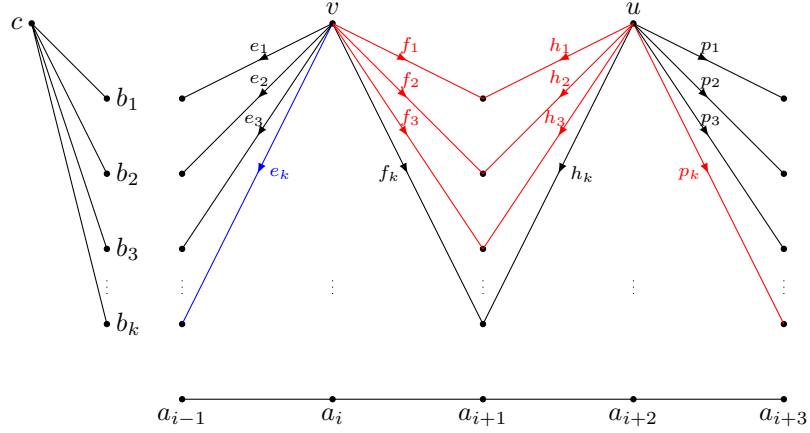
\begin{figure}[!h]
\begin{tikzpicture}

\draw (1,0) -- (9,0);

\filldraw (1,0) circle (1pt)  node[align=left, below] {$a_{i-1}$};
\filldraw (3,0) circle (1pt)  node[align=left, below] {$a_i$};
\filldraw (5,0) circle (1pt)  node[align=left, below] {$a_{i+1}$};
\filldraw (7,0) circle (1pt)  node[align=left, below] {$a_{i+2}$};
\filldraw (9,0) circle (1pt)  node[align=left, below] {$a_{i+3}$};

\filldraw (0,1) circle (1pt)  node[align=right, right] {$b_k$};
\filldraw (0,2) circle (1pt)  node[align=right, right] {$b_3$};
\filldraw (0,3) circle (1pt)  node[align=right, right] {$b_2$};
\filldraw (0,4) circle (1pt)  node[align=right, right] {$b_1$};
\draw[dotted] (0,1.4) -- (0, 1.6);
\filldraw (-1,5) circle (1pt)  node[align=left, left] {$c$};
\draw (-1,5) -- (0,1) (-1,5) -- (0,2) (-1,5) -- (0,3) (-1,5) -- (0,4);

\filldraw (1,1) circle (1pt); 
\filldraw (1,2) circle (1pt); 
\filldraw (1,3) circle (1pt); 
\filldraw (1,4) circle (1pt); 

\draw[dotted] (1,1.4) -- (1, 1.6);

\filldraw (3,5) circle (1pt)  node[align=right, above] {$v$};
\draw[dotted] (3,1.4) -- (3, 1.6);

\filldraw (5,1) circle (1pt); 
\filldraw (5,2) circle (1pt); 
\filldraw (5,3) circle (1pt); 
\filldraw (5,4) circle (1pt); 

\draw[dotted] (5,1.4) -- (5, 1.6);

\filldraw (7,5) circle (1pt)  node[align=right, above] {$u$};
\draw[dotted] (7,1.4) -- (7, 1.6);

\filldraw (9,1) circle (1pt); 
\filldraw (9,2) circle (1pt); 
\filldraw (9,3) circle (1pt); 
\filldraw (9,4) circle (1pt); 

\draw[dotted] (9,1.4) -- (9, 1.6);

\scriptsize

\draw[middlearrow={latex}] (3,5) -- (5,1)  node [midway, left] {$\: f_k$};
\draw[color=red,middlearrow={latex}] (3,5) -- (5,2)  node [midway,above=1pt] {$\: f_3$};
\draw[color=red,middlearrow={latex}] (3,5) -- (5,3)  node [midway,above=1pt] {$\: f_2$};
\draw[color=red,middlearrow={latex}] (3,5) -- (5,4)  node [midway, right, above] {$\: f_1$};

\draw[color=blue,middlearrow={latex}] (3,5) -- (1,1)  node [midway, right] {$\: e_k$};
\draw[middlearrow={latex}] (3,5) -- (1,2)  node [midway,left=2pt, above=1pt] {$\: e_3$};
\draw[middlearrow={latex}] (3,5) -- (1,3)  node [midway,above=2pt] {$\: e_2$};
\draw[middlearrow={latex}] (3,5) -- (1,4)  node [midway, right, above] {$\: e_1$};

\draw[color=red,middlearrow={latex}] (7,5) -- (9,1)  node [midway, left] {$\: p_k$};
\draw[middlearrow={latex}] (7,5) -- (9,2)  node [midway,above=1pt] {$\: p_3$};
\draw[middlearrow={latex}] (7,5) -- (9,3)  node [midway,above=1pt] {$\: p_2$};
\draw[middlearrow={latex}] (7,5) -- (9,4)  node [midway, right, above] {$\: p_1$};

\draw[middlearrow={latex}] (7,5) -- (5,1)  node [midway, right] {$\: h_k$};
\draw[color=red,middlearrow={latex}] (7,5) -- (5,2)  node [midway,left=2pt, above=1pt] {$\: h_3$};
\draw[color=red,middlearrow={latex}] (7,5) -- (5,3)  node [midway,above=2pt] {$\: h_2$};
\draw[color=red,middlearrow={latex}] (7,5) -- (5,4)  node [midway, right, above] {$\: h_1$};

\end{tikzpicture}

	  \caption{\small Part of the graph $\mathfrak{D}$ in the proof of Lemma \ref{l2}, in the case $2 \leq i \leq m-4$. The blue edge is not in $\mathfrak{C}$ by assumptions of the lemma, and the red edges are claimed not to be in $\mathfrak{C}$ by the lemma.} \label{fig:6}
	 \end{figure}

	By Lemmas \ref{VE} and \ref{EE} (applied to this case), we have 
	\begin{equation}\label{D1}
	\begin{split}
	R_1(v)=  \sum_{j=1}^k M_{11}(f_j)=\sum_{j=1}^k M_{11}(f_j^{-1})=\sum_{j=1}^k m_j(M_{12}(f_j^{-1})+M_{12}(h_j^{-1}))= \\ = 
	m_k \sum_{j=1}^k (M_{12}(f_j^{-1})+M_{12}(h_j^{-1})) + D,
	\end{split}
	\end{equation}
	where 
	$$
	D= \sum_{j=1}^{k-1} (m_j - m_k)(M_{12}(f_j^{-1})+M_{12}(h_j^{-1})).
	$$
	\begin{rem} \label{rem:2}
	Since $m_1 > m_2 > \ldots > m_{k-1}>m_k$, and $M$-labels are non-negative, we have that $D \geq 0$. Furthermore, $D=0$ if and only if $M_{12}(f_j^{-1})=M_{12}(h_j^{-1})=0$, for all $j=1, \ldots, k-1$, if and only if {\rm(by Equation (\ref{M0}))} $f_j, h_j \notin \CC$, for all $j=1, \ldots, k-1$.
	\end{rem}
	
	Denote the vertex $(i+2,c)$ by $u$. We continue Equation (\ref{D1}),
	\begin{equation}\label{D2}
		\begin{split}
		R_1(v)= m_k \sum_{j=1}^k (M_{12}(f_j^{-1})+M_{12}(h_j^{-1})) + D =m_k \sum_{j=1}^k M_{21}(f_j)+ m_k \sum_{j=1}^k M_{21}(h_j) + D = \\ = m_k (R_2(v)+R_2(u))+D = m_k (k-1) (M_{22}(f_k)+M_{22}(h_k)+M_{22}(p_k))+D,
		\end{split}
	\end{equation}
	where the last equality holds since $R_2(u)=(k-1)(M_{22}(h_k)+M_{22}(p_k))$; if $i=1$, then $R_2(v)=(k-1)M_{22}(f_k)$, and if $i \geq 2$, then still $R_2(v)=(k-1)(M_{22}(e_k)+M_{22}(f_k))=(k-1)M_{22}(f_k)$, because $e_k \notin \CC$.
	 
In the same way, we have $R_1(v)=(k-1)M_{12}(f_k)$. So 
\begin{equation}\label{D3}
\begin{split}
	R_1(v)=(k-1)M_{12}(f_k)=(k-1)M_{21}(f_k^{-1})=m_k(k-1)(M_{22}(f_k^{-1})+M_{22}(h_k^{-1}))= \\ 
	= m_k(k-1)(M_{22}(f_k)+M_{22}(h_k)).
	\end{split}
\end{equation}	
	Comparing Equations (\ref{D2}) and (\ref{D3}), we obtain that $m_k(k-1)M_{22}(p_k)+D=0$, but both $M_{22}(p_k)$ and $D$ are non-negative, so $M_{22}(p_k)=D=0$. By Equation (\ref{M0}), this implies that $p_k \notin \CC$. By Remark \ref{rem:2}, it follows that $f_j, h_j \notin \CC$, for all $j=1, \ldots, k-1$. This is exactly the claim of the lemma. 
	\end{proof}

	Suppose first that $m$ is odd. By Remark \ref{rem:f3}, we can suppose that $\CC$ lies inside $\DD_1$, which is the connected component of $\DD$ containing $(1,c)$. Then the vertex $v=(1,c)$ satisfies the conditions of Lemma \ref{l2}, so by this lemma the edge $e_{1,c}^{2,1}$ is not in $\CC$. This is a contradiction with Lemma \ref{DD}, since local surjectivity at $v$ of the projection $\pi_2$ does not hold. Thus, for odd $m$, the group $\GG(P_m)$ is not commensurable with a RAAG defined by a tree of diameter $4$.
	
	So we can suppose that $m$ is $0$ modulo $4$, in particular, $m \geq 8$. There are two cases --- either $\CC$ lies inside $\DD_1$, or inside $\DD_2$. If $\CC$ lies inside $\DD_1$, then again applying Lemma \ref{l2} at the vertex $v=(1,c)$ results in a contradiction, as in the case of odd $m$. Thus, we can assume that $\CC$ lies inside $\DD_2$. Note that the vertices $(2,c)$ and $(m-2,c)$ belong to $\DD_2$.
	
	\begin{lemma}\label{tr}
		If $m\equiv 0 \,(\modu 4)$ and $\CC$ lies inside $\DD_2$, then the edges $e_{2,c}^{1,j}$ and $e_{m-2,c}^{m-1,j}$, are not in $\CC$, where $j=2,\ldots,k$.
	\end{lemma}
	\begin{proof}
		The proof is similar to that of Lemma \ref{l2}.  
		By symmetry, it suffices to prove that the edges $e_{2,c}^{1,j}$, for $j=2,\ldots,k$, are not in $\CC$. To abbreviate the notation, denote the edges $e_{2,c}^{1,j}$ by $e_j$, the edges $e_{2,c}^{3,j}$ by $f_j$, and the edges $e_{4,c}^{3,j}$ by $h_j$, $j=1, \ldots, k$. Denote the vertex $(2,c)$ by $v$, see Figure \ref{fig:7}. 
		
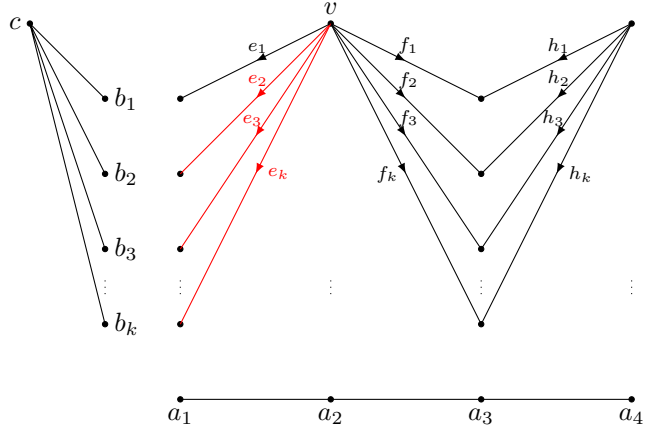
\begin{figure}[!h]
\begin{tikzpicture}

\draw (1,0) -- (7,0);

\filldraw (1,0) circle (1pt)  node[align=left, below] {$a_{1}$};
\filldraw (3,0) circle (1pt)  node[align=left, below] {$a_2$};
\filldraw (5,0) circle (1pt)  node[align=left, below] {$a_{3}$};
\filldraw (7,0) circle (1pt)  node[align=left, below] {$a_{4}$};

\filldraw (0,1) circle (1pt)  node[align=right, right] {$b_k$};
\filldraw (0,2) circle (1pt)  node[align=right, right] {$b_3$};
\filldraw (0,3) circle (1pt)  node[align=right, right] {$b_2$};
\filldraw (0,4) circle (1pt)  node[align=right, right] {$b_1$};
\draw[dotted] (0,1.4) -- (0, 1.6);
\filldraw (-1,5) circle (1pt)  node[align=left, left] {$c$};
\draw (-1,5) -- (0,1) (-1,5) -- (0,2) (-1,5) -- (0,3) (-1,5) -- (0,4);

\filldraw (1,1) circle (1pt); 
\filldraw (1,2) circle (1pt); 
\filldraw (1,3) circle (1pt); 
\filldraw (1,4) circle (1pt); 

\draw[dotted] (1,1.4) -- (1, 1.6);

\filldraw (3,5) circle (1pt)  node[align=right, above] {$v$};
\draw[dotted] (3,1.4) -- (3, 1.6);

\filldraw (5,1) circle (1pt); 
\filldraw (5,2) circle (1pt); 
\filldraw (5,3) circle (1pt); 
\filldraw (5,4) circle (1pt); 

\draw[dotted] (5,1.4) -- (5, 1.6);

\filldraw (7,5) circle (1pt); 
\draw[dotted] (7,1.4) -- (7, 1.6);

\scriptsize

\draw[middlearrow={latex}] (3,5) -- (5,1)  node [midway, left] {$\: f_k$};
\draw[middlearrow={latex}] (3,5) -- (5,2)  node [midway,above=1pt] {$\: f_3$};
\draw[middlearrow={latex}] (3,5) -- (5,3)  node [midway,above=1pt] {$\: f_2$};
\draw[middlearrow={latex}] (3,5) -- (5,4)  node [midway, right, above] {$\: f_1$};

\draw[color=red,middlearrow={latex}] (3,5) -- (1,1)  node [midway, right] {$\: e_k$};
\draw[color=red,middlearrow={latex}] (3,5) -- (1,2)  node [midway,left=2pt, above=1pt] {$\: e_3$};
\draw[color=red,middlearrow={latex}] (3,5) -- (1,3)  node [midway,above=2pt] {$\: e_2$};
\draw[middlearrow={latex}] (3,5) -- (1,4)  node [midway, right, above] {$\: e_1$};

\draw[middlearrow={latex}] (7,5) -- (5,1)  node [midway, right] {$\: h_k$};
\draw[middlearrow={latex}] (7,5) -- (5,2)  node [midway,left=2pt, above=1pt] {$\: h_3$};
\draw[middlearrow={latex}] (7,5) -- (5,3)  node [midway,above=2pt] {$\: h_2$};
\draw[middlearrow={latex}] (7,5) -- (5,4)  node [midway, right, above] {$\: h_1$};

\end{tikzpicture}
	  \caption{\small Part of the graph $\mathfrak{D}_2$ in the proof of Lemma \ref{tr}. The red edges are claimed not to be in $\mathfrak{C}$ by the lemma.} \label{fig:7}
	 \end{figure} 	
		
		Note that, by Lemma \ref{VE}, we have that $M_{11}(e_j^{-1})=m_jM_{12}(e_j^{-1})$ and $M_{21}(e_j^{-1})=m_jM_{22}(e_j^{-1})$ for all $j=1,\ldots,k$. Then, we have
		\begin{equation}\label{C1}
			R_1(v)= \sum_{j=1}^k M_{11}(e_j) = \sum_{j=1}^k M_{11}(e_j^{-1})= \sum_{j=1}^k m_j M_{12}(e_j^{-1})=m_1 \sum_{j=1}^k M_{12}(e_j^{-1}) - C,
		\end{equation}
		where 
		$$
		C= \sum_{j=2}^k (m_1-m_j)M_{12}(e_j^{-1}).
		$$
	Since $m_1 > m_2 > \ldots > m_{k-1}>m_k$, and $M$-labels are non-negative, we have $C \geq 0$, and $C=0$ if and only if $M_{12}(e_j^{-1})=0$, for all $j=2, \ldots, k$, if and only if (by Equation (\ref{M0})) $e_j \notin \CC$, for all $j=2, \ldots, k$.
		
		We continue Equation (\ref{C1}),
		\begin{equation}\label{C2}
		\begin{split}
		R_1(v)=m_1 \sum_{j=1}^k M_{12}(e_j^{-1}) - C= m_1 \sum_{j=1}^k M_{21}(e_j)-C =m_1R_2(v)-C= \\ =m_1(k-1)(M_{22}(e_1)+M_{22}(f_1))-C.		
		\end{split}
		\end{equation}
		On the other hand, we have
		\begin{equation}\label{C3}
		\begin{split}
		R_1(v)=(k-1)(M_{12}(e_1)+M_{12}(f_1))=(k-1)(M_{21}(e_1^{-1})+M_{21}(f_1^{-1}))= \\ =
		m_1(k-1)(M_{22}(e_1^{-1})+M_{22}(f_1^{-1})+M_{22}(h_1^{-1}))= \\ =m_1(k-1)(M_{22}(e_1)+M_{22}(f_1)+M_{22}(h_1)),
		\end{split}
		\end{equation}
		since $M_{21}(f_1^{-1})=M_{22}(f_1^{-1})+M_{22}(h_1^{-1})$ by Lemma \ref{VE}.
		
		Comparing (\ref{C2}) and (\ref{C3}), we obtain that $C + m_1(k-1)M_{22}(h_1)=0$, but both $C$ and $M_{22}(h_1)$ are non-negative, so $C=M_{22}(h_1)=0$. This means that $h_1 \notin \CC$, but also, by Remark \ref{rem:2}, that $e_j \notin \CC$ for $j=2, \ldots,k$. This is exactly the claim of the lemma.
	\end{proof}
	
	We now turn to the proof of Theorem \ref{non-com}. As we already mentioned, we use Lemma \ref{tr} to deduce the absence of certain edges and Lemma \ref{l2} to recursively remove other edges. We combine these two lemmas until we assure that there is no local surjectivity contradicting Lemma \ref{DD}, see Figure \ref{fig:8}.
	
		 \begin{figure}[!h]
		\begin{tikzpicture}[scale=1.5]
		\tikzset{near start abs/.style={xshift=1cm}}
		\small
		\draw (1.5,1) -- (6.5,1);

		\filldraw (1.5,1) circle (1pt)  node[align=left, below] {$a_1$};
		\filldraw (2,1) circle (1pt)  node[align=left, below] {$a_2$};
		\filldraw (2.5,1) circle (1pt)  node[align=left, below] {$a_3$};
		\filldraw (3,1) circle (1pt)  node[align=left, below] {$a_4$};
		\filldraw (3.5,1) circle (1pt)  node[align=left, below] {$a_5$};
		\filldraw (4,1) circle (1pt)  node[align=left, below] {$a_6$};
		\filldraw (4.5,1) circle (1pt)  node[align=left, below] {$a_7$};
		\filldraw (5,1) circle (1pt)  node[align=left, below] {$a_8$};
		\filldraw (5.5,1) circle (1pt)  node[align=left, below] {$a_9$};
		\filldraw (6,1) circle (1pt)  node[align=left, below] {$a_{10}$};
		\filldraw (6.5,1) circle (1pt)  node[align=left, below] {$a_{11}$};

		\draw  (1,1.5) -- (1,2.5);

		\filldraw (1,1.5) circle (1pt)  node[align=left, left]  {$b_2$};
		\filldraw (1,2) circle (1pt)  node[align=left, left]  {$c$};
		\filldraw (1,2.5) circle (1pt)  node[align=left, left] {$b_1$};

		\filldraw (1.5,2.5) circle (1pt);
		\filldraw (2.5,1.5) circle (1pt);
		\filldraw (3.5,2.5) circle (1pt);
		\filldraw (4.5,1.5) circle (1pt);
		\filldraw (5.5,2.5) circle (1pt);

		\draw(1.5,2.5) -- (2.5,1.5);
		\draw(2.5,1.5) -- (3.5,2.5);
		\draw(3.5,2.5) -- (4.5,1.5);
		\draw (4.5,1.5) -- (5.5,2.5);
		\draw (5.5,2.5) -- (6,2);

		\filldraw (1.5,1.5) circle (1pt);
		\filldraw (2.5,2.5) circle (1pt);
		\filldraw (3.5,1.5) circle (1pt);
		\filldraw (4.5,2.5) circle (1pt);
		\filldraw (5.5,1.5) circle (1pt);
		\filldraw (6.5,2.5) circle (1pt);
		\filldraw (6.5,1.5) circle (1pt);
		
		\filldraw (2,2) circle (1pt);
		\filldraw (3,2) circle (1pt);
		\filldraw (4,2) circle (1pt);
		\filldraw (5,2) circle (1pt);
		\filldraw (6,2) circle (1pt);

		\tikzset{middlearrow/.style={
				decoration={markings,
					mark= at position 0.6 with {\arrow{#1}} ,
				},
				postaction={decorate}
			}
		}
		
		\draw[color=red, middlearrow={latex}] (1.5,1.5) -- (2,2)  node [midway, right, below] {$\: e_{1,2}^{2,c}$};
		\draw[color=red, middlearrow={latex}] (2,2) -- (2.5,2.5)  node [midway, right=2pt, below] {$\: e_{2,c}^{3,1}$};
		\draw[color=red, middlearrow={latex}] (2.5,2.5) -- (3,2)  node [midway, right, above] {$\: e_{3,1}^{4,c}$};
		\draw[color=red, middlearrow={latex}] (3,2) -- (3.5,1.5)  node [midway, right=2pt, above] {$\: e_{4,c}^{5,2}$};
		
		\draw[color=red, middlearrow={latex}] (3.5,1.5) -- (4,2)  node [midway, right, below] {$\: e_{5,2}^{6,c}$};
		\draw[color=red, middlearrow={latex}] (4,2) -- (4.5,2.5)  node [midway, right=2pt, below] {$\: e_{6,c}^{7,1}$};
		\draw[color=red, middlearrow={latex}] (4.5,2.5) -- (5,2)  node [midway, right, above] {$\: e_{7,1}^{8,c}$};
		\draw[color=red, middlearrow={latex}] (5,2) -- (5.5,1.5)  node [midway, right=2pt, above] {$\: e_{8,c}^{9,2}$};
		
		\draw[color=red, middlearrow={latex}] (5.5,1.5) -- (6,2)  node [midway, right=2pt, below] {$\: e_{9,2}^{10,c}$};
		\draw[color=red, middlearrow={latex}] (6,2) -- (6.5,2.5)  node [midway, right, right] {$\: e_{10,c}^{11,1}$};
		\draw[color=red, middlearrow={latex}] (6,2) -- (6.5,1.5)  node [midway, right, right] {$\: e_{10,c}^{11,2}$};

		\end{tikzpicture}
		\caption{\small The graph $\mathfrak{D}_2$ for $m=12$ in the end of the proof of Theorem \ref{non-com}. The red edges are proved not to be in $\mathfrak{C}$, which leads to a contradiction.} \label{fig:8}
	\end{figure}
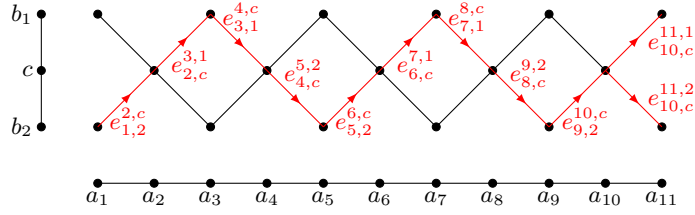

	By Lemma \ref{tr}, we see that the edges $e_{2,c}^{1,j}$ are not in $\CC$ for $j=2, \ldots, k$. This means that we can apply Lemma \ref{l2} to the vertex $v=(2,c)$ of $\DD_2$, and we conclude that the edges $e_{2,c}^{3,j'}$, for $j'=1,\ldots,k-1$, are not in $\CC$. If $k \geq 3$, then we get a contradiction with Lemma \ref{DD}: local surjectivity at $v$ of the projection $\pi_2$ does not hold, since neither $e_{2,c}^{1,2}$ nor $e_{2,c}^{3,2}$ are in $\CC$. Thus we can suppose that $k=2$. 
	
	We know that $e_{1,2}^{2,c} \notin \CC$, and by Lemma \ref{l2} the edges $e_{2,c}^{3,1}$, $e_{3,1}^{4,c}$ and $e_{4,c}^{5,2}$ are not in $\CC$. By Lemma \ref{DD} applied at the vertex $(5,2)$, we also see that $e_{5,2}^{6,c} \notin \CC$.  
	 We claim that the edges $e_{4a-2,c}^{4a-1,1}, \: e_{4a-1,1}^{4a,c}, \: e_{4a,c}^{4a+1,2}, \: e_{4a+1, 2}^{4a+2,c}$ are not in $\CC$, for $a=1, \ldots, m/4-1$. We prove it by induction on $a$.  Note that the claim holds for $a=1$, as proved above. Suppose it holds for $a \leq a'$, $1 \leq a' \leq m/4-2$, and we prove it for $a=a'+1$. Since $e_{4a'+2,c}^{4a'+1,2} \notin \CC$ by induction hypothesis for $a=a'$, by Lemma \ref {l2}, applied at the vertex $(4a'+2, c)$, we see that $e_{4a'+2,c}^{4a'+3,1}, \: e_{4a'+3,1}^{4a'+4,c}, \: e_{4a'+4,c}^{4a'+5,2}, \:  \notin \CC$. Finally, by Lemma \ref{DD}, we get $e_{4a'+5,2}^{4a'+6,c} \notin \CC$, and this proves the claim. 
	
	It follows from the claim with $a=m/4-1$ that $e_{m-3,2}^{m-2,c} \notin \CC$. By Lemma \ref{l2}, applied at the vertex $(m-2, c)$, we have that $e_{m-2,c}^{m-1,1} \notin \CC$. But by Lemma \ref{tr}, $e_{m-2,c}^{m-1,2} \notin \CC$, and this is a contradiction with Lemma \ref{DD}, see Figure \ref{fig:8}.

Thus, if $m$ is not $2$ modulo $4$, then $\GG(P_m)$ is not commensurable to a tree of diameter $4$. This proves the theorem.
\end{proof}

Note that in this way we were able to show that for $P_{4k+2}$ and $T_{k,k+1}$ the system has positive integer solutions and calculate these solutions, and this gave us a hint on how to construct the corresponding isomorphic finite index subgroups given in Section \ref{sec:5}.

\section{Commensurability of RAAGs defined by paths and trees of diameter $4$} \label{sec:5}
	
In this section, we characterise when a RAAG defined by a path $P_n$ is commensurable to a RAAG defined by a tree of diameter $4$. In Section \ref{sec:4}, we have seen that a necessary condition for commensurability is that $n \equiv 2\, (\modu  4)$. In this section, we show that this is a sufficient condition.  
	
Recall that by $T_{k,k+1}$ we  denote a tree of diameter $4$, with the central vertex of degree $2$ and so that the two vertices adjacent to the central vertex have  degrees $k+1$ and $k+2$ correspondingly, $T_{k,k+1}$ has $2k+1$ leaves, see Figure \ref{fig:T4}.

\begin{figure}[!h]		
\begin{tikzpicture}

\draw (1,0) -- (9,0);

\filldraw (1,0) circle (1pt)  node[align=left, left] {$a_1$};
\filldraw (3,0) circle (1pt)  node[align=left, below] {$b$};
\filldraw (5,0) circle (1pt)  node[align=left, below=2pt] {$c$};
\filldraw (7,0) circle (1pt)  node[align=left, below] {$d$};
\filldraw (9,0) circle (1pt)  node[align=right,right] {$e_1$};
\filldraw (1,1) circle (1pt)  node[align=left, left] {$a_k$};
\filldraw (9,1) circle (1pt)  node[align=right, right] {$e_k$};
\filldraw (9,1.5) circle (1pt)  node[align=right, right] {$e_{k+1}$};

\draw[dotted] (1,0.4)--(1,0.6) (9,0.4)--(9,0.6);

\draw (3,0) -- (1,1) (7,0) -- (9,1) (7,0) -- (9,1.5);

\end{tikzpicture}
	  \caption{\small The tree $T_{k,k+1}$.} \label{fig:T4}
	 \end{figure}
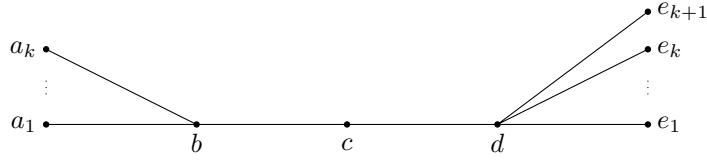

\begin{theorem}\label{com}
	 Let $k \geq 1$. Then $\GG(P_{4k+2})$ is commensurable to $\GG(T_{k, k+1})$. 
\end{theorem}

Note that Theorem \ref{com} together with Theorem \ref{non-com} immediately imply Theorem \ref{th2}.

The remaining part of this section will be devoted to the proof of Theorem \ref{com}.
In order to do so, we first define an abstract group as a fundamental group of a certain graph of groups $X$; we then exhibit finite index subgroups $H$ and $K$ of $\GG(T_{k,k+1})$ and $\GG(P_{4k+2})$ respectively and show that they are isomorphic to that abstract group. We define the subgroups by describing them as fundamental groups of finite covers of the Salvetti complexes of the corresponding RAAGs. We divide the proof into five subsections. In Section \ref{sec:51} we construct $X$, in Section \ref{sec:52} we construct $H$, in Section \ref{sec:53} we prove that $H$ is isomorphic to $\pi_1(X)$, in Section \ref{sec:54} we construct $K$, and in in Section \ref{sec:55} we prove that $K$ is isomorphic to $\pi_1(X)$.

Throughout this section we always denote the conjugation as follows: $g^h=hgh^{-1}$. Also, in a group $G$, we denote the centralizer of an element $g$ in $G$ by $C(g)$, and the centralizer of $g$ in a subgroup $H$ of $G$ by $C_H(g)$. We will occasionally use this notation even in the case when $g$ is not in $H$ (for some finite index subgroup $H$), and in this case, since centralizers in RAAGs are isolated (see \cite{CKZ}), we have $C_H(g)=H \cap C(g)=C_H(g^n)$, where $n$ is the minimal positive integer such that $g^n \in H$.

We will use basic facts from Bass-Serre theory, the reader is referred to \cite{Serre} for details.

\subsection{Construction of the graph of groups X} \label{sec:51}

Recall that by $F(A)$ we mean the free group on $A$. We begin by defining a graph of groups $X$. The graph of groups $X$ is built from some simpler pieces $D_i$, which are also graphs of groups. We begin by describing these pieces.  

Let $D_i$, $i=1,\dots, k-1$ be the ``diamond'' graph with $i+2$ vertices, namely $v,w, u_1, \dots, u_i$, where the vertices $u_1,\dots, u_i$ have degree $2$ and each of the vertices $u_i$ is adjacent to the two vertices $v$ and $w$ of degree $i$, see Figure \ref{fig:Di}. 

\begin{figure}[!h]		
\begin{tikzpicture}

\draw (0,0) -- (1,1) -- (2,0) -- (1,0.5) -- (0,0) -- (1,-1)--(2,0);

\filldraw (0,0) circle (1pt)  node[align=left, left] {$v$};
\filldraw (2,0) circle (1pt)  node[align=right,right] {$w$};
\filldraw (1,1) circle (1pt)  node[align=right,above] {$u_1$};
\filldraw (1,0.5) circle (1pt)  node[align=right,below=2pt] {$u_2$};
\filldraw (1,-1) circle (1pt)  node[align=right,below] {$u_i$};

\draw[dotted] (1,-0.1)--(1,-0.3);

\draw (4,0)--(5,1)--(6,1) (4,0)--(5,0.5)--(6,0.5) (4,0)--(5,-1)--(6,-1);
\filldraw (4,0) circle (1pt)  node[align=left, left] {$v$};
\filldraw (5,1) circle (1pt)  node[align=right,above] {$u_1$};
\filldraw (5,0.5) circle (1pt)  node[align=right,below=2pt] {$u_2$};
\filldraw (5,-1) circle (1pt)  node[align=right,below] {$u_k$};
\filldraw (6,1) circle (1pt)  node[align=right,right] {$u_1'$};
\filldraw (6,0.5) circle (1pt)  node[align=right,right] {$u_2'$};
\filldraw (6,-1) circle (1pt)  node[align=right,right] {$u_k'$};

\draw[dotted] (5,-0.1)--(5,-0.3);
\draw[dotted] (6,-0.1)--(6,-0.3);

\end{tikzpicture}
	  \caption{\small The ``diamond'' graph $D_i$, $1 \leq i \leq k-1$ (on the left) and the graph $D_k$ (on the right).} \label{fig:Di}
	 \end{figure}
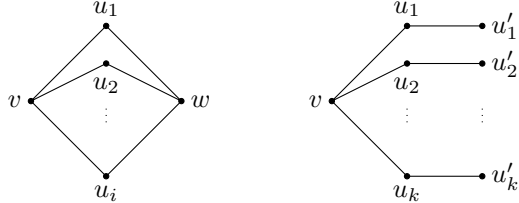

Let the vertex group at $u_j$ be 
$$
V(u_j)=\langle r_j\rangle \times F(f_j, g_j).
$$
Let the vertex groups at $v$ and $w$ be as follows,
$$
\begin{array}{ll}
V(v)=&\langle s\rangle \times F(x_1, \dots, x_{k^2}, y_1,\dots, y_{k+1-i}, z_1,\dots, z_i);\\
V(w)=&\langle s'\rangle \times F(x_1', \dots, x_{k^2+1}', y_1',\dots, y_{k-i}', z_1',\dots, z_i').
\end{array}
$$

All edge groups are isomorphic to $\BZ^2=\langle p,q\rangle$. 

The embedding of $\BZ^2$ of the edge $(v,u_j)$ into $V(v)$, is defined by the map 
$$
p \mapsto s, \
q\mapsto z_j,
$$
where $j=1, \dots, i$. The embedding of $\BZ^2$ of the edge $(v,u_j)$ into $V(u_j)$, is defined by the map 
$$
p \mapsto f_j, \
q\mapsto r_j,
$$
where $j=1, \dots, i$. The embedding of $\BZ^2$ of the edge $(w,u_j)$ into $V(u_j)$, is defined by the map 
$$
p \mapsto g_j, \ 
q\mapsto r_j,
$$
where $j=1, \dots, i$. The embedding of $\BZ^2$ of the edge $(w,u_j)$ into $V(w_j)$, is defined by the map 
$$
p \mapsto s', \
q\mapsto z_j',
$$
where $j=1, \dots, i$.

Let $D_k$ be the graph with $2k+1$ vertices defined as follows: it has $1$ vertex $v$ of degree $k$ adjacent to $k$ vertices $u_1,\dots, u_k$ of degree $2$; each vertex $u_i$ is adjacent to a vertex $u_i'$ of degree $1$, $i=1, \dots, k$, see Figure \ref{fig:Di}.

The vertex group $V(v)$ at $v$ is defined to be
$$
V(v)=\langle s\rangle \times F(x_1,\dots, x_{k^2}, z_1,\dots, z_k, y_1)
$$
Let the vertex groups at $u_i$ and $u_i'$ be as follows,
$$
V(u_j)  =  \langle r_j  \rangle \times F(f_j,g_j); \ \ \quad  V(u_j') =  \langle r_j' \rangle \times F(x_{j1}', \dots, x_{j,k+1}', z_j').
$$

All edge groups are isomorphic to $\BZ^2=\langle p,q\rangle$. The embedding of $\BZ^2$ of the edge  $(v,u_j)$ into $V(v)$, is defined by the map 
$$
p \mapsto s, \ q\mapsto z_j, \hbox{ where } j=1, \dots, k. 
$$
The embedding of $\BZ^2$ of the edge $(v,u_j)$ into $V(u_j)$, is defined by the map 
$$
p \mapsto f_j, \ q\mapsto r_j, \hbox{ where } j=1, \dots, k.
$$
The embedding of $\BZ^2$ of the edge $(u_j,u_j')$ into $V(u_j)$, is defined by the map 
$$
p \mapsto r_j, \ q\mapsto g_j, \hbox{ where } j=1, \dots, k. 
$$
The embedding of $\BZ^2$ of the edge $(u_j,u_j')$ into $V(u_j')$, is defined by the map 
$$
p \mapsto z_j', \ q\mapsto r_j', \hbox{ where } j=1, \dots, k.
$$

We now consider the graph of groups $X$ obtained by identifying the vertices $v$ and $w$ in the graphs of groups $D_i$ in the following sequence 
\begin{equation}\label{eq:diseq}
(D_k, D_1, D_{k-1}, D_2, \dots, D_{k-1}, D_1, D_k), 
\end{equation}
where $D_k$ and $D_1$  are identified along $v$, $D_1$ and $D_{k-1}$ are identified along $w$, $D_{k-1}$ and $D_2$ are identified along $v$ etc, see Figure \ref{fig:X}. This defines the graph of groups $X$.

\tikzset{node distance=2.5cm, 
every state/.style={minimum size=1pt, scale=0.15, 
semithick,
fill=black},
initial text={}, 
double distance=2pt, 
every edge/.style={ 
draw,  
auto}}

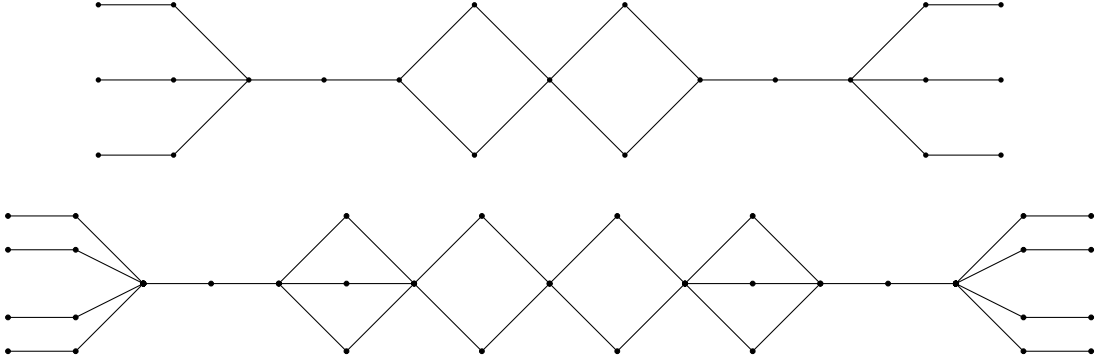
\begin{figure}[!h]	

\begin{tikzpicture}
	
\node[state] (1) at (0,0){};
\node[state] (2) at (-1,1) {};
\node[state] (3) at (-2,1) {};
\node[state] (4) at (-1,0) {};
\node[state] (5) at (-2,0) {};
\node[state] (6) at (-1,-1) {};
\node[state] (7) at (-2,-1) {};
\node[state] (8) at (1,0) {};
\node[state] (9) at (2,0) {};
\node[state] (10) at (3,1) {};
\node[state] (11) at (4,0) {};
\node[state] (12) at (3,-1) {};
\node[state] (13) at (5,1) {};
\node[state] (14) at (6,0) {};
\node[state] (15) at (5,-1) {};
\node[state] (16) at (7,0) {};
\node[state] (17) at (8,0) {};
\node[state] (18) at (9,0) {};
\node[state] (19) at (10,0) {};
\node[state] (20) at (9,1) {};
\node[state] (21) at (10,1) {};
\node[state] (22) at (9,-1) {};
\node[state] (23) at (10,-1) {};

\draw
(1) edge (2)
(2) edge (3)
(1) edge (4)
(4) edge (5)
(1) edge (6)
(6) edge (7)
(1) edge (8)
(8) edge (9)
(9) edge (10)
(10) edge (11)
(11) edge (12)
(12) edge (9)
(11) edge (13)
(13) edge (14)
(14) edge (15)
(15) edge (11)
(14) edge (16)
(16) edge (17)
(17) edge (18)
(18) edge (19)
(17) edge (20)
(20) edge (21)
(17) edge (22)
(22) edge (23);

\end{tikzpicture}

\vspace{0.7 cm}

\begin{tikzpicture}[scale=0.9]

\draw (0,0) -- (-1,1)--(-2,1) (0,0)--(-1,0.5)--(-2,0.5) (0,0)--(-1,-0.5)--(-2,-0.5) (0,0)--(-1,-1)--(-2,-1);
\draw (0,0) -- (1,0)-- (2,0) --(3,1)--(4,0)--(3,0)--(2,0)--(3,-1)--(4,0);
\draw (4,0)--(5,1)--(6,0)--(5,-1)--(4,0);
\draw (6,0)--(7,1)--(8,0)--(7,-1)--(6,0);
\draw (8,0)--(9,1)--(10,0)--(9,0)--(8,0)--(9,-1)--(10,0);
\draw (10,0)--(11,0)--(12,0);
\draw (12,0)--(13,1)--(14,1) (12,0)--(13,0.5)--(14,0.5) (12,0)--(13,-0.5)--(14,-0.5) (12,0)--(13,-1)--(14,-1);

\filldraw  
(0,0) circle (1pt)  (-1,1)circle (1pt) (-2,1) circle (1pt)  (0,0)circle (1pt) (-1,0.5)circle (1pt) (-2,0.5) circle (1pt)  (0,0)circle (1pt) (-1,-0.5)circle (1pt) (-2,-0.5) circle (1pt)  (0,0)circle (1pt) (-1,-1)circle (1pt) (-2,-1) circle (1pt) 
(0,0) circle (1pt)  (1,0) circle (1pt)  (2,0) circle (1pt) (3,1)circle (1pt) (4,0)circle (1pt) (3,0)circle (1pt) (2,0)circle (1pt) (3,-1)circle (1pt) (4,0)
(4,0)circle (1pt) (5,1)circle (1pt) (6,0)circle (1pt) (5,-1)circle (1pt) (4,0)
(6,0)circle (1pt) (7,1)circle (1pt) (8,0)circle (1pt) (7,-1)circle (1pt) (6,0)
(8,0)circle (1pt) (9,1)circle (1pt) (10,0)circle (1pt) (9,0)circle (1pt) (8,0)circle (1pt) (9,-1)circle (1pt) (10,0)
(10,0)circle (1pt) (11,0)circle (1pt) (12,0) circle (1pt) 
(12,0)circle (1pt) (13,1)circle (1pt) (14,1) circle (1pt)  (12,0)circle (1pt) (13,0.5)circle (1pt) (14,0.5) circle (1pt)  (12,0)circle (1pt) (13,-0.5)circle (1pt) (14,-0.5) circle (1pt)  (12,0)circle (1pt) (13,-1)circle (1pt) (14,-1) circle (1pt) ;

\end{tikzpicture}
	  \caption{\small The underlying graph for the graph of groups $X$ in the cases $k=3$ (above) and $k=4$ (below).} \label{fig:X}
	 \end{figure}

\subsection{Construction of the finite index subgroup $H$ of $\GG(T_{k, k+1})$} \label{sec:52}

Recall that the tree $T_{k,k+1}$ has a central vertex $c$ of degree 2, two adjacent vertices $b$ and $d$ of degree $k+1$ and $k+2$ respectively and leaves $a_i$ adjacent to $b$, $i=1, \dots, k$ and $e_j$ adjacent to $d$, $j=1, \dots, k+1$, see Figure \ref{fig:T4}.

The subgroup $H$ is defined as the full preimage under the natural epimorphism $\pi:\GG(T_{k,k+1}) \to F(a_1, e_1)$ of a finite index subgroup $H'$ of $F(a_1,e_1)$, that is $H:= \pi^{-1}(H')$, where $H'<_{fi}F(a_1, e_1)$.

The subgroup $H'$ is defined as the fundamental group of a finite cover $S$ of degree $k(k+1)$ of the bouquet of two circles. The aforementioned cover is defined as follows. 

Let $P_i$ be $k+1$-cycles labelled by $a_1$, $i=1,\dots, k$, and $Q_i$ be $k$-cycles labelled by $e_1$, $i=1,\ldots,k-1$. We now glue these cycles according to the following pattern:
\begin{itemize}
\item Identify $P_1$ and $Q_1$ by $1$ vertex, the basepoint;
\item Attach $P_2$ onto $Q_1$ by identifying $k-1$ vertices;
\item Attach $Q_2$ onto $P_2$ (attached in the previous step) by identifying $2$ (consecutive) vertices;
\item Attach $P_3$ onto $Q_2$ (attached in the previous step) by identifying $k-2$ (consecutive) vertices;
\item \dots
\item Attach $Q_{k-1}$ onto $P_{k-1}$ (attached in the previous step) by identifying $k-1$ vertices;
\item Attach $P_{k}$ onto $Q_{k-1}$ (attached in the previous step) by identifying $1$ vertex.
\item Add loops labelled by $e_1$ at all the $k$ vertices of $P_1$ which are not the basepoint (and so are not on $Q_1$), and similar for all the $k$ vertices of $P_k$ which are not on $Q_{k-1}$.
\end{itemize}
The attachments are always performed in such a way that the vertices in the intersection of $Q_i$ and $P_i$ appear in those cycles in the opposite order, and similar for the intersection of $P_{i}$ and $Q_{i-1}$, for all $i=2, \ldots,k-1$. This defines $S$, and so also $H$. 

By construction, the cycles $P_l$ and $P_m$, $l\neq m$, $Q_l$ and $Q_m$, $l\neq m$, $P_l$ and $Q_m$, $l\neq m+1, l \neq m,$ do not share any vertices. We refer the reader to Figure \ref{fig:ex4} for the construction of the cover $S$ in the cases $k=3$ and $k=4$. 

Note that the vertices of $S$ are in one-to-one correspondence with the right cosets of $H'$ as follows: if $g$ is the label of any path from the basepoint to a vertex $v$ in $S$, then $v$ corresponds to $H'g$. In other terms, $S$ is the Schreier graph of $H'$ in $F(a_1,e_1)$.

\tikzset{node distance=2.5cm, 
every state/.style={minimum size=1pt, scale=0.3, 
semithick,
fill=black},
initial text={}, 
double distance=2pt, 
every edge/.style={color=red, 
draw, 
->,>=latex', 
auto,
semithick}}

\begin{figure}[!h]	

\begin{tikzpicture}
	
\node[state] (1) at (0,0){};
\node[state] (2) at (-1,1) {};
\node[state] (3) at (-2,0) {};
\node[state] (4) at (-1,-1) {};
\node[state] (5) at (1,1) {};
\node[state] (6) at (3,1) {};
\node[state] (7) at (3,-1) {};
\node[state] (8) at (1,-1) {};
\node[state] (9) at (4,0) {};
\node[state] (10) at (5,1) {};
\node[state] (11) at (6,0) {};
\node[state] (12) at (5,-1) {};

\node at (0.2,0) {*};

\draw
(1) edge (2) 
(2) edge (3)
(3) edge (4)
(4) edge (1)
(5) edge[bend left] (6)
(6) edge[bend left] (7)
(7) edge[bend left] (8)
(8) edge[bend left] (5)
(9) edge (10)
(10) edge (11)
(11) edge (12)
(12) edge (9)

(1) edge[color=blue] (5)
(5) edge[color=blue, bend left] (8)
(8) edge[color=blue] (1)
(6) edge[color=blue] (9)
(9) edge[color=blue] (7)
(7) edge[color=blue, bend left] (6)

(2) edge[color=blue, in=135,out=45,looseness=30] (2)
(3) edge[color=blue, in=225,out=135,looseness=30] (3)
(4) edge[color=blue, in=315,out=225,looseness=30] (4)

(10) edge[color=blue, in=45,out=135,looseness=30] (10)
(11) edge[color=blue, in=-45,out=45,looseness=30] (11)
(12) edge[color=blue, in=225,out=315,looseness=30] (12)

;
\end{tikzpicture}

\vspace{0.3cm}

\begin{tikzpicture}
\node[state] (1) at (0,0){};
\node[state] (2) at (-1,1) {};
\node[state] (3) at (-2,0) {};
\node[state] (4) at (-1.5,-1) {};
\node[state] (5) at (-0.5,-1) {};
\node[state] (6) at (1,1) {};
\node[state] (7) at (2,0) {};
\node[state] (8) at (1,-1) {};
\node[state] (9) at (3,1) {};
\node[state] (10) at (3,-1) {};
\node[state] (11) at (4,1) {};
\node[state] (12) at (4,-1) {};
\node[state] (13) at (6,1) {};
\node[state] (14) at (5,0) {};
\node[state] (15) at (6,-1) {};
\node[state] (16) at (7,0) {};
\node[state] (17) at (8,1) {};
\node[state] (18) at (9,0) {};
\node[state] (19) at (8.5,-1) {};
\node[state] (20) at (7.5,-1) {};

\node at (0.2,0) {*};

\draw(1) edge (5) 
(5) edge (4)
(4) edge (3)
(3) edge (2)
(2) edge (1)
(2) edge[color=blue, in=135,out=45,looseness=30] (2)
(3) edge[color=blue, in=225,out=135,looseness=30] (3)
(4) edge[color=blue, in=315,out=225,looseness=30] (4)
(5) edge[color=blue, in=315,out=225,looseness=30] (5)
(1) edge[color=blue, bend left] (6)
(6) edge[color=blue, bend left] (7)
(7) edge[color=blue, bend left] (8)
(8) edge[color=blue, bend left] (1)
(8) edge[bend left] (7)
(7) edge[bend left] (6)
(6) edge (9)
(9) edge[bend left] (10)
(10) edge (8)
(10) edge[color=blue, bend left] (9)
(9) edge[color=blue] (11)
(11) edge[color=blue, bend left] (12)
(12) edge[color=blue] (10)
(12) edge[bend left] (11)
(11) edge (13)
(13) edge[bend left] (14)
(14) edge[bend left] (15)
(15) edge (12)
(15) edge[color=blue, bend left] (14)
(14) edge[color=blue, bend left] (13)
(13) edge[color=blue, bend left] (16)
(16) edge[color=blue, bend left] (15)
(16) edge (17)
(17) edge (18)
(18) edge (19)
(19) edge (20)
(20) edge (16)
(17) edge[color=blue, in=45,out=135,looseness=30] (17)
(18) edge[color=blue, in=-45,out=45,looseness=30] (18)
(19) edge[color=blue, in=225,out=315,looseness=30] (19)
(20) edge[color=blue, in=225,out=315,looseness=30] (20)
;
\end{tikzpicture}

	  \caption{\small The cover $S$ defining the subgroup $H'$ of $F(a_1,e_1)$ in the cases $k=3$ (above) and $k=4$ (below). Red edges are labelled by {\color{red} $a_1$}, and blue edges are labelled by {\color{blue} $e_1$}. The basepoint is marked by a star.} \label{fig:ex4}
	 \end{figure}
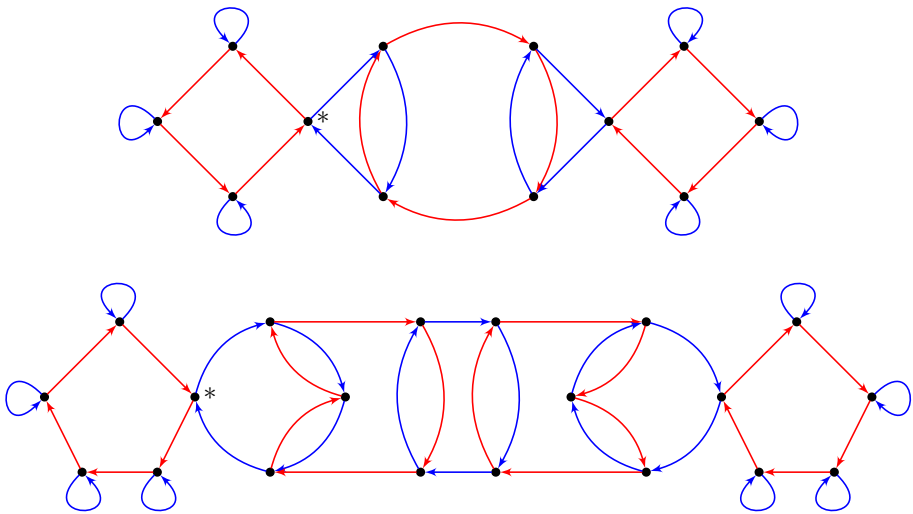

The group $H'$ is a subgroup of $F(a_1,e_1)$ of index $k(k+1)$, since it is defined by a finite cover of degree $k(k+1)$. The group $\GG(T_{k, k+1})$ retracts onto $F(a_1,e_1)$ and so the full preimage $H$ of $H'$ in $\GG(T_{k, k+1})$ is a subgroup of index $k(k+1)$.

\subsection{Isomorphism between $H$ and $\pi_1(X)$} \label{sec:53}

The finite index subgroup $H$ of $\GG(T_{k, k+1})$ acts on the Bass-Serre tree $T$ of the reduced centraliser splitting of $\GG(T_{k, k+1})$ and so $H$ has an induced graph of groups structure determined by the quotient of $T$ by the action of $H$. In this subsection we prove the following.
 
\begin{prop}\label{H}
In the above notation, the subgroup $H$ is isomorphic to the fundamental group of the graph of groups $X$. Namely, the induced splitting of $H$ given by its action on $T$ is $X$.
\end{prop}
\begin{proof}
By Bass-Serre theory the vertices of $T$ correspond to the left cosets of the vertex groups of the reduced centralizer splitting, i.e. left cosets of centralizers of $b$, $c$ and $d$ in $\GG(P_{k, k+1})$, and the action of $\GG(P_{k, k+1})$ on $T$ is by left multiplication.

We proceed by describing the fundamental domain of the action of $H$ on $T$. 

Let $1 \leq i \leq k-1$. Let $\ti{D}_i$ be the graph with $2i+1$ vertices described as follows: it has $1$ vertex $v$ of degree $i$ adjacent to vertices $v_1,\dots, v_i$ of degree $2$; each $v_j$ is adjacent to a vertex $v_j'$ of degree $1$, $j=1, \dots, i$. 

We now consider the subtree of the Bass-Serre tree $T$ which is isomorphic to $\ti{D}_i$, with vertex $v$ labelled by $e_1(a_1e_1)^{i-2}C(b)$, $v_j$ labelled by $e_1(a_1e_1)^{i-2}a_1^jC(c)$, $v_j'$ labelled by $e_1(a_1e_1)^{i-2}a_1^jC(d)=(e_1a_1)^{i-1}a_1^{j-1}C(d)$, $j=1,\dots, i$. We now remove the vertices $v_j'$, $j=2, \dots, i$. The obtained subtree $D_i'$ (without the removed vertices; no edges are removed) corresponds to a lift of $D_i$, where $D_i$ is the $2i$-th member of Sequence (\ref{eq:diseq}), for $1 \leq i \leq k-1$. Note that, strictly speaking, $D_i'$ is not a subtree, since for some of its edges one of the ends is not in $D_i'$, but, abusing the terminology, we will call it a subtree; the same observation applies for similar constructions below.

Similarly, consider a subtree of the Bass-Serre tree which is isomorphic to $\ti{D}_{k-i}$, with vertex $v$ labelled by $(e_1a_1)^{i-1}C(d)$, $v_j$ labelled by $(e_1a_1)^{i-1}e_1^jC(c)$, $v_j'$ labelled by $(e_1a_1)^{i-1}e_1^jC(b)$, $j=1,\dots, k-i$. We now remove the vertices $v_j'$, $j=2, \dots, k-i$. The obtained subtree $D_{k-i}''$ (without the removed vertices; no edges are removed) corresponds to a lift of $D_{k-i}$, where $D_{k-i}$ is the $2i+1$-th member of Sequence (\ref{eq:diseq}).

Let also $D_k'$ be the subtree of the Bass-Serre tree $T$  which is isomorphic to $\ti{D}_k$, with vertex $v$ labelled by $C(b)$, $v_j$ labelled by $a_1^jC(c)$, $v_j'$ labelled by $a_1^jC(d)$, $j=1,\dots, k$. Note that no vertices are removed in this case.
The subtree $D_k'$ corresponds to a lift of $D_k$, which is the first member of Sequence (\ref{eq:diseq}).

Similarly, let $D_k''$ be the subtree of the Bass-Serre tree $T$  which is isomorphic to $\ti{D}_k$, with vertex $v$ labelled by $e_1(a_1e_1)^{k-2}C(b)$, $v_j$ labelled by $e_1(a_1e_1)^{k-2}a_1^jC(c)$, $v_j'$ labelled by $e_1(a_1e_1)^{k-2}a_1^jC(d)=(e_1a_1)^{k-1}a_1^{j-1}C(d)$, $j=1,\dots, k$. Note that no vertices are removed in this case.
The subtree $D_k''$ corresponds to a lift of $D_k$, which is the last member Sequence (\ref{eq:diseq}).

One can readily check that the union $Y$ of all the subtrees $D_i'$ and $D_{k-i}''$ of $T$, for $1 \leq i \leq k-1$, together with $D_k'$ and $D_k''$, is connected. See Figure \ref{fig:Y} for $Y$ in the cases $k=3$ and $k=4$.

\tikzset{node distance=2.5cm, 
	every state/.style={minimum size=1pt, scale=0.3, 
		semithick,
		fill=black},
	initial text={}, 
	double distance=2pt, 
	every edge/.style={ 
		draw,  
		auto}}

\begin{figure}[!h]		
	\begin{tikzpicture}
	
	\footnotesize
	\node[state, color=red] (1) at (0,0) {};  \filldraw (0,0) node[align=right,above] {$1$};	
	\node[state] (2) at (-1,1) {}; 	\filldraw (-1,1) node[align=right,above] {$a_1$};	
	\node[state, color=blue] (3) at (-2,1) {}; \filldraw (-2,1) node[align=right,above] {$a_1$};	
	\node[state] (4) at (-1,0) {};  	\filldraw (-1,0) node[align=right,above] {$a_1^2$};	
	\node[state, color=blue] (5) at (-2,0) {};  	\filldraw (-2,0) node[align=right,above] {$a_1^2$};	
	\node[state] (6) at (-1,-1) {};  	\filldraw (-1,-1) node[align=right,above] {$a_1^3$};	
	\node[state, color=blue] (7) at (-2,-1) {}; 	\filldraw (-2,-1) node[align=right,above] {$a_1^3$};	
	\node[state](8) at (1,0) {};  \filldraw (1,0) node[align=right,above] {$1$};
	\node[state, color=blue] (9) at (2,0) {};   \filldraw (2,0) node[align=right,above] {$1$};
	\node[state] (10) at (3,1) {};   \filldraw (3,1) node[align=right,above] {$e_1^2$};
	\node[state, color=red, fill=white] (24) at (4,1) {};   \filldraw (4,1) node[align=right,above] {$e_1^2$};
	\node[state, color=red] (11) at (4,0) {};   \filldraw (4,0) node[align=right,above] {$e_1$};
	\node[state] (12) at (3,0) {};   \filldraw (3,0) node[align=right,above] {$e_1$};
	\node[state] (13) at (5,1) {}; \filldraw (5,1) node[align=right,above] {$e_1a_1^2$};
	\node[state, color=blue, fill=white] (25) at (6,1) {};  \filldraw (6,1) node[align=right,above] {$e_1a_1^2$};
	\node[state, color=blue] (14) at (6,0) {}; \filldraw (6,0) node[align=right,above] {$e_1a_1$};
	\node[state] (15) at (5,0) {};  \filldraw (5,0) node[align=right,above] {$e_1a_1$};
	\node[state] (16) at (7,0) {};   \filldraw (7,0) node[align=right,above] {$e_1a_1e_1$};
	\node[state, color=red] (17) at (8,0) {};  \filldraw (8,0) node[align=left, left=8pt, below=1pt] {$e_1a_1e_1$};
	\node[state] (18) at (9,0) {}; \filldraw (9,0) node[align=right,above] {$(e_1a_1)^2a_1$};
	\node[state, color=blue] (19) at (10,0) {}; \filldraw (10,0) node[align=right,right] {$(e_1a_1)^2a_1$};
	\node[state] (20) at (9,1) {}; \filldraw (9,1) node[align=right,above] {$(e_1a_1)^2$};
	\node[state, color=blue] (21) at (10,1) {}; \filldraw (10,1) node[align=right,right] {$(e_1a_1)^2$};
	\node[state] (22) at (9,-1) {}; \filldraw (9,-1) node[align=right,below] {$(e_1a_1)^2a_1^2$};
	\node[state, color=blue] (23) at (10,-1) {};  \filldraw (10,-1) node[align=right,right] {$(e_1a_1)^2a_1^2$};
	
	\draw
	(1) edge (2)
	(2) edge (3)
	(1) edge (4)
	(4) edge (5)
	(1) edge (6)
	(6) edge (7)
	(1) edge (8)
	(8) edge (9)
	(9) edge (10)
	(10) edge (24)
	(11) edge (12)
	(12) edge (9)
	(11) edge (13)
	(13) edge (25)
	(14) edge (15)
	(15) edge (11)
	(14) edge (16)
	(16) edge (17)
	(17) edge (18)
	(18) edge (19)
	(17) edge (20)
	(20) edge (21)
	(17) edge (22)
	(22) edge (23);
	
	\end{tikzpicture}
	
	\vspace{0.3 cm}
	\resizebox{430pt}{!}{
		\begin{tikzpicture}
		\footnotesize
		\node[state, color=red] (1) at (0,0) {};  \filldraw (0,0) node[align=right,above] {$1$};
		\node[state] (2) at (-1,1.5) {}; 	\filldraw (-1,1.5) node[align=right,above] {$a_1$};	
		\node[state] (3) at (-1,0.5) {}; 	\filldraw (-1,0.5) node[align=right,above] {$a_1^2$};	
		\node[state] (4) at (-1,-0.5) {}; 	\filldraw (-1,-0.5) node[align=right,above] {$a_1^3$};	
		\node[state] (5) at (-1,-1.5) {}; 	\filldraw (-1,-1.5) node[align=left, left=2pt,above] {$a_1^4$};	
		\node[state, color=blue] (6) at (-2,1.5) {}; \filldraw (-2,1.5) node[align=right,above] {$a_1$};	
		\node[state, color=blue] (7) at (-2,0.5) {}; 	\filldraw (-2,0.5) node[align=right,above] {$a_1^2$};	
		\node[state, color=blue] (8) at (-2,-0.5) {}; 	\filldraw (-2,-0.5) node[align=right,above] {$a_1^3$};	
		\node[state, color=blue] (9) at (-2,-1.5) {}; 	\filldraw (-2,-1.5) node[align=right,above] {$a_1^4$};
		\node[state] (10) at (1,0) {}; 	\filldraw (1,0) node[align=right,above] {$1$};	
		\node[state, color=blue] (11) at (2,0) {}; 	\filldraw (2,0) node[align=right,above] {$1$};	
		\node[state] (12) at (3,0) {}; 	\filldraw (3,0) node[align=right,above] {$e_1$};	
		\node[state,color=red] (13) at (4,0) {}; 	\filldraw (4,0) node[align=right,above] {$e_1$};	
		\node[state] (14) at (3,1) {}; 	\filldraw (3,1) node[align=right,above] {$e_1^2$};	
		\node[state,color=red, fill=white] (15) at (4,1) {}; 	\filldraw (4,1) node[align=right,above] {$e_1^2$};	
		\node[state] (16) at (3,2) {}; 	\filldraw (3,2) node[align=right,above] {$e_1^3$};	
		\node[state,color=red, fill=white] (17) at (4,2) {}; 	\filldraw (4,2) node[align=right,above] {$e_1^3$};	
		\node[state] (18) at (5,0) {}; 	\filldraw (5,0) node[align=right,above] {$e_1a_1$};	
		\node[state,color=blue] (19) at (6,0) {}; 	\filldraw (6,0) node[align=right,below] {$e_1a_1$};	
		\node[state] (20) at (5,1) {}; 	\filldraw (5,1) node[align=right,above] {$e_1a_1^2$};	
		\node[state,color=blue, fill=white] (21) at (6,1) {}; 	\filldraw (6,1) node[align=right,above] {$e_1a_1^2$};	
		\node[state] (22) at (7,0) {}; 	\filldraw (7,0) node[align=right,above] {$e_1a_1e_1$};	
		\node[state,color=red] (23) at (8,0) {}; 	\filldraw (8,0) node[align=right,below] {$e_1a_1e_1$};	
		\node[state] (24) at (7,1) {}; 	\filldraw (7,1) node[align=right,above] {$e_1a_1e_1^2$};	
		\node[state,color=red, fill=white] (25) at (8,1) {}; 	\filldraw (8,1) node[align=right,above] {$e_1a_1e_1^2$};	
		\node[state] (26) at (9,0) {}; 	\filldraw (9,0) node[align=right,above] {$(e_1a_1)^2$};	
		\node[state,color=blue] (27) at (10,0) {}; 	\filldraw (10,0) node[align=right,below] {$(e_1a_1)^2$};	
		\node[state] (28) at (9,1) {}; 	\filldraw (9,1) node[align=right,above] {$\qquad (e_1a_1)^2a_1$};	
		\node[state,color=blue, fill=white] (29) at (10,1) {}; 	\filldraw (10,1) node[align=right,right] {$(e_1a_1)^2a_1$};	
		\node[state] (30) at (9,2) {}; 	\filldraw (9,2) node[align=right,above] {$(e_1a_1)^2a_1^2$};	
		\node[state,color=blue, fill=white] (31) at (10,2) {}; 	\filldraw (10,2) node[align=right,right] {$(e_1a_1)^2a_1^2$};
		\node[state] (32) at (11,0) {}; 	\filldraw (11,0) node[align=right,above] {$(e_1a_1)^2e_1$};	
		\node[state,color=red] (33) at (12,0) {}; 	\filldraw (12,0) node[align=left, left=12pt, below] {$(e_1a_1)^2e_1$};	
		\node[state] (34) at (13,1.5) {}; 	\filldraw (13,1.5) node[align=right,above] {$(e_1a_1)^3$};	
		\node[state,color=blue] (35) at (14,1.5) {}; 	\filldraw (14,1.5) node[align=right,below] {$(e_1a_1)^3$};	
		\node[state] (36) at (13,0.5) {}; 	\filldraw (13,0.5) node[align=right,above] {$\qquad (e_1a_1)^3a_1$};	
		\node[state,color=blue] (37) at (14,0.5) {}; 	\filldraw (14,0.5) node[align=right,below] {$(e_1a_1)^3a_1$};
		\node[state] (38) at (13,-0.5) {}; 	\filldraw (13,-0.5) node[align=right,above] {$\qquad (e_1a_1)^3a_1^2$};	
		\node[state,color=blue] (39) at (14,-0.5) {}; 	\filldraw (14,-0.5) node[align=right,below] {$(e_1a_1)^3a_1^2$};
		\node[state] (40) at (13,-1.5) {}; 	\filldraw (13,-1.5) node[align=left, left] {$(e_1a_1)^3a_1^3$};	
		\node[state,color=blue] (41) at (14,-1.5) {}; 	\filldraw (14,-1.5) node[align=right,above] {$(e_1a_1)^3a_1^3$};	
		
		\draw
		(1) edge (2) edge (3) edge (4) edge (5)
		(2) edge (6)
		(3) edge (7) (4) edge (8) (5) edge (9)
		(1) edge (10) (10) edge (11) 
		(11) edge (12) edge (14) edge (16) (12) edge (13) (14) edge (15) (16) edge (17)
		(13) edge (18) edge (20) (20) edge (21) (18) edge (19)
		(19) edge (24) edge (22) (24) edge (25) (22) edge (23)
		(23) edge (26) edge (28) edge (30) (30) edge (31) (28) edge (29) (26) edge (27)
		(27) edge (32) (32) edge (33)
		(33) edge (34) edge (36) edge (38) edge (40)
		(34) edge (35) (36) edge (37) (38) edge (39) (40) edge (41)
		
		;
		\end{tikzpicture}
	}
	\caption{\small The fundamental domain $Y$ for the action of $H$ on $T$ in the case $k=3$ (above) and $k=4$ (below). Only the vertices that are denoted by disks belong to $Y$, while those denoted by circles don not. Black vertices correspond to the cosets of $C(c)$, red vertices -- to the cosets of $C(b)$, and blue vertices -- to the cosets of $C(d)$. Representatives of the corresponding cosets are written next to the vertices.} \label{fig:Y}
\end{figure}
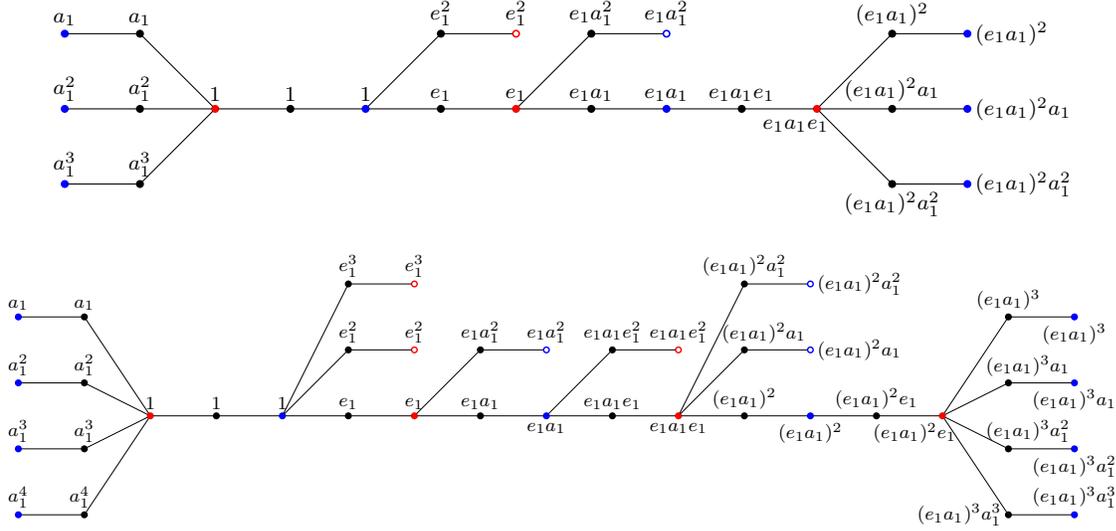

We need a few additional lemmas.

\begin{lemma}\label{rem1}
There is a one-to-one correspondence $\theta_c$ between the vertices of $Y$ which are the left cosets of $C(c)$ and the vertices in the cover $S$ defining $H'$, i.e. all the right cosets of $H'$. 

Under $\theta_c$ a coset $gC(c)$ representing a vertex of $Y$ is mapped to the vertex of $S$, where one gets after reading the word $\pi(g)$ starting at the basepoint, i.e. to the vertex representing the right coset $H'\pi(g)$. 
\end{lemma}
\begin{proof}
	Note that $\theta_c$  is a well-defined map, since $C(c)=\langle b,c,d \rangle$ is in the kernel of $\pi$.
	
	Now, by the definition of $Y$, the set of left cosets of $C(c)$ which are the vertices of $Y$ has the following set of representatives (one for each coset), for $ 1\leq i \leq k-1$:
	\begin{equation}\label{porto}
	e_1(a_1e_1)^{i-2}a_1^j,   \: 1 \leq j \leq i; 
	\quad (e_1a_1)^{i-1}e_1^j,   \: 1 \leq j \leq k-i;
	\quad a_1^j, \: e_1(a_1e_1)^{k-2}a_1^j, \: 1 \leq j \leq k.
	\end{equation}
Then it follows from the definition of $H'$ that the elements from (\ref{porto}) form the set of right coset representatives for $H'$, or, in other words, for each vertex $v$ in $S$ there is exactly one of the elements in (\ref{porto}) which labels a path from the basepoint to $v$ in $S$.
	Indeed,  the elements $a_1^j, 1 \leq j \leq i,$ correspond to the vertices in $P_1$ but not $Q_1$, the elements $(e_1a_1)^{i-1}e_1^j,   \: 1 \leq j \leq k-i,$ correspond to the vertices in $Q_i \cap P_{i+1}$ for all $ 1\leq i \leq k-1$, the elements $e_1(a_1e_1)^{i-2}a_1^j,   \: 1 \leq j \leq i,$ correspond to the vertices in $P_i \cap Q_i$ for all $ 1\leq i \leq k-1$, and the elements  $e_1(a_1e_1)^{k-2}a_1^j, \: 1 \leq j \leq k,$ correspond to the vertices in $P_k$ but not $Q_{k-1}$.
	
	This proves that $\theta_c$ is indeed a bijection.
\end{proof}

\begin{lemma}\label{rem2}
	There is a one-to-one correspondence $\theta_b$ between the vertices of $Y$ which are the left cosets of $C(b)$ and the cycles labelled by $a_1$ in the cover $S$ defining $H'$. 
Under this correspondence, a coset $gC(b)$ representing a vertex of $Y$ is mapped to the $a_1$-cycle passing through the vertex of $S$ where one gets after reading the word $\pi(g)$ starting at the basepoint {\rm(}i.e. through the vertex representing the coset $H'\pi(g)${\rm)}. 

Similarly, there is a one-to-one correspondence $\theta_d$ between the vertices of $Y$ which are the left cosets of $C(d)$ and all the cycles labelled by $e_1$ in the cover $S$ defining $H'$. Under this correspondence, a coset $gC(d)$ representing a vertex of $Y$ is mapped to the $e_1$-cycle passing through the vertex of $S$ where one gets after reading the word $\pi(g)$ starting at the basepoint {\rm(}i.e., through the vertex representing the coset $H'\pi(g)${\rm)}. 
\end{lemma}
\begin{proof}
Note that $\theta_b$ is a well-defined map, since $C(b)$ does not contain $e_1$ and so taking another representative $g'$ from the coset $gC(b)$ would result in a vertex on the same $a_1$-cycle.
 
Now, by the definition of $Y$, the set of left cosets of $C(b)$ which are the vertices of $Y$ has the following set of representatives (one for each coset):
$$
 	e_1(a_1e_1)^{i-2}, \: 1\leq i \leq k.
$$
 It follows from the definition of $H'$ that each word $e_1(a_1e_1)^{i-2}$ labels a path in $S$ from the basepoint to a vertex on the cycle $P_i$, for all $1\leq i \leq k$. Since $P_1, \ldots, P_k$ are all the cycles in $S$ labelled by $a_1$, this proves the first claim.
 
 Similarly, $\theta_d$ is a well-defined map, and  the set of left cosets of $C(d)$ which are vertices of $Y$ has the following set of representatives (one for each coset):
$$
 	(e_1a_1)^{i-1}, \:  1\leq i \leq k-1, \:  a_1^j, \: (e_1a_1)^{k-1}a_1^{j-1}, \: 1 \leq j \leq k.
$$
 It follows from the definition of $H'$ that each word $(e_1a_1)^{i-1}$ labels a path in $S$ from the basepoint to a vertex on the cycle $Q_i$, for all $1\leq i \leq k-1$, and together with  $a_1^j, \: (e_1a_1)^{k-1}a_1^{j-1}$, which label paths from the basepoint to $e_1$-loops in $S$, this gives all the $e_1$-cycles, and hence the desired result.
\end{proof}

The correspondences from Lemmas \ref{rem1} and \ref{rem2} can be easily traced on Figures \ref{fig:ex4} and \ref{fig:Y}  for $k=3$ and $k=4$ (notice that they respect the colours).

\begin{lemma}\label{f1}
$Y$ is a fundamental domain of the action of $H$ on $T$, i.e.,
\begin{enumerate}
	\item no $2$ vertices (edges) of $Y$ belong to the same $H$-orbit;
	\item any vertex (edge) of the tree $T$ can be brought to one of the vertices (edges) of $Y$ by the action of $H$.
\end{enumerate}
\end{lemma}
\begin{proof}
Claim (1) follows from Lemmas \ref{rem1} and \ref{rem2}.
Indeed, suppose that two vertices $v=gC(c)$ and $v'=g'C(c)$ of $Y$ are equivalent under the action of $H$, i.e., $g'C(c)=hgC(c)$ for some $h \in H$. Then $\pi(h)\in H'$ and so $\theta_c(g'C(c))=\theta_c(gC(c)),$ since both correspond to the coset $H'\pi(g)=H'\pi(h)\pi(g)$. Then $v=v'$, since $\theta_c$ is injective.

Similarly, if  $v=gC(b)$ and $v'=g'C(b)$, and $g'C(b)=hgC(b)$ for some $h \in H$, then $\theta_b(g'C(b))=\theta_b(gC(b)),$ since both correspond to the $a_1$-cycle passing through the vertex $H'\pi(g)=H'\pi(h)\pi(g)$, so $v=v'$, since $\theta_b$ is injective. In the same way, if $v=gC(d)$ and $v'=g'C(d)$, and $g'C(d)=hgC(d)$, then $v=v'$. Since no two vertices of $Y$ are in the same $H$-orbit, the same is true for edges of $Y$. This proves the first claim.

We now prove claim (2). Let $v$ be a vertex of $Y$, labelled by $gC(\alpha)$, where $\alpha$ is $b,c$ or $d$, and let $St(v)$ be the star at $v$ in the Bass-Serre tree $T$. Note that the stabilizer of $v$ under the action of $H$ is $C_{H}(\alpha^g)$. We show that, modulo $C_{H}(\alpha^g)$, any vertex $w\in St(v)\smallsetminus \{v\}$ belongs to the orbit of a vertex labelled by:
\begin{itemize}
	\item $gC(b), gC(d)$ if $\alpha=c$;
	\item $ge_1^j C(c)$, if $\alpha=d$, where $j=0,\dots, k-1$;
	\item $ga_1^j C(c)$, if $\alpha=b$, where $j=0,\dots, k$.
\end{itemize}

Indeed, suppose first $v=gC(c)\in Y$ and let $e=(gC(c),gxC(b))\notin Y$. Then $x\in C(c)=\langle b,c,d \rangle$ and without loss of generality we can assume that $x\in \langle b,d\rangle$. It suffices to find $h\in H\cap C(c^g)$ such that $h\cdot gxC(b)=gC(b)$. Hence, $h=gx^{-1}g^{-1}$ satisfies the requirements. The case when $e=(gC(c),gxC(d))\notin Y$ is similar.

Let now $v=gC(d)\in Y$, and let $e=(gC(d),gxC(c))\notin Y$. Then $x\in C(d)=\langle c, d, e_i, i=1,\dots, k+1\rangle$, and without loss of generality we can assume that $x\in \langle c,e_i, i=1,\dots, k+1\rangle$. We now find $h\in H\cap C(d^g)$ such that $h\cdot gxC(c)=ge_1^jC(c)$, for some $j=0,\dots, k-1$. Set $h=ge_1^j x^{-1}g^{-1} \in H$, where $j$ is the sum of exponents of $e_1$ in $x$ modulo $k$. Note that the choice of $j$ guarantees that $e_1^jx^{-1}\in C(d)$ is a loop in the graph.

The argument for the case $\alpha=b$ is identical.

We now show that all the edges incident to a vertex in $Y$ can be taken to $Y$ by elements of $H$.

For the edges that are incident to a vertex of type $gC(c)$ in $Y$ it follows directly from the above claim: any edge connecting $gC(c)$ to $w=g'C(b)$ can be brought to an edge connecting $gC(c)$ to $gC(b)$, which is in $Y$, and similar for $d$.

We now consider the case of edges incident to a vertex $v=gC(b)$ of $Y$. According to the above claim, it suffices to prove that any edge connecting $v$ to $ga_1^j C(c)$ for $j=0, \ldots, k,$ can be brought to an edge of $Y$. If $v=C(b)$ or $v=e_1(a_1e_1)^{k-2}C(b)$, then all these edges are already in $Y$, and there is nothing to prove. 
Otherwise, we have $v=e_1(a_1e_1)^{i-2}C(b)$ for some $1 \leq i \leq k-1$, and then the edges $(gC(b),ga_1^j C(c))$ are already in $Y$ for $j=1, \dots, i$, but not for $j=i+1, \ldots, k$. Thus it suffices to show that the edges $(gC(b),ga_1^j C(c))$ can be brought to $Y$ for $j=i+1, \ldots, k$.
Indeed, let $h_j=g e_1^{k+1-j}a_1^{-j} g^{-1}$. It follows from the construction of $S$ that $h_j \in H$. We then have 
\begin{gather}\notag
\begin{split}
h_j&\cdot g a_1^j C(c)=ge_1^{k+1-j} a_1^{-j} g^{-1}ga_1^jC(c)=ge_1^{k+1-j}C(c) \quad \hbox{ and } \\
h_j&\cdot gC(b)=ge_1^{k+1-j} a_1^{-j}C(b)=ge_1^{k+1-j}C(b).
\end{split}
\end{gather}

It follows that the edge $(gC(b),ga_1^jC(c))$ is mapped by $h_j$ to the edge $(ge_1^{k+1-j}C(b), ge_1^{k+1-j}C(c))$, which belongs to $Y$ since $1 \leq k+1-j \leq k-i$.

Finally, consider the case of edges incident to a vertex $v=gC(d)$ of $Y$. According to the above claim, it suffices to prove that any edge connecting $v$ to $ge_1^j C(c)$ for $j=0, \ldots, k-1$ can be brought to an edge of $Y$. 
Suppose first $v=a_1^lC(d)$ for $1 \leq l \leq k$. Then the edge $(a_1^lC(d),a_1^le_1^jC(c))$ can be brought to the edge $(a_1^lC(d),a_1^lC(c)) \in Y$, by the element $a_1^le_1^{-j}a_1^{-l}$, which belongs to $H$ by the construction of $S$.
  The case when $v=(e_1a_1)^{k-1}a_1^{l-1}C(d)$ is similar. Otherwise, $v=(e_1a_1)^{i-1}C(d)$ for some $1 \leq i \leq k-1$, and the proof in this case is similar to the one above for the cosets of $C(b)$.

This shows that indeed all the edges incident to a vertex in $Y$ can be taken to $Y$ by elements of $H$.

To finish the proof of claim (2) we are left to consider the case when $v\in Y$, $(v,w)\in Y$  but $w$ is not in $Y$. By definition of $Y$, this is possible only in the following setting. Let $(u_j, v_j,w,u_1, v_1)$ be a path of length $4$ in the Bass-Serre tree $T$, where $g=(e_1a_1)^{i-1}$ for some $1 \leq i \leq k-1$ and
$$
u_j = ge_1^jC(b),\  v_j= g e_1^j C(c), \ w= gC(d), \hbox{ where } j=1,\dots, k-i.
$$
We show that $u_1$ and $u_j$ are in the same $H$-orbit, that is there exist $h_j\in H$ so that $h_j\cdot ge_1^jC(b)=ge_1C(b)$, for all $j=1, \ldots, k-i$. Indeed, one can take $h_j=ge_1a_1^{1-j}e_1^{-j}g^{-1}$, and it follows from the construction of $S$ that $h_j \in H$. 

The case when $w=gC(b)$ is identical and is left to the reader. 

Now it is standard that $Y$ is the fundamental domain. Indeed, one can see by induction on the distance between an edge $e$ of $T$ and the closest to $e$ vertex of $Y$ that any edge of $T$ can be taken to $Y$ by an element of $H$, and similar for vertices.

This proves Lemma \ref{f1}.
\end{proof}

 It follows that the quotient of the action of $T$ by $H$ is a graph isomorphic to the one associated to $X$. We need one more lemma about the structure of centralizers in $H$.

\begin{lemma}\label{centH}
The following formulas hold for centralizers in $H$:
\begin{gather}\notag
\begin{split}
C_{H}(c)=\langle c \rangle \times F(b,d) , \quad C_H(d)&=\langle d \rangle \times F({e_i}^{e_1^{j}}, e_1^k, {c}^{e_1^{j}} \mid  j=0,\dots, k-1, \: i=2,\dots,k+1 ),\\ 
C_H(b)&=\langle b \rangle \times F( {a_i}^{a_1^{j}},a_1^{k+1},{c}^{a_1^{j}} \mid  j=0,\dots, k, \: i=2,\dots,k).
\end{split}
\end{gather}

More generally, the following holds.
\begin{itemize}
\item
$ C_{H}(c^g)=\langle c^g \rangle \times  F(b^g, d^g) \cong \BZ \times F_2$ for every $g \in G$;
\item $C_H(b^g)=\langle b^g \rangle \times F({a_i}^{ga_1^{j}},(a_1^{k+1})^g,{c}^{ga_1^{j}} \mid  j=0,\dots, k, \: i=2,\dots,k) \cong \BZ \times F_{k^2+k+1}$ for every $g \in G$;
\item If $g \in G$ is such that $\pi(g)$ labels a path in $S$ from the basepoint to a vertex on one of the cycles $Q_1, \ldots, Q_{k-1}$, then
  $$
  C_H(d^g)=\langle d^g \rangle \times F( {e_i}^{ge_1^{j}}, (e_1^k)^g, {c}^{ge_1^{j}} \mid  j=0,\dots, k-1, \: i=2,\dots,k+1) \cong \BZ \times F_{k^2+k+1};
  $$
\item If $g \in G$ is such that $\pi(g)$ labels a path in $S$ from the basepoint to a vertex on the cycle $P_1$ but not $Q_1$, or $P_k$ but not $Q_{k-1}$, then    
$$
C_H(d^g)=\langle d^g \rangle \times F(e_i^g, c^g \mid i=1, \ldots, k+1 ) \cong \BZ \times F_{k+2}.
$$ 
\end{itemize}
\end{lemma}
\begin{proof}
	We just prove the first statement, the others follow from the construction of the cover $S$ in a similar way, since $C(\alpha^g)=(C(\alpha))^g$, where $C(\alpha)$ is a subgroup of $H^{g^{-1}}=\pi^{-1}((H')^{\pi(g)^{-1}})$ for $\alpha=b,c,d$, and the cover defining $(H')^{\pi(g)^{-1}}$ can be obtained from $S$ just by moving the basepoint along a path with the label $\pi(g)$. 
	 
	Note that 
	$$
	C(b)=\langle a_1,\ldots,a_k, b, c \rangle \cong \BZ \times F_{k+1}, \: C(c)=\langle b,c, d \rangle \cong \BZ \times F_2, \: C(d)=\langle c,d,e_1,\ldots,e_{k+1} \rangle \cong \BZ \times F_{k+2}.
	$$
	It is immediate that $C_{H}(c)=\langle b,c,d\rangle$. Now, we have
	$$
	C_H(b)=H \cap C(b)=H \cap \langle a_1,\ldots,a_k, b, c \rangle=\langle b \rangle \times (H \cap \langle a_1, \ldots, a_k, c \rangle),
	$$
	and $H_0=H \cap \langle a_1, \ldots, a_k, c \rangle$ is a subgroup of the free group  $F(a_1, \ldots, a_k, c)$ which is defined by the following cover: take the cycle $P_1$ (labelled by $a_1$, of length $k+1$) and add loops labelled by $a_2, \ldots, a_k, c$ at every vertex. It follows that $H_0$ has a basis $\{{a_i}^{a_1^{j}},a_1^{k+1},{c}^{a_1^{j}}, \: j=0,\dots, k, \: i=2,\dots,k \}$, and so $C_H(b)$ has the desired form. The proof for $C_H(d)$ is similar.
\end{proof}

We now show that $H$ is isomorphic to the fundamental group of the graph of groups $X$. 
To see this we need to check that vertex groups, edge groups and the embeddings are the same. The vertex groups are simply centralisers (in $H$) of conjugates of generators, which we computed in Lemma \ref{centH}. Edge groups are clearly free abelian groups of rank two and the embeddings are mapping generators of the edge groups to the corresponding generators of the vertex groups. Now it follows directly from the definition of $X$, Lemma \ref{f1} and Lemma \ref{centH} that $H$ is isomorphic to the fundamental group of the graph of groups $X$.
This proves Proposition \ref{H}.
\end{proof}

\subsection{Construction of the finite index subgroup $K$ in $\GG(P_{4k+2})$}\label{sec:54}

Let now $P_{4k+2}$ be the path of length $4k+2$. Let 
$$
\{A, D_1, C_1,B_1,C_1', D_2, C_2, B_2, C_2', \dots, D_k, C_k, B_k, C_k', D_{k+1}, E\}
$$ 
be the ordered list of vertices of $P_{4k+2}$. From now on denote by $G=\GG(P_{4k+2})$ the corresponding RAAG. Let $F=F(C_1,\dots, C_k, C_1', \dots, C_k')$ be the free group on the indicated set of generators. We construct a finite index subgroup $K'$ of $F$. The group $K'$ corresponds to the degree $k(k+1)$ cover $Z$ of the bouquet of $2k$ circles defined as follows. 
\begin{itemize}
	\item There are exactly $i$ cycles spanned by edges labelled by $C_i'$, $i=1, \ldots, k$. In particular, there is only one cycle spanned by $C_1'$, its length is $k(k+1)$.
	\item Edges labelled by $C_k'$ span $k$ cycles of length $k+1$ each.
	\item Edges labelled by $C_i'$ span $i$ cycles, $i=1,\dots, k-1$: one cycle of length $k(k+1)-(i-1)$ and $i-1$ loops.
	\item There are $k+1-i$ cycles spanned by $C_i$, $i=1,\dots, k$. In particular, there is one cycle spanned by $C_k$, its length is $k(k+1)$.
	\item Edges labelled by $C_1$ span $k$ cycles of length $k+1$ each.
	\item Edges labelled by $C_i$ span $k+1-i$ cycles, $i=2,\dots, k$: one cycle of length $k(k+1)-(k-i)$ and $k-i$ loops.
	\item There are no loops at the basepoint.
	\item The graph spanned by the edges labelled by $C_i$ and $C_i'$ is connected, for every $i=1, \ldots,k$.
    \item The graph spanned by the edges labelled by $C_{i-1}'$ and $C_i$ is connected, for every $i=2, \ldots,k$.
\end{itemize}

It is easy to see that such a cover always exists. For instance, one could have at most one loop at each vertex, it is possible since there are $(k-1)(k-2)$ loops altogether and $k(k+1)$ vertices, and in this case the last two conditions follow from the previous ones. 
Such a cover is not unique, but we can choose any cover satisfying the above assumptions, and this will give isomorphic subgroups, as we will see below. See Figures \ref{fig:cover} and \ref{fig:cover2} for covers $Z$ in the cases $k=3$ and $k=4$ respectively.

Without loss of generality, we assume that the vertices of $Z$ which can be reached from the basepoint by reading the words $C_1'^j$, $j=0,\ldots,k-1$, all belong to different cycles labelled by $C_1$, and similarly the vertices of $Z$ which can be reached from the basepoint  by reading the words $C_k^j$, $j=0,\ldots,k-1$, all belong to different cycles labelled by $C_k'$. 

\tikzset{->-/.style={decoration={
  markings,
  mark=at position #1 with {\arrow{latex}}},postaction={decorate}}, node distance=2.5cm, 
every state/.style={minimum size=1pt, scale=0.3, 
semithick,
fill=black},
initial text={}, 
double distance=2pt, 
every edge/.style={color=red, 
draw, 
->,>=latex',
auto,
semithick}}

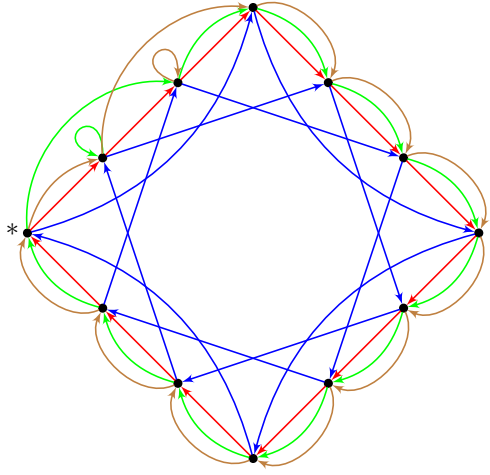
\begin{figure}[h]

\begin{tikzpicture}[scale=1]
	
\node[state] (1) at (-3,0){};
\node[state] (2) at (-2,1) {};
\node[state] (3) at (-1,2) {};
\node[state] (4) at (0,3) {};
\node[state] (5) at (1,2) {};
\node[state] (6) at (2,1) {};
\node[state] (7) at (3,0) {};
\node[state] (8) at (2,-1) {};
\node[state] (9) at (1,-2) {};
\node[state] (10) at (0,-3) {};
\node[state] (11) at (-1,-2) {};
\node[state] (12) at (-2,-1) {};

\node at (-3.2,0) {*};

\draw
(1) edge (2) 
(2) edge (3)
(3) edge (4)
(4) edge (5)
(5) edge (6)
(6) edge (7)
(7) edge (8)
(8) edge (9)
(9) edge (10)
(10) edge (11)
(11) edge (12)
(12) edge (1)

(1) edge[bend right, color=blue] (4)
(4) edge[bend right, color=blue] (7)
(7) edge[bend right, color=blue] (10)
(10) edge[bend right, color=blue] (1)
(2) edge[color=blue] (5)
(5) edge[color=blue] (8)
(8) edge[color=blue] (11)
(11) edge[color=blue] (2)
(3) edge[color=blue] (6)
(6) edge[color=blue] (9)
(9) edge[color=blue] (12)
(12) edge[color=blue] (3)

(1) edge[color=green, bend left=50] (3)
(3) edge[color=green, bend left] (4)
(4) edge[color=green, bend left] (5)
(5) edge[color=green, bend left] (6)
(6) edge[color=green, bend left] (7)
(7) edge[color=green, bend left] (8)
(8) edge[color=green, bend left] (9)
(9) edge[color=green, bend left] (10)
(10) edge[color=green, bend left] (11)
(11) edge[color=green, bend left] (12)
(12) edge[color=green, bend left] (1)
(2) edge[color=green, in=160,out=90,looseness=30] (2)

(1) edge[color=brown, bend left] (2)
(2) edge[color=brown, bend left=50] (4)
(3) edge[color=brown, in=90,out=160,looseness=30] (3)
(4) edge[color=brown,bend left=70] (5)
(5) edge[color=brown,bend left=70] (6)
(6) edge[color=brown,bend left=70] (7)
(7) edge[color=brown,bend left=70] (8)
(8) edge[color=brown,bend left=70] (9)
(9) edge[color=brown,bend left=70] (10)
(10) edge[color=brown,bend left=70] (11)
(11) edge[color=brown,bend left=70] (12)
(12) edge[color=brown,bend left=70] (1)

;
\end{tikzpicture}
 \caption{\small A cover $Z$ defining the subgroup $K'$ in $F$ in the case $k=3$. The basepoint is marked by a star. Each red edge corresponds to two edges, one labelled by {\color{red} $C_1'$} and the other by {\color{red} $C_3$}; each blue edge corresponds to two edges, one labelled by {\color{blue} $C_1$} and the other by {\color{blue} $C_3'$}; green edges are labelled by {\color{green} $C_2$} and brown edges are labelled by {\color{brown} $C_2'$}. Here we have   
 $\alpha_{2,1}=1$, $\beta_{2,1}=1$.} \label{fig:cover}
	 \end{figure} 
	 
\begin{figure}[h]
\begin{tikzpicture}[scale=0.6]
	
\node[state] (1) at (0,0){};
\node[state] (2) at (2,2) {};
\node[state] (3) at (4,4) {};
\node[state] (4) at (6,6) {};
\node[state] (5) at (8,8) {};
\node[state] (6) at (10,6) {};
\node[state] (7) at (12,4) {};
\node[state] (8) at (14,2) {};
\node[state] (9) at (16,0) {};
\node[state] (10) at (15,-2) {};
\node[state] (11) at (14,-4) {};
\node[state] (12) at (13,-6) {};
\node[state] (13) at (12,-8) {};
\node[state] (14) at (10,-8) {};
\node[state] (15) at (8,-8) {};
\node[state] (16) at (6,-8) {};
\node[state] (17) at (4,-8) {};
\node[state] (18) at (3,-6) {};
\node[state] (19) at (2,-4) {};
\node[state] (20) at (1,-2) {};

\node at (-0.4,0) {*};

\draw
(1) edge (2) 
(2) edge (3)
(3) edge (4)
(4) edge (5)
(5) edge (6)
(6) edge (7)
(7) edge (8)
(8) edge (9)
(9) edge (10)
(10) edge (11)
(11) edge (12)
(12) edge (13)
(13) edge (14)
(14) edge (15)
(15) edge (16)
(16) edge (17)
(17) edge (18)
(18) edge (19)
(19) edge (20)
(20) edge (1)

(1) edge[bend right, color=blue] (5)
(5) edge[bend right, color=blue] (9)
(9) edge[bend right, color=blue] (13)
(13) edge[bend right, color=blue] (17)
(17) edge[bend right, color=blue] (1)
(2) edge[color=blue] (6)
(6) edge[color=blue] (10)
(10) edge[color=blue] (14)
(14) edge[color=blue] (18)
(18) edge[color=blue] (2)
(3) edge[color=blue] (7)
(7) edge[color=blue] (11)
(11) edge[color=blue] (15)
(15) edge[color=blue] (19)
(19) edge[color=blue] (3)
(4) edge[color=blue] (8)
(8) edge[color=blue] (12)
(12) edge[color=blue] (16)
(16) edge[color=blue] (20)
(20) edge[color=blue] (4)

(1) edge[color=green, bend left=60] (4)
(4) edge[color=green, bend left] (5)
(2) edge[color=green, in=0,out=-70,looseness=20] (2)
(3) edge[color=green, in=-20,out=-90,looseness=20] (3)
(5) edge[color=green, bend left] (6)
(6) edge[color=green, bend left] (7)
(7) edge[color=green, bend left] (8)
(8) edge[color=green, bend left] (9)
(9) edge[color=green, bend left] (10)
(10) edge[color=green, bend left] (11)
(11) edge[color=green, bend left] (12)
(12) edge[color=green, bend left] (13)
(13) edge[color=green, bend left] (14)
(14) edge[color=green, bend left] (15)
(15) edge[color=green, bend left] (16)
(16) edge[color=green, bend left] (17)
(17) edge[color=green, bend left] (18)
(18) edge[color=green, bend left] (19)
(19) edge[color=green, bend left] (20)
(20) edge[color=green, bend left] (1)

(1) edge[color=brown, bend left] (2)
(2) edge[color=brown, bend left] (3)
(3) edge[color=brown, bend left=50] (5)
(4) edge[color=brown, in=90,out=160,looseness=30] (4)
(5) edge[color=brown, bend left=70] (6)
(6) edge[color=brown, bend left=70] (7)
(7) edge[color=brown, bend left=70] (8)
(8) edge[color=brown, bend left=70] (9)
(9) edge[color=brown, bend left=70] (10)
(10) edge[color=brown, bend left=70] (11)
(11) edge[color=brown, bend left=70] (12)
(12) edge[color=brown, bend left=70] (13)
(13) edge[color=brown, bend left=70] (14)
(14) edge[color=brown, bend left=70] (15)
(15) edge[color=brown, bend left=70] (16)
(16) edge[color=brown, bend left=70] (17)
(17) edge[color=brown, bend left=70] (18)
(18) edge[color=brown, bend left=70] (19)
(19) edge[color=brown, bend left=70] (20)
(20) edge[color=brown, bend left=70] (1)

(1) edge[color=gray, bend left=65] (2)
(2) edge[color=gray, bend left=65] (3)
(3) edge[color=gray, bend left] (4)
(4) edge[color=gray] (6)
(5) edge[color=gray, in=120,out=60,looseness=30] (5)
(6) edge[color=gray, bend left=90, looseness=1.5, shorten >= 5pt] (7)
(7) edge[color=gray, bend left=90, looseness=1.5, shorten >= 5pt] (8)
(8) edge[color=gray, bend left=90, looseness=1.5, shorten >= 3pt] (9)
(9) edge[color=gray, bend left=90, looseness=1.5, shorten >= 5pt] (10)
(10) edge[color=gray, bend left=90, looseness=1.5, shorten >= 5pt] (11)
(11) edge[color=gray, bend left=90, looseness=1.5, shorten >= 5pt] (12)
(12) edge[color=gray, bend left=90, looseness=1.5, shorten >= 3pt] (13)
(13) edge[color=gray, bend left=90, looseness=1.5, shorten >= 5pt] (14)
(14) edge[color=gray, bend left=90, looseness=1.5, shorten >= 5pt] (15)
(15) edge[color=gray, bend left=90, looseness=1.5, shorten >= 5pt] (16)
(16) edge[color=gray, bend left=90, looseness=1.5, shorten >= 3pt] (17)
(17) edge[color=gray, bend left=90, looseness=1.5, shorten >= 5pt] (18)
(18) edge[color=gray, bend left=90, looseness=1.5, shorten >= 5pt] (19)
(19) edge[color=gray, bend left=90, looseness=1.5, shorten >= 5pt] (20)
(20) edge[color=gray, bend left=90, looseness=1.5, shorten >= 3pt] (1)

;
\end{tikzpicture}

	  \caption{\small A cover $Z$ defining the subgroup $K'$ in $F$ in the case $k=4$. The basepoint is marked by a star.  Each red edge corresponds to two edges, one labelled by {\color{red} $C_1'$} and the other by {\color{red} $C_4$}; each blue edge corresponds to two edges, one labelled by {\color{blue} $C_1$} and the other by {\color{blue} $C_4'$}; each green edge corresponds to two edges, one labelled by {\color{green} $C_2$} and the other by {\color{green} $C_3'$}; brown edges are labelled by {\color{brown} $C_2'$} and grey edges are labelled by {\color{gray} $C_3$}. Here we have $\alpha_{2,1}=1 , \alpha_{2,2}=2 , \alpha_{3,1}=3 , \beta_{2,1}=1 , \beta_{3,1}=1 ,\beta_{3,2}=2 $.} \label{fig:cover2}
	 \end{figure}
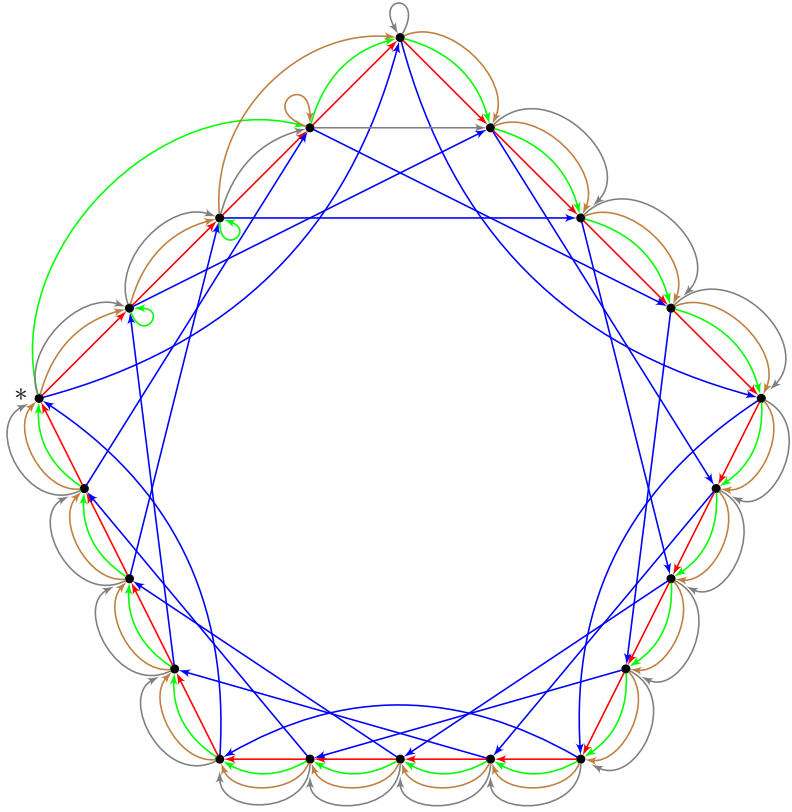 

We let $K$ to be the full preimage of $K'$ in $G$ under the natural epimorphism $\pi:G \rightarrow F$. By definition, $K$ has index $k(k+1)$ in $G$. 

Let $2 \leq i \leq k-1$. By construction, the graph spanned by the edges labelled by $C_{i-1}'$ and $C_i$ is connected, and the graph spanned by the edges labelled only by $C_i$ has $k+1-i$ connected components, all but one being loops, and so these loops are incident to the vertices $v_1, \ldots, v_{k-i}$, which belong to the cycle labelled by $C_{i-1}'$ that goes through the basepoint (of length $k(k+1)-(i-1)$). Let $\alpha_{i,1}<\alpha_{i,2}< \ldots <\alpha_{i,k-i}$ be the lengths of the shortest oriented paths labelled by $C_{i-1}'$ starting at the basepoint and ending at the vertices $v_1,\ldots,v_{k-i}$ (i.e, those which have loops labelled by $C_i$), see  Figures \ref{fig:cover}, \ref{fig:cover2}. In the same way, exchanging the roles of $C_{i-1}'$ and $C_i$, one can define $\alpha'_{i,1}<\alpha'_{i,2}< \ldots <\alpha'_{i,i-2}$ to be the lengths of the shortest oriented paths labelled by $C_i$ starting at the basepoint and ending at the vertices with loops labelled by $C_{i-1}'$, one for each of $i-2$ such loops.

Similarly, for $2 \leq i \leq k-1$, the graph spanned by the edges labelled by $C_i$ and $C_i'$ is connected, and the graph spanned by the edges labelled only by $C_i'$ has $i$ connected components, all but one being loops, and so these loops are incident to the vertices $w_1, \ldots, w_{i-1}$, which belong to the cycle labelled by $C_i$ that goes through the basepoint (of length $k(k+1)-(k-i)$). Let $\beta_{i,1}<\beta_{i,2}< \ldots <\beta_{i,i-1}$ be the lengths of the shortest oriented paths labelled by $C_i$ starting at the basepoint and ending in the vertices $w_1,\ldots,w_{i-1}$ (i.e, those which have loops labelled by $C_i'$), see Figures \ref{fig:cover}, \ref{fig:cover2}. 
In the same way, exchanging the roles of $C_i$ and $C_i'$, one can define $\beta'_{i,1}<\beta'_{i,2}< \ldots <\beta'_{i,k-i}$ to be the lengths of the shortest oriented paths labelled by $C_i'$ starting at the basepoint and ending at the vertices with loops labelled by $C_i$, one for each of $k-i$ such loops.

We additionally define $\alpha_{i,0}=\beta_{i,0}=0$ for $2 \leq i \leq k-1$. 

\subsection{Isomorphism between $K$ and $\pi_1(X)$}\label{sec:55}

From now on we will denote by $T$ the Bass-Serre tree of the reduced centraliser splitting of $\GG(P_{4k+2})$. The finite index subgroup $K$ of $\GG(P_{4k+2})$ acts on $T$, and so $K$ has an induced graph of groups structure determined by the quotient of $T$ by the action of $K$. In this subsection we prove the following.
 
\begin{prop}\label{K}
In the above notation, the subgroup $K$ is isomorphic to the fundamental group of the graph of groups $X$. Namely, the induced splitting of $K$ given by its action on $T$ gives precisely $X$.
\end{prop}

\begin{proof}

Consider the full subgraph (subtree) $Y_0$ of the Bass-Serre tree $T$ spanned by the following vertices
\begin{itemize}
	\item $C(D_i), C(C_i), C(B_i), C(C_i'), \: i=1, \ldots, k, \: C(D_{k+1})$;
	\item $C_1'^jC(C_1), \: C_1'^jC(D_1), \: j=1, \ldots, k-1$;
	\item $C_k^jC(C_k'), \: C_k^jC(D_{k+1}), \: j=1, \ldots, k-1$;
	\item $C_{i-1}'^{\alpha_{i,j}}C(C_i), \: C_{i-1}'^{\alpha_{i,j}} C(B_i), \: i=2, \ldots, k-1, \: j=1,\ldots,k-i$;
	\item $C_i^{\beta_{i,j}}C(C_i'), \: C_i^{\beta_{i,j}}C(D_{i+1}), \: i=2, \ldots, k-1, \: j=1, \ldots, i-1$.
\end{itemize}
Let $Y$ be obtained from $Y_0$ by deleting the following vertices (without deleting any edges): 
\begin{itemize}
	\item $C_{i-1}'^{\alpha_{i,j}} C(B_i), \: i=2, \ldots, k-1, \: j=1,\ldots,k-i$;
	\item $C_i^{\beta_{i,j}}C(D_{i+1}), \: i=2, \ldots, k-1, \: j=1, \ldots, i-1$. 
\end{itemize}

The subtree $Y$ is shown on Figure \ref{fig:fd2} for $k=3$ and $k=4$.

\tikzset{node distance=2.5cm, 
every state/.style={minimum size=1pt, scale=0.3, 
semithick,
fill=black},
initial text={}, 
double distance=2pt, 
every edge/.style={ 
draw,  
auto}}

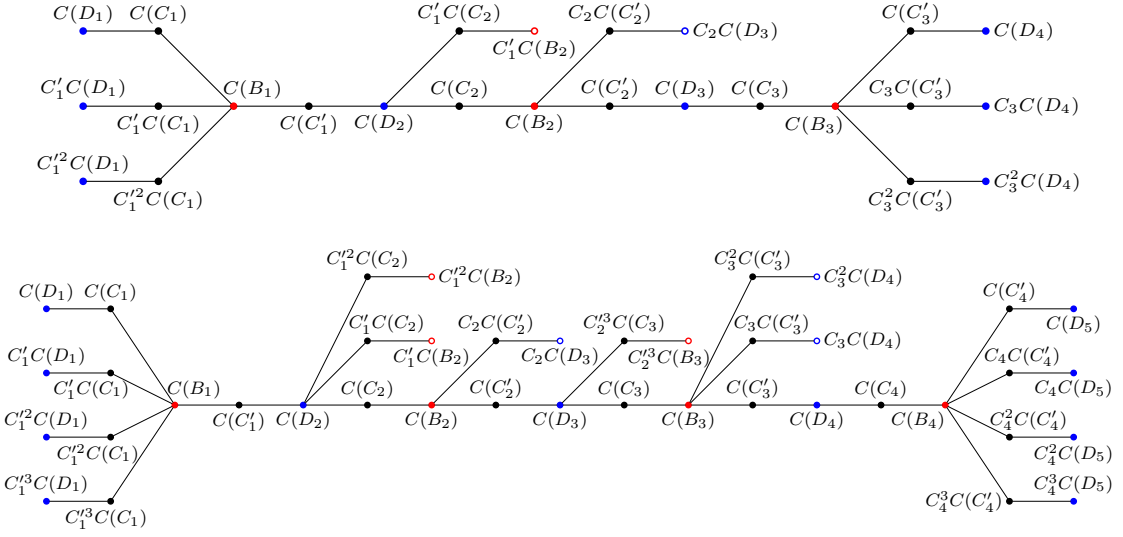
\begin{figure}[!h]	
\begin{tikzpicture}

\scriptsize
\node[state, color=red] (1) at (0,0) {};  \filldraw (0.25,0) node[align=right,above] {$C(B_1)$};	
\node[state] (2) at (-1,1) {}; 	\filldraw (-1,1) node[align=right,above] {$C(C_1)$};	
\node[state, color=blue] (3) at (-2,1) {}; \filldraw (-2,1) node[align=right,above] {$C(D_1)$};	
\node[state] (4) at (-1,0) {};  	\filldraw (-1,0) node[align=right,below] {$C_1'C(C_1)$};	
\node[state, color=blue] (5) at (-2,0) {};  	\filldraw (-2,0) node[align=right,above] {$C_1'C(D_1)$};	
\node[state] (6) at (-1,-1) {};  	\filldraw (-1,-1) node[align=right,below] {$C_1'^2C(C_1)$};	
\node[state, color=blue] (7) at (-2,-1) {}; 	\filldraw (-2,-1) node[align=right,above] {$C_1'^2C(D_1)$};	
\node[state](8) at (1,0) {};  \filldraw (1,0) node[align=right,below] {$C(C_1')$};
\node[state, color=blue] (9) at (2,0) {};   \filldraw (2,0) node[align=right,below] {$C(D_2)$};
\node[state] (10) at (3,1) {};   \filldraw (3,1) node[align=right,above] {$C_1'C(C_2)$};
\node[state, color=red, fill=white] (24) at (4,1) {};   \filldraw (4,1) node[align=right,below] {$C_1'C(B_2)$};
\node[state, color=red] (11) at (4,0) {};   \filldraw (4,0) node[align=right,below] {$C(B_2)$};
\node[state] (12) at (3,0) {};   \filldraw (3,0) node[align=right,above] {$C(C_2)$};
\node[state] (13) at (5,1) {}; \filldraw (5,1) node[align=right,above] {$C_2C(C_2')$};
\node[state, color=blue, fill=white] (25) at (6,1) {};  \filldraw (6,1) node[align=right,right] {$C_2C(D_3)$};
\node[state, color=blue] (14) at (6,0) {}; \filldraw (6,0) node[align=right,above] {$C(D_3)$};
\node[state] (15) at (5,0) {};  \filldraw (5,0) node[align=right,above] {$C(C_2')$};
\node[state] (16) at (7,0) {};   \filldraw (7,0) node[align=right,above] {$C(C_3)$};
\node[state, color=red] (17) at (8,0) {};  \filldraw (8,0) node[align=left, left=8pt, below=1pt] {$C(B_3)$};
\node[state] (18) at (9,0) {}; \filldraw (9,0) node[align=right,above] {$C_3C(C_3')$};
\node[state, color=blue] (19) at (10,0) {}; \filldraw (10,0) node[align=right,right] {$C_3C(D_4)$};
\node[state] (20) at (9,1) {}; \filldraw (9,1) node[align=right,above] {$C(C_3')$};
\node[state, color=blue] (21) at (10,1) {}; \filldraw (10,1) node[align=right,right] {$C(D_4)$};
\node[state] (22) at (9,-1) {}; \filldraw (9,-1) node[align=right,below] {$C_3^2C(C_3')$};
\node[state, color=blue] (23) at (10,-1) {};  \filldraw (10,-1) node[align=right,right] {$C_3^2C(D_4)$};

\draw
(1) edge (2)
(2) edge (3)
(1) edge (4)
(4) edge (5)
(1) edge (6)
(6) edge (7)
(1) edge (8)
(8) edge (9)
(9) edge (10)
(10) edge (24)
(11) edge (12)
(12) edge (9)
(11) edge (13)
(13) edge (25)
(14) edge (15)
(15) edge (11)
(14) edge (16)
(16) edge (17)
(17) edge (18)
(18) edge (19)
(17) edge (20)
(20) edge (21)
(17) edge (22)
(22) edge (23);

\end{tikzpicture}

\vspace{0.3 cm}
	\resizebox{430pt}{!}{
\begin{tikzpicture}
\footnotesize
	\node[state, color=red] (1) at (0,0) {};  \filldraw (0.3,0) node[align=right,above] {$C(B_1)$};
	\node[state] (2) at (-1,1.5) {}; 	\filldraw (-1,1.5) node[align=right,above] {$C(C_1)$};	
	\node[state] (3) at (-1,0.5) {}; 	\filldraw (-1.3,0.5) node[align=right,below] {$C_1'C(C_1)$};	
	\node[state] (4) at (-1,-0.5) {}; 	\filldraw (-1.2,-0.5) node[align=right,below] {$C_1'^2C(C_1)$};	
	\node[state] (5) at (-1,-1.5) {}; 	\filldraw (-1,-1.5) node[align=left, left=2pt,below] {$C_1'^3C(C_1)$};	
	\node[state, color=blue] (6) at (-2,1.5) {}; \filldraw (-2,1.5) node[align=right,above] {$C(D_1)$};	
	\node[state, color=blue] (7) at (-2,0.5) {}; 	\filldraw (-2,0.5) node[align=right,above] {$C_1'C(D_1)$};	
	\node[state, color=blue] (8) at (-2,-0.5) {}; 	\filldraw (-2,-0.5) node[align=right,above] {$C_1'^2C(D_1)$};	
	\node[state, color=blue] (9) at (-2,-1.5) {}; 	\filldraw (-2,-1.5) node[align=right,above] {$C_1'^3C(D_1)$};
	\node[state] (10) at (1,0) {}; 	\filldraw (1,0) node[align=right,below] {$C(C_1')$};	
	\node[state, color=blue] (11) at (2,0) {}; 	\filldraw (2,0) node[align=right,below] {$C(D_2)$};	
	\node[state] (12) at (3,0) {}; 	\filldraw (3,0) node[align=right,above] {$C(C_2)$};	
	\node[state,color=red] (13) at (4,0) {}; 	\filldraw (4,0) node[align=right,below] {$C(B_2)$};	
	\node[state] (14) at (3,1) {}; 	\filldraw (3.3,1) node[align=right,above] {$C_1'C(C_2)$};	
	\node[state,color=red, fill=white] (15) at (4,1) {}; 	\filldraw (4,1) node[align=right,below] {$C_1'C(B_2)$};	
	\node[state] (16) at (3,2) {}; 	\filldraw (3,2) node[align=right,above] {$C_1'^2C(C_2)$};	
	\node[state,color=red, fill=white] (17) at (4,2) {}; 	\filldraw (4,2) node[align=right,right] {$C_1'^2C(B_2)$};	
	\node[state] (18) at (5,0) {}; 	\filldraw (5,0) node[align=right,above] {$C(C_2')$};	
	\node[state,color=blue] (19) at (6,0) {}; 	\filldraw (6,0) node[align=right,below] {$C(D_3)$};	
	\node[state] (20) at (5,1) {}; 	\filldraw (5,1) node[align=right,above] {$C_2C(C_2')$};	
	\node[state,color=blue, fill=white] (21) at (6,1) {}; 	\filldraw (6,1) node[align=right,below] {$C_2C(D_3)$};	
	\node[state] (22) at (7,0) {}; 	\filldraw (7,0) node[align=right,above] {$C(C_3)$};	
	\node[state,color=red] (23) at (8,0) {}; 	\filldraw (8,0) node[align=right,below] {$C(B_3)$};	
	\node[state] (24) at (7,1) {}; 	\filldraw (7,1) node[align=right,above] {$C_2'^3C(C_3)$};	
	\node[state,color=red, fill=white] (25) at (8,1) {}; 	\filldraw (7.7,1) node[align=right,below] {$C_2'^3C(B_3)$};	
	\node[state] (26) at (9,0) {}; 	\filldraw (9,0) node[align=right,above] {$C(C_3')$};	
	\node[state,color=blue] (27) at (10,0) {}; 	\filldraw (10,0) node[align=right,below] {$C(D_4)$};	
	\node[state] (28) at (9,1) {}; 	\filldraw (9,1) node[align=right,above] {$\qquad C_3C(C_3')$};	
	\node[state,color=blue, fill=white] (29) at (10,1) {}; 	\filldraw (10,1) node[align=right,right] {$C_3C(D_4)$};	
	\node[state] (30) at (9,2) {}; 	\filldraw (9,2) node[align=right,above] {$C_3^2C(C_3')$};	
	\node[state,color=blue, fill=white] (31) at (10,2) {}; 	\filldraw (10,2) node[align=right,right] {$C_3^2C(D_4)$};
	\node[state] (32) at (11,0) {}; 	\filldraw (11,0) node[align=right,above] {$C(C_4)$};	
	\node[state,color=red] (33) at (12,0) {}; 	\filldraw (12,0) node[align=left, left=12pt, below] {$C(B_4)$};	
	\node[state] (34) at (13,1.5) {}; 	\filldraw (13,1.5) node[align=right,above] {$C(C_4')$};	
	\node[state,color=blue] (35) at (14,1.5) {}; 	\filldraw (14,1.5) node[align=right,below] {$C(D_5)$};	
	\node[state] (36) at (13,0.5) {}; 	\filldraw (13.2,0.5) node[align=right,above] {$C_4C(C_4')$};	
	\node[state,color=blue] (37) at (14,0.5) {}; 	\filldraw (14,0.5) node[align=right,below] {$C_4C(D_5)$};
	\node[state] (38) at (13,-0.5) {}; 	\filldraw (13,-0.5) node[align=right,above] {$\qquad C_4^2C(C_4')$};	
	\node[state,color=blue] (39) at (14,-0.5) {}; 	\filldraw (14,-0.5) node[align=right,below] {$C_4^2C(D_5)$};
	\node[state] (40) at (13,-1.5) {}; 	\filldraw (13,-1.5) node[align=left, left] {$C_4^3C(C_4')$};	
	\node[state,color=blue] (41) at (14,-1.5) {}; 	\filldraw (14,-1.5) node[align=right,above] {$C_4^3C(D_5)$};

	\draw
	(1) edge (2) edge (3) edge (4) edge (5)
	(2) edge (6)
	(3) edge (7) (4) edge (8) (5) edge (9)
	(1) edge (10) (10) edge (11) 
	(11) edge (12) edge (14) edge (16) (12) edge (13) (14) edge (15) (16) edge (17)
	(13) edge (18) edge (20) (20) edge (21) (18) edge (19)
	(19) edge (24) edge (22) (24) edge (25) (22) edge (23)
	(23) edge (26) edge (28) edge (30) (30) edge (31) (28) edge (29) (26) edge (27)
	(27) edge (32) (32) edge (33)
	(33) edge (34) edge (36) edge (38) edge (40)
	(34) edge (35) (36) edge (37) (38) edge (39) (40) edge (41)
	
	;
\end{tikzpicture}
}

	  \caption{\small The fundamental domain $Y$ for the action of $K$ on $T$ in the case $k=3$ (above) and $k=4$ (below), where $K$ is defined via $K'$ as on Figures \ref{fig:cover}, \ref{fig:cover2}. Only the vertices that are filled belong to $Y$, while those with blank interior do not. Black vertices correspond to the cosets of $C(C_i), C(C_i')$, red vertices -- to the cosets of $C(B_i)$, and blue vertices -- to the cosets of $C(D_i)$.} \label{fig:fd2}
	 \end{figure}

\begin{lemma}\label{fY} 
$Y$ is a fundamental domain of the action of $K$ on $T$. 
\end{lemma}

\begin{proof}

We first show that no $2$ vertices in $Y$ belong to the same $K$-orbit. This is immediate for the cosets of the centralizers of $C_1', C_k, B_1, B_2, \ldots, B_k$ and $D_2,\ldots, D_k$, since for them there is only one vertex in $Y$ even in each $G$-orbit.

Suppose that there exists $h\in K$ such that for some $2 \leq i \leq k-1$ and $0 \leq j<l \leq k-i$ we have
$h\cdot C_{i-1}'^{\alpha_{ij}}C(C_i)=C_{i-1}'^{\alpha_{il}}C(C_i)$. By definition, $h=C_{i-1}'^{\alpha_{il}}xC_{i-1}'^{-\alpha_{ij}}$, where $x\in C(C_i)=\langle B_i,C_i,D_i \rangle$. Then $\pi(h)=C_{i-1}'^{\alpha_{il}}wC_{i-1}'^{-\alpha_{ij}} \in K'$, where $w=\pi(x)$ is a power of $C_i$. This means that the vertices in $Z$ obtained from the basepoint after reading $C_{i-1}'^{\alpha_{il}}$ and $C_{i-1}'^{\alpha_{ij}}$ are on the same cycle labelled by $C_i$, which contradicts the definition of $\alpha$'s. Similar argument shows that there is no $h\in K$ such that for some $2 \leq i \leq k-1$ and $0 \leq j<l \leq k-i$ we have $hC_i^{\beta_{i,j}}C(C_i')=C_i^{\beta_{i,l}}C(C_i')$.

In the same way, if there is $h \in K$ such that $hC_1'^jC(C_1)=C_1'^lC(C_1)$ for some $0 \leq j < l \leq k-1$, then 
$C_1'^jxC_1'^{-l}\in K$ for some $x \in C(C_1)=\langle D_1,C_1,B_1 \rangle$, and so $C_1'^jwC_1'^{-l}\in K'$, where $w=\pi(x)$ is a power of $C_1$, so the vertices in $Z$ obtained from the basepoint after reading $C_1'^j$ and $C_1'^l$ are on the same cycle labelled by $C_1$, which contradicts our choice of $Z$. Also, if there is $h \in K$ such that $hC_1'^jC(D_1)=C_1'^lC(D_1)$ for some $0 \leq j < l \leq k-1$, then $C_1'^jxC_1'^{-l}\in K$ for some $x \in C(D_1)=\langle A,D_1,C_1 \rangle$, so $C_1'^jwC_1'^{-l}\in K'$, where $w=\pi(x)$ is a power of $C_1$, which is again a contradiction as above. Similar arguments apply to the cosets of $C_k'$ and $D_{k+1}$. Thus no $2$ vertices of $Y$ are in the same $K$-orbit.

We now show that all the vertices $C_{i-1}'^{\alpha_{ij}}C(B_i)$, $i=2, \ldots, k-1, \: j=1,\ldots,k-i$, can be mapped to $C(B_i)$ and  all the vertices $C_{i}^{\beta_{ij}}C(D_{i+1})$, $\: i=2, \ldots, k-1, \: j=1,\ldots,i-1$, can be mapped to $C(D_{i+1})$ by some elements of $K$.

Indeed, we show that there exists $h\in K$ such that $h\cdot C(B_i)=C_{i-1}'^{\alpha_{ij}}C(B_i)$. It suffices to take $h=C_{i-1}'^{\alpha_{ij}}w_j \in K$, where $w_j\in C(B_i)= \langle C_i,B_i,C_i' \rangle$. Let $v$ be the vertex of $Z$ where one gets after reading the label $C_{i-1}'^{\alpha_{ij}}$ from the basevertex. It now suffices to choose $w_j=w_j(C_i, C_i')$ to be the label of a path joining $v$ to the basepoint inside the graph spanned by the edges labelled by $C_i$ and $C_i'$, and such a path indeed exists since this graph is connected.

The proof for cosets of $C(D_i)$ is analogous and is left to the reader.

We are left to show that all the edges $e$ in the Bass-Serre tree $T$ which share a vertex with $Y$ can be mapped into $Y$ by an element of $K$.

Suppose first that $e=(xC(C_i), xgC(B_i))$, where $x=1$ and $1 \leq i \leq k$, or $x=C_{i-1}'^{\alpha_{i,j}}$, $i=2, \ldots, k-1, \: j=1,\ldots,k-i$, or $x=C_1'^j, \: j=1,\ldots,k-1$ with $i=1$. In particular,  $xC(C_i)$ is in $Y$. Then $g\in C(C_i)=\langle B_i, C_i, D_i \rangle$, so, without loss of generality, we can assume that $g=w(D_i,B_i)$, and so $g'=xg^{-1}x^{-1} \in K$ and $g'$ takes $e$ to $(xC(C_i), xC(B_i))$ which is in $Y$ by construction.

Suppose next that $e=(C(B_i), gC(C_i'))$, $i=1,\ldots,k$. Then $g\in C(B_i)$, and so, without loss of generality, we can assume that $g=w(C_i,C_i')$.
 We show that for all such $g$ there exists $h\in C_K(B_i)$ and $j=0, \ldots, i-1$ such that $hgC(C_i')=C_i^{\beta_{ij}}C(C_i')$. Indeed, let $v$ be the vertex of $Z$ where one gets after reading the word $w$ from the basepoint of $Z$. By definition, one can choose $j$ such that the path labelled by $C_i^{\beta_{ij}}$ and starting at the basepoint of $Z$ ends in the $C_i'$-cycle that passes  through $v$, and then $h=C_i^{\beta_{ij}}C_i'^lg^{-1}\in K$ for some $l$, and $h \in C(B_i)$, as desired. Then $h$ takes the edge $e$ into $Y$.
 
 Let now $e=( C_1'^jC(D_1), C_1'^jgC(C_1)), \: j=0, \ldots, k-1$. Then $g \in C(D_1)$, and without loss of generality we can assume that $g=w(A,C_1)$. Hence, $\pi(g)=C_1^l$ for some $l$, and so $\pi(C_1^{l}g^{-1})=1$. Therefore, $h=C_1'^jC_1^{l}g^{-1}C_1'^{-j} \in K$. Now $h$ takes $e$ into $(C_1'^jC(D_1), C_1'^jC(C_1))$, which is in $Y$.
 
All the other cases are similar to the above and left to the reader.

Now it is routine to see that $Z$ is a fundamental domain for the action of $K$ on $T$. This proves Lemma \ref{fY}.
\end{proof}

We established that the quotient of the action of $T$ by $K$ is a graph isomorphic to the one associated to $X$. We need one more lemma about centralizers in $K$.  

\begin{lemma}\label{centK}
	The following formulas hold for centralizers in $K$:
	\begin{enumerate}
		\item $C(C_i^g)=(K \cap \langle C_i^g \rangle) \times F(D_i^g, B_i^g) \cong \BZ \times F_2$ for all $g \in G$, $i=1, \ldots, k$;
		\item $C(C_i'^g)=(K \cap \langle C_i'^g \rangle) \times F(B_i^g, D_{i+1}^g) \cong \BZ \times F_2$ for all $g \in G$, $i=1, \ldots, k$;
		\item $C(B_i)= \langle B_i \rangle \times L_i \cong \BZ \times F_{k^2+k+1}$ for $i=2, \ldots, k-1$, where $L_i$ is a subgroup of index $k(k+1)$ in $F(C_i,C_i')$ with basis 
$$
		\left.\left\{
	\begin{array}{l}
		C_i^{k(k+1)-(k-i)}, C_i^{(C_i'^{\beta'_{i,j}})},  C_i'^{k(k+1)-(i-1)},\\ C_i'^{(C_i^{\beta_{i,l}})},	 U_{i,1},\ldots,U_{i,k^2}
    \end{array}
		\right| j=1,\ldots,k-i, \: l=1,\ldots, i-1
		\right\},
$$
		where each of $U_{i,1},\ldots,U_{i,k^2}$ is not conjugate to a power of $C_i$ or $C_i'$;
		\item $C(D_i)= \langle D_i \rangle \times L_i' \cong \BZ \times F_{k^2+k+1} $ for $i=2, \ldots, k$, where $L'_i$ is a subgroup of index $k(k+1)$ in $F(C_i,C_{i-1}')$ with the basis 
		$$
		\left.\left\{
	\begin{array}{l}
		C_i^{k(k+1)-(k-i)},C_i^{(C_{i-1}'^{\alpha_{i,j}})},  C_{i-1}'^{k(k+1)-(i-2)}, \\ C_{i-1}'^{(C_i^{\alpha'_{i,l}})},
		U'_{i,1},\ldots, U'_{i,k^2+1} 
	\end{array}
		\right| j=1,\ldots,k-i, \: l=1,\ldots, i-2
			\right\},
		$$
		where each of $U'_{i,1},\ldots,U'_{i,k^2+1}$ is not conjugate to a power of $C_i$ or $C_{i-1}'$;
		\item $C(B_1)= \langle B_1 \rangle \times L_1 \cong \BZ \times F_{k^2+k+1}$, where $L_1$ is a subgroup of index $k(k+1)$ in $F(C_1,C_1')$ with the basis $\{C_1'^{k(k+1)}, (C_1^{k+1})^{(C_1'^j)}, U_{1,1},\ldots,U_{1,k^2}, \: j=0, \ldots, k-1\}$,  where each of $U_{1,1},\ldots,U_{1,k^2}$ is not conjugate to a power of  $C_1$ or $C_1'$;
		\item $C(B_k)= \langle B_k \rangle \times L_k \cong \BZ \times F_{k^2+k+1}$, where $L_k$ is a subgroup of index $k(k+1)$ in $F(C_k,C_k')$ with the basis $\{C_k^{k(k+1)}, (C_k'^{k+1})^{(C_k^j)}, U_{k,1},\ldots,U_{k,k^2}, \: j=0, \ldots, k-1\}$,  where each of $U_{k,1},\ldots,U_{k,k^2}$ is not conjugate to a power of $C_k$ or $C_k'$;
		\item $C(D_1^{C_1'^j})= \langle D_1^{C_1'^j} \rangle \times F((C_1^{k+1})^{C_1'^j}, A^{C_1'^j}, A^{C_1'^jC_1}, \ldots, A^{C_1'^jC_1^{k}}) \cong \BZ \times F_{k+2}$, where $j=0,\ldots,k-1$;
		\item $C(D_{k+1}^{C_k^j})= \langle D_{k+1}^{C_k^j} \rangle \times F((C_k'^{k+1})^{C_k^j}, E^{C_k^j}, E^{C_k^jC'_k}, \ldots, E^{C_k^jC_k^{'k}}) \cong \BZ \times F_{k+2}$, where  $j=0,\ldots,k-1$.
	\end{enumerate}
\end{lemma}

\begin{proof}
Recall that $C(C_i)=\langle C_i \rangle \times F(B_i,D_i)$, $C(C_i')=\langle C_i' \rangle \times F(B_i, D_{i+1})$, 
$C(B_i)= \langle B_i \rangle \times F(C_i,C_i')$ for $i=1,\ldots,k$, and $C(D_i)=\langle D_i \rangle \times F(C_{i-1}',C_i)$ for $i=1,\ldots,k-1$, $C(D_1)=\langle D_1 \rangle \times F(A,C_1)$, $C(D_{k+1})=\langle D_{k+1} \rangle \times F(C_k',E)$.

    The first two claims of the lemma are immediate. We now prove the third one. 
    Let $Z_i$ be the (connected) graph spanned by the edges labelled by $C_i$ and $C_i'$ in $Z$. By definition, $Z_i$ is a cover of the bouquet of two circles, labelled by $C_i$ and $C_i'$. Let $L_i<F(C_i,C_i')$ be the free group corresponding to the cover $Z_i$. Since the index of $L_i$ in $F(C_i,C_i')$ is $k(k+1)$, its rank is $k(k+1)+1$. Now it is easy to see that $L_i$ has the desired basis by first choosing a basis corresponding to a maximal subtree in $Z_i$ with $k(k+1)-(k-i)-1$ edges labelled by $C_i$ and $k-i$ edges labelled by $C_i'$, and then applying appropriate Nielsen transformations.   
    
Similarly, to prove the fifth claim, we let $Z_1$ to be the graph spanned by $C_1$ and $C_1'$ and $L_1$ be the corresponding subgroup of index $k(k+1)$ in $F(C_1,C_1')$. Again, one can see that $L_1$  has the desired basis by first choosing a basis corresponding to a maximal subtree in $Z_1$ which includes $k$ edges from each cycle labelled by $C_1$ in $Z_1$, as well as a path from the basepoint of length $k-1$ with all the edges labelled by $C_1'$, and then applying appropriate Nielsen transformations.

To prove claim (7) for $j=0$, we define $L_0$ to be the index $k+1$ subgroup of $F(A,C_1)$ given by the cover $Z_0$ obtained from a cycle of length $k+1$ labelled by $C_1$ by adding loops labelled by $A$ at every vertex. The desired basis for $L_0$ then corresponds to a maximal subtree in $Z_0$. The case $j>0$ is similar.
    
The proofs of all the other claims are similar and left to the reader.
\end{proof}

Finally, we show that $K$ is isomorphic to the fundamental group of the graph of groups $X$.
To see this we need to check that vertex groups, edge groups and the embeddings are the same. The vertex groups are simply centralisers (in $K$) of conjugates of generators, which we computed in Lemma \ref{centK}. Edge groups are clearly free abelian groups of rank two and the embeddings are mapping generators of the edge groups to the corresponding generators of the vertex groups. Now Proposition \ref{K} follows directly from the definition of $X$, Lemma \ref{fY} and Lemma \ref{centK} that $K$ is isomorphic to the fundamental group of the graph of groups $X$.
\end{proof}

Propositions \ref{H} and \ref{K} together imply that $H$ and $K$ are isomorphic, and so $\GG(T_{k,k+1})$ and $\GG(P_{4k+2})$ are indeed commensurable. This finishes the proof of Theorem \ref{com}.

\section{Path RAAGs are not commensurable}\label{sec:6}
	
	In this section we address the proof of Theorem \ref{thm1} which states that $\GG(P_m)$ and $\GG(P_n)$ are not commensurable, with the only exception of $\GG(P_3)$ and $\GG(P_4)$.
	
	By Corollary \ref{cor:reductionLSE}, it suffices to show that the linear system of equations and inequalities associated to the product graph does not have integer solutions. The key tool is Lemma \ref{fork}, which allows, given a local pattern in the graph, to deduce that some edges do not exist, see Figure \ref{fig:11}. Applying this lemma recursively, the structure of the graph is significantly simplified and a case-by-case analysis on the parities of $n$ and $m$ allows us to conclude that the system does not have integer solutions.

	We can suppose that $m, n \geq 3$, since $\GG(P_0) \cong \BZ$, $\GG(P_1)\cong \BZ^2$, $\GG(P_2) \cong F_2 \times \BZ$ and $\GG(P_n)$ for some fixed $n \geq 3$ are pairwise not commensurable (not even quasi-isometric, see \cite{BN}). Furthermore, we can  suppose that $m > n \geq 5$, since other cases are already covered by Theorem \ref{P3}. 
			 
	\subsection{Product graph for two paths}
			We fix some $m > n \geq 3$. Let $a_0, a_1, \ldots, a_m$ be the vertices of $P_m$, considered as canonical generators of $\GG(P_m)$, and $b_0, b_1, \ldots, b_n$ be the vertices of $P_n$, considered as canonical generators of $\GG(P_n)$. Then, in the above notation, $\Gamma_1=P_m$, $\Gamma_2=P_n$, and $\ti{\Gamma}_1=P_{m-2}$, with vertices $a_1, \ldots, a_{m-1}$, $\ti{\Gamma}_2=P_{n-2}$, with vertices $b_1, \ldots, b_{n-1}$. 
			 
			Suppose that $\GG(P_m)$ and $\GG(P_n)$ are commensurable. 
			Note that in our case $\DD = \ti{\Gamma}_1 \times \ti{\Gamma}_2=P_{m-2} \times P_{n-2}$ is the following graph: its set of vertices is $\{ (a_i, b_j), \: i=1, \ldots, m-1; \: j=1, \ldots, n-1 \}$, and two vertices $(a_{i_1}, b_{j_1})$ and $(a_{i_2}, b_{j_2})$ are connected by an edge in $\DD$ if and only if $|i_1-i_2| = 1$ and $|j_1-j_2|=1$, for $i_1, i_2=1, \ldots, m-1; \: j_1, j_2=1, \ldots, n-1$. To abbreviate the notation, we will denote the vertex $(a_i,b_j)$ of $\DD$ by $(i,j)$, for $i=1, \ldots, m-1; \: j=1, \ldots, n-1$, see Figure \ref{fig:12}. 
			
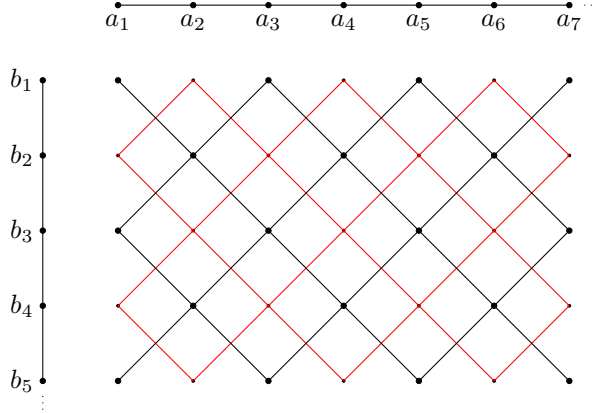
\begin{figure}[!h]
		\begin{tikzpicture}
	
\draw (1,5) -- (7,5);
\draw[dotted] (7.2, 5)--(7.4,5);

\filldraw (1,5) circle (1pt)  node[align=left, below] {$a_1$};
\filldraw (2,5) circle (1pt)  node[align=left, below] {$a_2$};
\filldraw (3,5) circle (1pt)  node[align=left, below] {$a_3$};
\filldraw (4,5) circle (1pt)  node[align=left, below] {$a_4$};
\filldraw (5,5) circle (1pt)  node[align=left, below] {$a_5$};
\filldraw (6,5) circle (1pt)  node[align=left, below] {$a_{6}$};
\filldraw (7,5) circle (1pt)  node[align=left, below] {$a_{7}$};

\draw  (0,4) -- (0,0);
\draw[dotted] (0,-0.2) -- (0,-0.4);

\filldraw (0,4) circle (1pt)  node[align=left, left] {$b_1$};
\filldraw (0,3) circle (1pt)  node[align=left, left]  {$b_2$};
\filldraw (0,2) circle (1pt)  node[align=left, left] {$b_3$};
\filldraw (0,1) circle (1pt)  node[align=left, left] {$b_4$};
\filldraw (0,0) circle (1pt)  node[align=left, left] {$b_5$};

\filldraw (1,4) circle (1pt);
\filldraw (3,4) circle (1pt);
\filldraw (5,4) circle (1pt);
\filldraw (7,4) circle (1pt);

\filldraw (2,3) circle (1pt);
\filldraw (4,3) circle (1pt);
\filldraw (6,3) circle (1pt);

\filldraw (1,2) circle (1pt);
\filldraw (3,2) circle (1pt);
\filldraw (5,2) circle (1pt);
\filldraw (7,2) circle (1pt);

\filldraw (2,1) circle (1pt);
\filldraw (4,1) circle (1pt);
\filldraw (6,1) circle (1pt);

\filldraw (1,0) circle (1pt);
\filldraw (3,0) circle (1pt);
\filldraw (5,0) circle (1pt);
\filldraw (7,0) circle (1pt);

\filldraw (2,4) circle (0.5pt);
\filldraw (4,4) circle (0.5pt);
\filldraw (6,4) circle (0.5pt);

\filldraw (1,3) circle (0.5pt);
\filldraw (3,3) circle (0.5pt);
\filldraw (5,3) circle (0.5pt);
\filldraw (7,3) circle (0.5pt);

\filldraw (2,2) circle (0.5pt);
\filldraw (4,2) circle (0.5pt);
\filldraw (6,2) circle (0.5pt);

\filldraw (1,1) circle (0.5pt);
\filldraw (3,1) circle (0.5pt);
\filldraw (5,1) circle (0.5pt);
\filldraw (7,1) circle (0.5pt);

\filldraw (2,0) circle (0.5pt);
\filldraw (4,0) circle (0.5pt);
\filldraw (6,0) circle (0.5pt);

\draw[color=red] (1,1)--(2,0);
\draw[color=red] (1,3)--(4,0);
\draw[color=red] (2,4)--(6,0);
\draw[color=red] (4,4)--(7,1);
\draw[color=red] (6,4)--(7,3);

\draw[color=red] (1,1)--(4,4);
\draw[color=red] (1,3)--(2,4);
\draw[color=red] (2,0)--(6,4);
\draw[color=red] (4,0)--(7,3);
\draw[color=red] (6,0)--(7,1);

\draw (1,0) -- (5,4);
\draw (1,2) -- (3,4);
\draw (3,0) -- (7,4);
\draw (5,0) -- (7,2);

\draw (1,4) -- (5,0);
\draw (1,2) -- (3,0);
\draw (3,4) -- (7,0);
\draw (5,4) -- (7,2);

		\end{tikzpicture}
		\caption{\small Structure of the product graph $\mathfrak{D}=P_{m-2} \times P_{n-2}$. The black edges are in $\mathfrak{D}_1$, and the red edges are in $\mathfrak{D}_2$. } \label{fig:12}
\end{figure}	
		
 Note that, by Lemma \ref{DD}, $\DD$ has two connected components, one of them, denoted by $\DD_1$, consisting of vertices $(i,j)$ with $i+j$ even, and the other one, denoted by $\DD_2$ with $i+j$ odd, and $\CC$ lies in one of them. 
 
 \begin{rem} \label{rem:f12}
 If $m$ is odd, then the automorphism of $P_m$ which reverses the order of its vertices {\rm(}it also induces an automorphism of $\GG(P_m)${\rm)} switches these components, which are in this case isomorphic graphs, so, after applying this automorphism of $\GG(P_m)$ if necessary, without loss of generality we can assume that $\CC$ lies in a particular component of $\DD$. The same is true if $n$ is odd. However, if both $m$ and $n$ are even, then the two connected components of $\DD$ are not isomorphic, and we should consider two cases, depending on whether $\CC$ lies in one or the other connected component of $\DD$.	
 \end{rem}
 
 Note that Lemmas \ref{EE} and \ref{VE} apply and provide us with a system of equations on the edge labels of $\DD$. The equations of Lemma \ref{VE} get simplified in our case, in particular, in the notation of this lemma we always have $D_1=D_2=1$.
  We will now show that this system of equations has no solutions, provided that restrictions on $\CC$ given by Lemmas \ref{DD} and \ref{CC} hold, and this derives a contradiction. 
 	
	\subsection{Notation}
		 We now introduce some auxiliary notation used in the proof.  
		 
		Note that $\DD$ is a planar graph, so we can think of $\DD$ as a graph on the plane, and use ``compass notation'', with the first coordinate increasing from west to east, and the second coordinate increasing from north to south. So we have the vertex $(1,1)$ in the top left (NW) corner, vertex $(1,n-1)$ in bottom left corner (SW), vertex $(m-1,1)$ in top right corner (NE), and $(m-1, n-1)$ in bottom right (SE) corner.  Every vertex $(i,j)$ in $\DD$ has some of the following incident edges (with a minimum of one): the NW edge, going to $(i-1,j-1)$; the SW edge, going to $(i-1,j+1)$; the NE edge, going to $(i+1,j-1)$; and the SE edge, going to $(i+1,j+1)$. Thus, the vertices of $\DD$ can be subdivided into inner vertices, which have degree 4 --- those which are of the form $(i,j)$ with $1<i<m-1, \: 1<j<n-1$, and boundary vertices --- all the rest. Among boundary vertices there are four corner vertices, 
		which have degree 1, and all the rest, which have degree $2$, see Figure \ref{fig:12}.

		 Fix some $s=1,2$, and recall that $\DD_s$ is one of the connected components of $\DD$. Consider the following auxiliary graph $\DD_s'$: the set of vertices of $\DD_s'$ coincides with the set of vertices of $\DD_s$, and the set of edges of $\DD_s'$ is equal to the union of the set of edges of $\DD_s$ and the set of new edges called boundary, which connect the vertices $(1, k_1)$ with $(1, k_1+2)$ (west boundary); $(m-1, k_2)$ with $(m-1,k_2+2)$ (east boundary); $(k_3,1)$ with $(k_3+2,1)$ (north boundary); $(k_4,n-1)$ with $(k_4+2,n-1)$ (south boundary) for all such natural $k_1, k_2,k_3,k_4$ that the vertices above belong to $\DD_s$. 
Note that $\DD_s'$ contains $\DD_s$ as a subgraph.

Obviously, $\DD_s'$ is also a planar graph, so we can speak about faces of $\DD_s'$ --- the set of all regions bounded by edges, here we do not consider the unbounded region. Let $\mathfrak{F}$ be the set of all such bounded faces of $\DD_s'$ (we omit the index $s$ which is fixed). Abusing the terminology, we will also call them faces of $D_s$. Note that there are two types of faces in $\FF$ --- square faces, which are bounded by four edges, all belonging to $\DD_s$, and triangle (boundary) faces, which are bounded by three edges, one of them boundary and the other two belonging to $\DD_s$. Boundary faces can be further subdivided into west, east, north and south boundary faces, depending on their boundary edge. Each square face has four sides, which are all edges of $\DD_s$ -- the NW side, the SW side, the NE side, and the SE side; for a triangle face only two of the sides are defined. Each square face has four corners, which are all vertices of $\DD_s$ --- the north, south, east and west corner; for a triangle face only three of the corners are defined. Two faces are adjacent if they have a common side, see Figure \ref{fig:13}.

\begin{figure}[!h]
		\begin{tikzpicture}
	
\draw (1,5) -- (8,5);

\filldraw (1,5) circle (1pt)  node[align=left, below] {$a_1$};
\filldraw (2,5) circle (1pt)  node[align=left, below] {$a_2$};
\filldraw (3,5) circle (1pt)  node[align=left, below] {$a_3$};
\filldraw (4,5) circle (1pt)  node[align=left, below] {$a_4$};
\filldraw (5,5) circle (1pt)  node[align=left, below] {$a_5$};
\filldraw (6,5) circle (1pt)  node[align=left, below] {$a_{6}$};
\filldraw (7,5) circle (1pt)  node[align=left, below] {$a_{7}$};
\filldraw (8,5) circle (1pt)  node[align=left, below] {$a_{8}$};

\draw  (0,4) -- (0,0);

\filldraw (0,4) circle (1pt)  node[align=left, left] {$b_1$};
\filldraw (0,3) circle (1pt)  node[align=left, left]  {$b_2$};
\filldraw (0,2) circle (1pt)  node[align=left, left] {$b_3$};
\filldraw (0,1) circle (1pt)  node[align=left, left] {$b_4$};
\filldraw (0,0) circle (1pt)  node[align=left, left] {$b_5$};

\filldraw (1,4) circle (1pt);
\filldraw (3,4) circle (1pt);
\filldraw (5,4) circle (1pt);
\filldraw (7,4) circle (1pt);

\filldraw (2,3) circle (1pt);
\filldraw (4,3) circle (1pt);
\filldraw (6,3) circle (1pt);
\filldraw (8,3) circle (1pt);

\filldraw (1,2) circle (1pt);
\filldraw (3,2) circle (1pt);
\filldraw (5,2) circle (1pt);
\filldraw (7,2) circle (1pt);

\filldraw (2,1) circle (1pt);
\filldraw (4,1) circle (1pt);
\filldraw (6,1) circle (1pt);
\filldraw (8,1) circle (1pt);

\filldraw (1,0) circle (1pt);
\filldraw (3,0) circle (1pt);
\filldraw (5,0) circle (1pt);
\filldraw (7,0) circle (1pt);

\draw (1,0) -- (5,4);
\draw (1,2) -- (3,4);
\draw (3,0) -- (7,4);
\draw (5,0) -- (8,3);

\draw (1,4) -- (5,0);
\draw (1,2) -- (3,0);
\draw (3,4) -- (7,0);
\draw (5,4) -- (8,1);

\draw (7,0) -- (8,1);
\draw (7,4) -- (8,3);

\draw[color=blue] (1,4)--(1,0);
\draw[color=blue] (1,4)--(7,4);
\draw[color=blue] (1,0)--(7,0);
\draw[color=blue] (8,1)--(8,3);

		\end{tikzpicture}
		\caption{\small The graphs $\mathfrak{D}_1$ (black edges) and $\mathfrak{D}_1'$ (black and blue edges) for $m=9$ and $n=6$. } \label{fig:13}
		\end{figure}
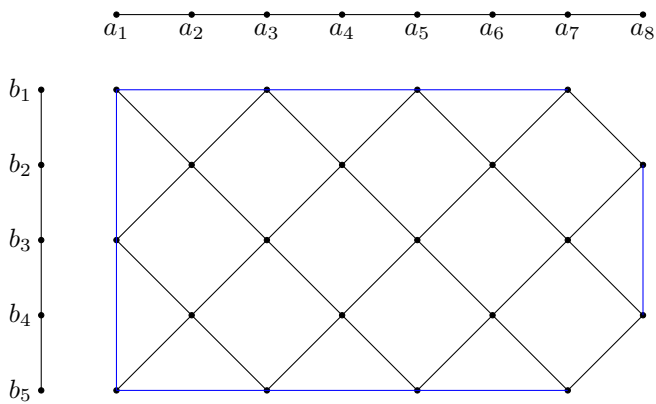	
	
Recall that $\DD_1$ and $\DD_2$ are the connected components of $\DD$, such that $\DD_1$ contains vertices $(i,j)$ with $i+j$ even, and $\DD_2$ with $i+j$ odd. Note that for every face of $\DD_2$ there exists exactly one vertex of $\DD_1$ inside this face (when considered on the plane), and this vertex is not a corner vertex; and vice versa, each vertex of $\DD_1$ which is not a corner vertex belongs to exactly one face of $\DD_2$.	This means that there is a bijection between the faces of $\DD_2$ and the vertices of $\DD_1$ which are not corner vertices. Analogous statement holds with the roles of $\DD_1$ and $\DD_2$ interchanged. We denote the face of $\DD_2$ (or $\DD_1$) corresponding to the non-corner vertex $(i,j)$ of $\DD_1$ (or of $\DD_2$, respectively) by $Q_{i,j}$. This means that if $(i,j)$ is an inner vertex of $\DD$ (i.e., $1 <  i < m-1$, $1 < j < n-1$), then $Q_{i,j}$ is a face of $\DD_1$ if $i+j$ is odd, and a face of $\DD_2$ if $i+j$ is even, and in both cases it is a square face with corners $(i-1,j)$ (west), $(i, j-1)$ (north), $(i+1,j)$ (east) and $(i, j+1)$ (south). For a vertex $(1, j)$, $1<j<n-1$, $Q_{1,j}$ is a face of $\DD_1$, if $j$ is even, and a face of $\DD_2$, if $j$ is odd, and in both cases it is a west boundary triangle face, with the corners $(1, j-1)$ (north), $(2,j)$ (east) and $(1, j+1)$ (south); analogously for other boundaries, see Figure \ref{fig:9}.
	
		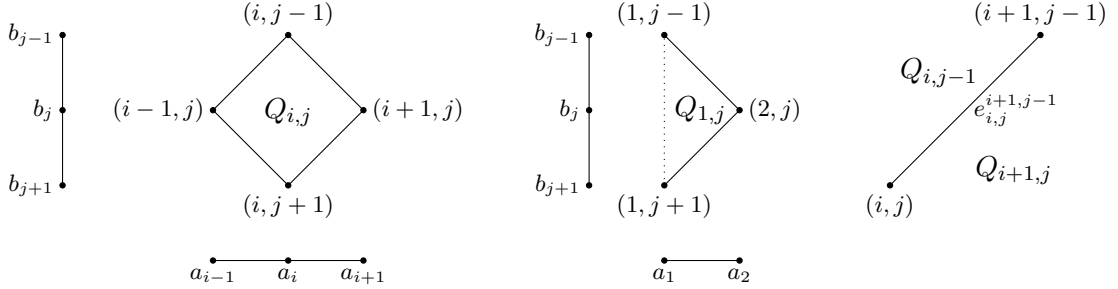
\begin{figure}[!h]
		\begin{tikzpicture}
		\small
		
		\draw (2,1) -- (4,1);
		
		\filldraw (2,1) circle (1pt)  node[align=left, below] {$a_{i-1}$};
		\filldraw (3,1) circle (1pt)  node[align=left, below] {$a_i$};
		\filldraw (4,1) circle (1pt)  node[align=left, below] {$a_{i+1}$};

		\draw  (0,2) -- (0,4);

		\filldraw (0,2) circle (1pt)  node[align=left, left]  {$b_{j+1}$};
		\filldraw (0,3) circle (1pt)  node[align=left, left]  {$b_j$};
		\filldraw (0,4) circle (1pt)  node[align=left, left] {$b_{j-1}$};
		
		\filldraw (2,3) circle (1pt)  node[align=left, left]  {$(i-1,j)$};
		\filldraw (3,2) circle (1pt)  node[align=left, below]  {$(i,j+1)$};
		\filldraw (3,4) circle (1pt)  node[align=left, above] {$(i,j-1)$};
		\filldraw (4,3) circle (1pt)  node[align=right, right] {$(i+1,j)$};

		\draw (2,3)-- (3,4) -- (4,3) -- (3,2) -- (2,3);
		
		\large
		
		\filldraw (3,3.3) node[align=left, below] {$Q_{i,j}$};

		\small
		
		\draw (8,1) -- (9,1);
		
		\filldraw (8,1) circle (1pt)  node[align=left, below] {$a_{1}$};
		\filldraw (9,1) circle (1pt)  node[align=left, below] {$a_2$};

		\draw  (7,2) -- (7,4);

		\filldraw (7,2) circle (1pt)  node[align=left, left]  {$b_{j+1}$};
		\filldraw (7,3) circle (1pt)  node[align=left, left]  {$b_j$};
		\filldraw (7,4) circle (1pt)  node[align=left, left] {$b_{j-1}$};

		\filldraw (8,2) circle (1pt)  node[align=left, below]  {$(1,j+1)$};
		\filldraw (8,4) circle (1pt)  node[align=left, above] {$(1,j-1)$};
		\filldraw (9,3) circle (1pt)  node[align=right, right] {$(2,j)$};

		\draw (8,4)--(9,3)--(8,2);
		\draw[dotted] (8,2)--(8,4);
		
		\large
		
		\filldraw (8,3) node[align=right=-2pt, right] {$Q_{1,j}$};

		\small
		
		\draw (11, 2)--(13,4);
		\filldraw (11,2) circle (1pt)  node[align=left, below]  {$(i,j)$};
		\filldraw (13,4) circle (1pt)  node[align=left, above]  {$(i+1,j-1)$};
		
		\filldraw (12,3)  node[align=right, right]  {$e_{i,j}^{i+1,j-1}$};
		
		\large
		
		\filldraw (11,3.5)  node[align=right, right]  {$Q_{i,j-1}$};
		\filldraw (12,2.2)  node[align=right, right]  {$Q_{i+1,j}$};
		
		\end{tikzpicture}
		\caption{\small Notation for edges and faces of $\mathfrak{D}_1$ or $\mathfrak{D}_2$.} \label{fig:9}
	\end{figure}
	 
	We also denote by $e_{i,j}^{k,l}$ (or $e_{k,l}^{i,j}$) the (non-oriented) edge of $\DD$ which connects the vertices $(i,j)$ and $(k,l)$, for all possible $i,j,k,l$. This means that $e_{i,j}^{k,l}$ is also a side of faces $Q_{i,l}$ and $Q_{k,j}$ (if these faces exist, which is always true except when $(i,l)$ or $(k,j)$ is a corner vertex), see Figure \ref{fig:9}.
	
\subsection{System of equations for the product graph of two paths}

Note that, in our case and in the above notation, Lemma \ref{VE} and Equation (\ref{R}) have the following form. If $w$ is an inner vertex of $\DD$, and $e_1,e_2,e_3,e_4$ are the NW, NE, SE, SW edges of $\DD$ incident to $w$ respectively, all oriented from $w$, then
$$
\begin{array}{l}
				R_1(w)=M_{11}(e_1)+M_{11}(e_4)=M_{11}(e_2)+M_{11}(e_3)=M_{12}(e_1)+M_{12}(e_2)=M_{12}(e_4)+M_{12}(e_3),	\\
		R_2(w)=M_{21}(e_1)+M_{21}(e_4)=M_{21}(e_2)+M_{21}(e_3)=M_{22}(e_1)+M_{22}(e_2)=M_{22}(e_4)+M_{22}(e_3).
\end{array}
$$		
	Also it follows from Lemma \ref{DD} (claim 3, local surjectivity) that if such $w$ is in $\CC \subseteq \DD_s$, then at least one of each pair of the edges $(e_1,e_2)$, $(e_2,e_3)$, $(e_3, e_4)$, $(e_4, e_1)$ is in $\CC$.
For the boundary vertices, we have similar equations.		
For example, if $w$ is the NW corner with the SE edge $e_3$ beginning in $w$, we have
\begin{equation*}
	R_1(w)=M_{11}(e_3)=M_{12}(e_3); \quad R_2(w)=M_{21}(e_3)=M_{22}(e_3).
\end{equation*}
If $w$ is on the west boundary, but not in a corner, and $e_2, e_3$ are the NE, SE edges beginning in $w$ respectively, then
$$
\begin{array}{c}
	R_1(w)=M_{11}(e_2)+M_{11}(e_3)=M_{12}(e_2)=M_{12}(e_3),
R_2(w)=M_{21}(e_2)+M_{21}(e_3)=M_{22}(e_2)=M_{22}(e_3),
\end{array}
$$
and analogous equations hold for the other boundary vertices.

It follows from Lemma \ref{DD} (local surjectivity) that if $w$ is a boundary vertex (of degree 1 or 2) which is in $\CC \subseteq \DD_s$, then all the edges of $\DD_s$ incident to $w$ are also in $\CC$.

\subsection{Face labels}

		For every face $F$ in $\FF$ define two labels as follows. If $F$ is a square face, and $w_1,w_2,w_3,w_4$ are its west, north, east and south corners respectively, then let
\begin{equation}\label{RR}		
		R_1(F)=R_1(w_2)+R_1(w_4), \quad R_2(F)=R_2(w_1)+R_2(w_3).
\end{equation}
		 If $F$ is a triangle face, then exactly one of $w_1,w_2,w_3,w_4$ above will be missing, say $w_1$ (so $F$ is west boundary face), and then define 
\begin{equation}\label{RR'}		
		R_1(F)=R_1(w_2)+R_1(w_4), \quad R_2(F)=R_2(w_3);
\end{equation}		
		 the other cases are analogous (just think of the missing vertex as having labels 0).

		\begin{lemma}\label{RF}
			In the above notation, for every face $F$ in $\FF$ we have $R_1(F)=R_2(F)$.
		\end{lemma}
		\begin{proof}
			Suppose first that $F$ is a square face. Let $w_1,w_2,w_3,w_4$ be the corners of $F$ as above. 
			
			Let $e_1,e_2,e_3,e_4$ be oriented edges which are sides of $F$, such that $e_1$ goes from $w_1$ to $w_2$ (NW side), $e_2$ goes from $w_2$ to $w_3$ (NE side), $e_3$ goes from $w_3$ to $w_4$ (SE side), and $e_4$ goes from $w_4$ to $w_1$ (SW side), see Figure \ref{fig:10}.
			
				\begin{figure}[!h]
				\begin{tikzpicture}
				\filldraw (2,3) circle (1pt)  node[align=left, left]  {$w_1$};
				\filldraw (3,2) circle (1pt)  node[align=left, below]  {$w_4$};
				\filldraw (3,4) circle (1pt)  node[align=left, above] {$w_2$};
				\filldraw (4,3) circle (1pt)  node[align=right, right] {$w_3$};

				\draw[middlearrow={latex}] (2,3) -- (3,4)  node [midway, left=2pt, above] {$e_1$};
				\draw[middlearrow={latex}] (3,4) -- (4,3)  node [midway, right, right] {$e_2$};
				\draw[middlearrow={latex}] (4,3) -- (3,2)  node [midway, right, right=2pt] {$e_3$};
				\draw[middlearrow={latex}] (3,2) -- (2,3)  node [midway, left=4pt, below] {$e_4$};
				
				\large
				\filldraw (3,3)   node {$F$};
				
				\normalsize
				
				\filldraw (8,2) circle (1pt)  node[align=left, below]  {$w_4$};
				\filldraw (8,4) circle (1pt)  node[align=left, above] {$w_2$};
				\filldraw (9,3) circle (1pt)  node[align=right, right] {$w_3$};
				\draw[middlearrow={latex}] (8,4) -- (9,3)  node [midway, right, right=2pt] {$e_2$};
				\draw[middlearrow={latex}] (9,3) -- (8,2)  node [midway, right=2pt, right] {$e_3$};
				\large
				\filldraw (8.3,3)   node {$F$};
				
				\draw[dotted] (8,2)--(8,4);

				\end{tikzpicture}
				\caption{\small Face labels: Lemma \ref{RF} claims that $R_1(w_2)+R_1(w_4)=R_2(w_1)+R_2(w_3)$ for the square face on the left, and $R_1(w_2)+R_1(w_4)=R_2(w_3)$ for the triangle face on the right, and similar for other triangle faces.} \label{fig:10}
			\end{figure}
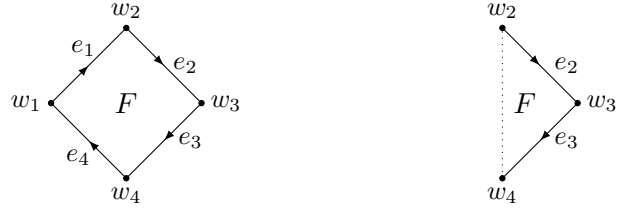		
			 
			By Lemma \ref{VE}, we have 
			$$
			R_1(w_2)=M_{12}(e_1^{-1})+M_{12}(e_2), \quad R_1(w_4)=M_{12}(e_3^{-1})+M_{12}(e_4). 
			$$
			Together with Equation (\ref{RR}) and Lemma \ref{EE}, this means that  
			\begin{equation}\label{ttt}
			\begin{split}
			R_1(F)=R_1(w_2)+R_1(w_4)=M_{12}(e_1^{-1})+M_{12}(e_2)+M_{12}(e_3^{-1})+M_{12}(e_4)  = 
			\\
			= M_{21}(e_1)+M_{21}(e_2^{-1})+M_{21}(e_3)+M_{21}(e_4^{-1}). 
			\end{split}
			\end{equation}
			On the other hand, again by Lemma \ref{VE}, we have 
			$$
			R_2(w_1)=M_{21}(e_1)+M_{21}(e_4^{-1}), \quad R_2(w_3)=M_{21}(e_2^{-1})+M_{21}(e_3), 
			$$
			so, by Equation (\ref{RR}), this gives us 
			$$
			R_2(F)=R_2(w_1)+R_2(w_3)=M_{21}(e_1)+M_{21}(e_4^{-1})+M_{21}(e_2^{-1})+M_{21}(e_3),
			$$
			which is the same as the right-hand side of Equation (\ref{ttt}), so $R_1(F)=R_2(F)$.
			
			If $F$ is a triangle face, then the proof is similar, with some summands missing in the argument above. For example, if $F$ is a west boundary face, $w_2, w_3,w_4$ are its corners as above, and $e_2$ goes from $w_2$ to $w_3$ (NE side), $e_3$ goes from $w_3$ to $w_4$ (SE side), see Figure \ref{fig:10}, then, by Lemma \ref{VE}, Lemma \ref{EE} and Equation (\ref{RR'}), we have
\begin{gather}\notag
			\begin{split}
			R_1(F)=R_1(w_2)+R_1(w_4)=&M_{12}(e_2)+M_{12}(e_3^{-1})=\\
			=&M_{21}(e_2^{-1})+M_{21}(e_3)	=R_2(w_3)=R_2(F).		
			\end{split}
\end{gather}
			Other cases are analogous. 
	\end{proof}

		We fix $s$ equal to 1 or 2 such that $\CC$ is a subgraph of $\DD_s$, as above.

\subsection{Key lemma}		

		The following lemma is key in this proof. It provides us with a way of applying consequently the equations from above to prove that some edges of $\DD_s$ do not belong to $\CC$, until we obtain a contradiction with Lemma \ref{DD}.
		
		\begin{lemma}\label{fork}
			In the above notation, suppose that $Q_1, Q_2$ are two adjacent faces in $\FF$, such that $Q_1$ is either square or west boundary face, $Q_2$ is either square or north boundary face, and NE side of $Q_1$ coincides with SW side of $Q_2$. Suppose, in addition, that if $Q_1$ is square, then the west corner of $Q_1$ does not have a NW edge in $\CC$, and, if $Q_2$ is square, then the north corner of $Q_2$ does not have a NW edge in $\CC$. 
			 
			Then the south corner of $Q_1$ does not have a SE edge in $\CC$, and the east corner of $Q_2$ does not have a SE  edge in $\CC$.

			Analogous three statements hold with all the directions above rotated by $\pi/2$, $\pi$ and $3\pi/2$.
		\end{lemma}
		
		In the statement of Lemma \ref{fork}, when we say that the west corner of $Q_1$ does not have a NW edge in $\CC$, we mean that there is either no such edge in $\DD$ (this will be the case when the west corner of $Q_1$ is on the left boundary, i.e. is of the form $(1,k)$), or there is such an edge in $\DD$, but it does not belong to $\CC$, which is equivalent to saying that all (or just one) of its labels are 0, by Equation (\ref{M0}). 
		
		By the expression ``all the directions rotated by $3\pi/2$'', we mean that in the statement north is changed to east, east -- to south, south -- to west, west -- to north, and NE -- to SE, SE -- to SW, SW -- to NW, NW -- to NE; other rotations are defined analogously in a natural way.
	
		\begin{proof}
			Note that it suffices to prove the original statement above, the proofs of all statements with rotated directions are similar, up to corresponding change of directions. 
			 
			Suppose first that $Q_1$ and $Q_2$ are both square faces.
			  
			Let $w_1$ be the west vertex of $Q_1$, $w_2$ be the north vertex of $Q_1$ (which is also the west vertex of $Q_2$), $w_3$ be the north vertex of $Q_2$, $w_4$ be the east vertex of $Q_2$, $w_5$ be the south vertex of $Q_2$ (which is also the east vertex of $Q_1$), and $w_6$ be the south vertex of $Q_1$.  Let $h_1$ be the edge going from $w_2$ to $w_3$ (NW side of $Q_2$), $h_2$ be the edge going from $w_2$ to $w_5$ (NE side of $Q_1$ and SW side of $Q_2$),  $h_3$ be the edge going from $w_5$ to $w_6$ (SE side of $Q_1$),  $h_4$ be the edge going from $w_1$ to $w_2$ (NW side of $Q_1$), and $h_5$ be the edge going from $w_4$ to $w_5$ (SE side of $Q_2$).  Since we will use only $M_{11}$ and $M_{22}$ labels below, the orientation of edges is not important here, see Figure \ref{fig:11}.
			
			\begin{figure}[!h]
				\begin{tikzpicture}
				\filldraw (2,0) circle (1pt)  node[align=right, right] {$w_6$};
				\filldraw (1,1) circle (1pt)  node[align=right, right] {$w_1$}; 
				\filldraw (3,1) circle (1pt)  node[align=right, right] {$w_5$}; 
				\filldraw (2,2) circle (1pt)  node[left, above=2pt] {$w_2$}; 
				\filldraw (3,3) circle (1pt)  node[align=right, right] {$w_3$}; 
				\filldraw (4,2) circle (1pt)  node[align=right, right] {$w_4$}; 
				
				\draw (2,0) -- (3,1)  node [midway, right=2pt, right] {$h_3$};
				\draw (3,1) -- (4,2)  node [midway, right=2pt, right] {$h_5$};	
				\draw (2,0) -- (1,1);
				\draw (4,2) -- (3,3);
				\draw (3,1) -- (2,2)  node [midway, right, right] {$h_2$};
				\draw (1,1) -- (2,2)  node [midway, left=2pt, above] {$h_4$};
				\draw (2,2) -- (3,3)  node [midway, left=2pt, above] {$h_1$};
				\draw[color=blue] (1,1) -- (0.5,1.5);
				\draw[color=blue] (3,3) -- (2.5,3.5);
				\draw[color=red] (2,0) -- (2.5,-0.5)  node [left, below] {$f_1$};
				\draw[color=red] (4,2) -- (4.5,1.5)  node [right, below] {$f_2$};
				\draw (2,0)--(1.5,-0.5) (1,1)--(0.5,0.5) (3,1)--(3.5,0.5) (4,2)--(4.5,2.5) (3,3)--(3.5,3.5) (2,2)--(1.5,2.5);
				
				\large
				\filldraw (2,1)   node[align=center] {$Q_1$}; 
				\filldraw (3,2)   node[align=center] {$Q_2$}; 
				
				\end{tikzpicture}
				\caption{\small Part of the graph $\mathfrak{D}$ in the proof of Lemma \ref{fork}, consisting of two adjacent faces  in the case when both faces are square. Blue edges are not in $\mathfrak{C}$ by the conditions of the lemma, and red edges are claimed not to be in $\mathfrak{C}$ by the lemma.} \label{fig:11}
			\end{figure}
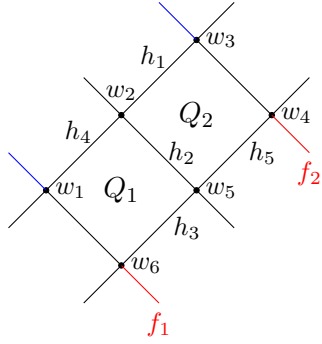	 
			
			We can suppose that there are SE edges in $\DD_s$ both from $w_6$ and $w_4$, otherwise the proof is similar. Let $f_1$ be the SE edge at $w_6$, and $f_2$ be the SE edge at $w_4$.  By Lemma \ref{RF}, we have $R_1(Q_1)=R_2(Q_1)$.  Note that $R_1(Q_1)=R_1(w_2)+R_1(w_6)$ by Equation (\ref{RR}), and $R_1(w_2)=M_{11}(h_1)+M_{11}(h_2)$, $R_1(w_6)=M_{11}(h_3)+M_{11}(f_1)$ by Equation (\ref{R}). Also we have $R_2(Q_1)=R_2(w_1)+R_2(w_5)$ by Equation (\ref{RR}), and $R_2(w_1)=M_{22}(h_4)$, $R_2(w_5)=M_{22}(h_2)+M_{22}(h_5)$ by Equation (\ref{R}) and since, by the conditions of this lemma, $w_1$ has no NW edge in $\CC$ (and so this edge, even if it exists in $\DD_s$, has $M_{22}$ label equal to $0$ by Equation (\ref{M0})). Thus, we have
			\begin{equation}\label{T1}
			R_1(Q_1)=M_{11}(h_1)+M_{11}(h_2)+M_{11}(h_3)+M_{11}(f_1)=R_2(Q_1)=M_{22}(h_4)+M_{22}(h_2)+M_{22}(h_5).
			\end{equation}
In the same way, by Lemma \ref{RF}, we have $R_1(Q_2)=R_2(Q_2)$. Note that $R_1(Q_2)=R_1(w_3)+R_1(w_5)$, by Equation (\ref{RR}), and $R_1(w_	3)=M_{11}(h_1)$, $R_1(w_5)=M_{11}(h_2)+M_{11}(h_3)$ by Equation (\ref{R}) and since, 	by the conditions of this lemma, $w_3$ has no NW edge in $\CC$ (and so this edge, even if it exists in $\DD_s$, has $M_{11}$ label equal to $0$ by Equation (\ref{M0})). Also, we have $R_2(Q_2)=R_2(w_2)+R_2(w_4)$ by Equation (\ref{RR}), and $R_2(w_2)=M_{22}(h_4)+M_{22}(h_2)$, $R_2(w_4)=M_{22}(h_5)+M_{22}(f_2)$ by Equation (\ref{R}).  Thus, we have 
		\begin{equation}\label{T2}
			R_1(Q_2)=M_{11}(h_1)+M_{11}(h_2)+M_{11}(h_3)=R_2(Q_2)=M_{22}(h_4)+M_{22}(h_2)+M_{22}(h_5)+M_{22}(f_2).
		\end{equation}
	
Substracting (\ref{T2}) from (\ref{T1}), we obtain 
	\begin{equation}\label{T}
		R_1(Q_1)-R_1(Q_2)=M_{11}(f_1)=R_2(Q_1)-R_2(Q_2)=-M_{22}(f_2),
	\end{equation}	
	but $M_{11}(f_1) \geq 0$, $-M_{22}(f_2) \leq 0$, so (\ref{T}) implies that $M_{11}(f_1)=M_{22}(f_2)=0$, and so $f_1, f_2 \notin \CC$ by Equation (\ref{M0}), as required. 
	
	The other cases, when $Q_1$ is the west boundary, or $Q_2$ is the north boundary, or both, are analogous, with the only difference in the proof being that some summands do not appear (we can think of them as being equal to $0$). For example, if $Q_1$ is a west boundary, and $Q_2$ is a square, then all the notation and equalities are almost the same as above, except that $w_1$ and $h_4$ do not exist, so $R_2(Q_1)=R_2(w_5)=M_{22}(h_2)+M_{22}(h_5)$, and $R_2(w_2)=M_{22}(h_2)$, so $R_2(Q_2)=M_{22}(h_2)+M_{22}(h_5)+M_{22}(f_2)$. The same argument as above applies, hence the lemma.
		\end{proof}

\subsection{Case of odd $m$ or $n$}	
	
	We now continue with the proof of the theorem.  Suppose first that at least one of $m$ and $n$ is odd. Then, by Remark \ref{rem:f12}, we can always choose one of the connected components of $\DD_1$ and $\DD_2$ of $\DD$, and suppose that $\CC$ is a (connected) subgraph of this component. We will choose now $\DD_2$, and so we can suppose that $\CC \subseteq \DD_2$. 

	Note that all faces $Q_{i,j}$ of $\DD_2$ have even $i+j \geq 4$. For an even $k$ such that $4 \leq k \leq n+1$ denote by $S_k$ the set of all faces $Q_{i,j}$ of $\DD_2$ such that $i+j=k$. Then for a given $k \leq n-1$ all faces in $S_k$ form a ``SW-NE diagonal'': 
	$$
	S_k = \{ Q_{1,k-1}, \: Q_{2,k-2}, \ldots, Q_{k-2,2}, \: Q_{k-1,1} \}, 
	$$ 
	with the NE side of $Q_{i,k-i}$ coinciding with the SW side of $Q_{i+1,k-i-1}$ for $i=1,\ldots, k-2$; for $k = n$ or $k=n+1$ (the one which is even) $S_k$  is as above, but with the first face missing, and if $k=n+1$ and $m=n+1$, also the last face missing. 
	
	Let also $E_k$, for an even $k$ such that $4 \leq k \leq n+1$, be the set of all edges of $\DD_2$ which go from NW to SE and are the sides of some faces in $S_k$. This means that, if $k \leq n$, then 
	$$
	E_k =  \{ e_{1,k-2}^{2,k-1}, \: e_{2, k-3}^{3,k-2}, \ldots, e_{k-3,2}^{k-2,3}, \: e_{k-2,1}^{k-1,2}    \};
	$$   
for $k=n+1$ (in the case it is even) $E_k$ is as above, but with the first edge missing. Note also that each edge set $E_k$, for $4 \leq k \leq n+1$, separates our graph $\DD_2$ in two, which means that, after deleting all the edges of $E_k$ from $\DD_2$, the remaining graph will become disconnected, namely it will have two connected components.	

		The idea is to proceed from NW to SE, showing that the edges in $E_k$ are not in $\CC$, with increasing $k$, until we get a contradiction when $k=n$ or $k=n+1$.

		First apply Lemma \ref{fork} to $Q_{1,3}$ and $Q_{2,2}$; the conditions of the lemma are satisfied, since NE side of $Q_{1,3}$ coincides with the SW side of $Q_{2,2}$, $Q_{1,3}$ is west boundary face, and the north vertex of $Q_{2,2}$ is on the north boundary of $\DD$, in particular it has no NW edge.  By Lemma \ref{fork}, the south corner of $Q_{1,3}$ does not have a SE edge in $\CC$, which, in our terms, means that if $n \geq 6$, then $e_{1,4}^{2,5} \notin \CC$ (if $n=5$, then the south corner of $Q_{1,3}$ is on the south boundary, so this condition is vacuous); and the east corner of $Q_{2,2}$ does not have a SE edge in $\CC$, which in our terms means that $e_{3,2}^{4,3} \notin \CC$, see Figure \ref{fig:20}.
		
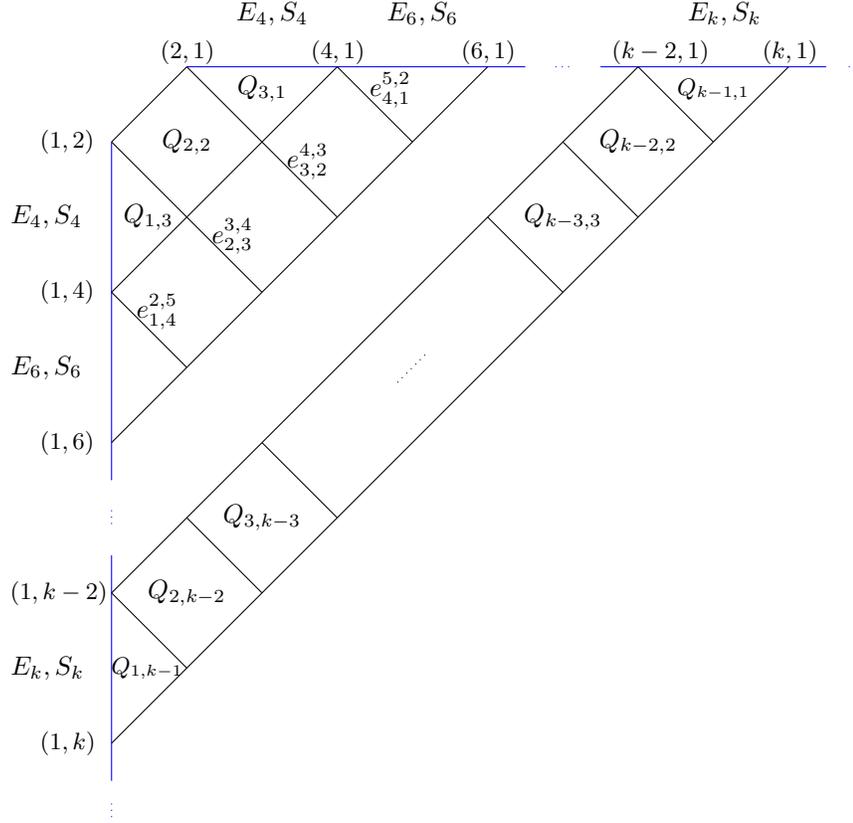
\begin{figure}[!h]
\begin{tikzpicture}
	
	\draw[color=blue] (0,9)--(0,4.5) (1,10)--(5.5,10);
	\draw[color=blue,dotted] (0,4.1) -- (0,3.9) (5.9, 10) -- (6.1,10);
	
	\draw (0,9) -- (1, 10) (0,7) -- (3,10) (0,5) -- (5,10);
	\draw (0,7) -- (1,6) (0,9) -- (2,7) (1,10) -- (3, 8) (3,10) -- (4,9);
	
	\draw[color=blue] (0,3.5)--(0,0.5) (6.5,10) -- (9.5,10);
	\draw (0,3) --   (7,10);
	\draw (0,1) --  (9,10);
	\draw[color=blue,dotted] (0,0.2) -- (0,0) (9.8,10)--(10,10);
	\draw (0,3) -- (1,2) (1,4) -- (2,3) (2,5)--(3,4) (5,8) -- (6,7) (6,9) -- (7,8) (7,10)--(8,9);
	\draw[dotted] (3.8,5.8)--(4.2,6.2);

	\draw (0,8)  node[right=1pt] {$Q_{1,3}$};
		\draw (1,9)  node {$Q_{2,2}$};
			\draw (2,10)  node[right, below] {$Q_{3,1}$};

		\draw (-1.5,8)  node[right=1pt] {$E_4, S_4$};	
			\draw (-1.5,6)  node[right=1pt] {$E_6, S_6$};
				\draw (-1.5,2)  node[right=1pt] {$E_k, S_k$};
				
			\draw (1.5,10.7)  node[right=1pt] {$E_4, S_4$};	
			\draw (3.5,10.7)  node[right=1pt] {$E_6, S_6$};
				\draw (7.5,10.7)  node[right=1pt] {$E_k, S_k$};	
					
				\draw (1,6.75)  node[left] {$e_{1,4}^{2,5}$};	
				\draw (2,7.75)  node[left] {$e_{2,3}^{3,4}$};	
				\draw (3,8.75)  node[left] {$e_{3,2}^{4,3}$};
				\draw (4.1,9.7)  node[left] {$e_{4,1}^{5,2}$};		

		\small
	\draw (-0.05,2)  node[right=-2pt] {$Q_{1,k-1}$};
	\normalsize
		\draw (1,3)  node {$Q_{2,k-2}$};
			\draw (2,4)  node {$Q_{3,k-3}$};
				\draw (6,8)  node {$Q_{k-3,3}$};
			\draw (7,9)  node {$Q_{k-2,2}$};
			\small
	\draw (8,9.7)  node {$Q_{k-1,1}$};

	\draw (-1.1,9)  node[right=1pt] {$(1,2)$};	
			\draw (-1.1,7)  node[right=1pt] {$(1,4)$};
				\draw (-1.1,5)  node[right=1pt] {$(1,6)$};
				\draw (-1.5,3)  node[right=1pt] {$(1,k-2)$};
				\draw (-1.1,1)  node[right=1pt] {$(1,k)$};
				
			\draw (0.5,10.2)  node[right=1pt] {$(2,1)$};	
			\draw (2.5,10.2)  node[right=1pt] {$(4,1)$};
				\draw (4.5,10.2)  node[right=1pt] {$(6,1)$};	
					\draw (6.5,10.2)  node[right=1pt] {$(k-2,1)$};
				\draw (8.5,10.2)  node[right=1pt] {$(k,1)$};	

\end{tikzpicture}
\caption{\small Part of the graph $\mathfrak{D}_2$ in the case of odd $m$ or $n$. The black edges are in $\mathfrak{D}_2$, and the blue edges are the boundary edges of $\mathfrak{D}_2'$, which are not in $\mathfrak{D}_2$.} \label{fig:20}
\end{figure}
		
		 	Analogously, apply Lemma \ref{fork} to $Q_{2,2}$ and $Q_{3,1}$; the conditions are again satisfied. We obtain that the south corner of $Q_{2,2}$ does not have a SE edge in $\CC$, which means that $e_{2,3}^{3,4} \notin \CC$, and the east corner of $Q_{3,1}$ does not have a SE edge in $\CC$, which means that $e_{4,1}^{5,2} \notin \CC$ (recall that $m \geq 6$). Thus, all the edges of $E_6$ are not in $\CC$, but $\CC$ is connected, so it should be contained in one of the connected components of the graph obtained by deleting the edges of $E_6$ from $\DD_2$. It cannot be the ``NW component'' (i.e., the one containing the vertex $(2,1)$), since in this component there are no vertices of the form $(5, i)$ for some $i$, so it does not project surjectively to $\ti{\Gamma}_1$, but $\CC$ should, by Lemma \ref{DD}. This means that $\CC$ is contained in the other component. In particular, we see that none of the edges of $E_4$, as well as $E_6$, are in $\CC$. 
		 	
		 	We claim that none of the edges of $E_k$, for even $k$ such that $4 \leq k \leq n+1$, are in $\CC$. (For $n=5$ and $n=6$ this is already proved). We prove this by induction, for $E_4$ and $E_6$ it is proved above, so suppose the claim is proved for all even $k$ such that $4 \leq k \leq k_0$, where $6 \leq k_0 \leq n-1$, $k_0$ is even, and let us prove it for $k=k_0+2$.

		 	Consider all the pairs of adjacent faces in $S_{k_0}$: $Q_{1,k_0-1}$ and $Q_{2,k_0-2}$, $Q_{2,k_0-2}$ and $Q_{3,k_0-3}$, $\ldots$, $Q_{k_0-2,2}$ and $Q_{k_0-1,1}$. By induction hypothesis, all these pairs satisfy the conditions of Lemma \ref{fork}. Indeed, all the existing in $\DD_2$ NW edges at the north and west corners of the faces from $S_{k_0}$ belong to $E_{k_0-2}$ by definition, and so they do not belong to $\CC$. Applying Lemma \ref{fork} to all the pairs above, we obtain that all the SE edges at the south and east corners of the faces from $S_{k_0}$, which are exactly all the edges from $E_{k_0+2}$, are not in $\CC$, and the claim is proved.

		 	Thus, in particular, if $n$ is even, then no edges from $E_n$ are in $\CC$, and if $n+1$ is even, then no edges from $E_{n+1}$ are in $\CC$. Suppose first that $n$ is even. Then deleting all the edges of $E_n$ from $\DD_2$ results in a graph with two connected components, none of which projects surjectively to $\ti{\Gamma}_1$ (namely, one of them does not contain vertices with the first coordinate equal to 1, and the other one does not contain vertices with the first coordinate equal to $n-1$). This is impossible by Lemma \ref{DD}. In the same way, if $n$ is odd, then again deleting all the edges of $E_{n+1}$ from $\DD_2$ results in a graph with two connected components, none of which projects surjectively to $\ti{\Gamma}_1$ (namely, one of them does not contain vertices with the first coordinate equal to 1, and the other one does not contain vertices with the first coordinate equal to $n$), and this is impossible by Lemma \ref{DD}. 
		 	
		 	Thus, $\GG(P_m)$ and $\GG(P_n)$ are not commensurable if $m > n \geq 5$ and at least one of $m$ and $n$ is odd.

	\subsection{Case of even $m$ and $n$}	 	
		 	It remains to consider the case when both $m$ and $n$ are even and $m > n \geq 6$; in particular, $m \geq n+2$. In this case $\DD_1$ (the connected component of $\DD$ containing $(1,1)$) contains all four corners of $\DD$, and $\DD_2$ contains no corners. We now have to consider two subcases, depending on whether $\CC$ lies inside $\DD_1$ or $\DD_2$.

		Suppose first that $\CC$ lies inside $\DD_2$. Then similarly to the case when $n$ and $m$ are odd we can derive a contradiction. Indeed, in the above notation and in the same way as above we can prove by induction that none of the edges of $E_k$, for even $k$ such that $4 \leq k \leq n$, are in $\CC$. In particular, no edges from $E_n$ are in $\CC$, and deleting all the edges of $E_n$ from $\DD_2$ results in a graph with two connected components, none of which projects surjectively to $\ti{\Gamma}_1$ (namely, one of them does not contain vertices with the first coordinate equal to 1, and the other one does not contain vertices with the first coordinate equal to $n-1$). This is impossible by Lemma \ref{DD}.

		So it remains to consider the case when $\CC$ lies inside $\DD_1$, which is more subtle. 
		 
		For a vertex $(i,j)$ of $\DD_1$ denote by $A_{i,j}$ the set of all edges on the ``NW-SE diagonal'' of $\DD_1$ passing through $(i,j)$, i.e.,  $A_{i,j}$ consists of all edges connecting vertices $(i',j')$ of $\DD_1$ with $i'-j'=i-j$. These are the edges of the form $e_{i+k,j+k}^{i+k+1,j+k+1}$, where $\max \{1-i, 1-j\} \leq k \leq \min \{m-2-i, n-2-j\}$. Note that every $A_{i,j}$ is equal to $A_{i',j'}$ with $i'=1$ or $j'=1$ (west or north boundary), and to $A_{i'',j''}$ with $i''= m-1$ or $j''=n-1$ (east or south boundary).
		
		Analogously, for a vertex $(i,j)$ of $\DD_2$ denote by $B_{i,j}$ the set of all edges on the ``SW-NE diagonal'' of $\DD_2$ passing through $(i,j)$, i.e., $B_{i,j}$ consists of all edges connecting vertices $(i',j')$ of $\DD_1$ with $i'+j'=i+j$. These are the edges of the form $e_{i+k,j-k}^{i+k+1,j-k-1}$, where $\max \{1-i, j+1-n\} \leq k \leq \min\{m-2-i, j-2\}$. Note that every $B_{i,j}$ is equal to $B_{i',j'}$ with $i'=1$ or $j'=n-1$ (west or south boundary), and to $B_{i'',j''}$ with $i''= m-1$ or $j''=1$ (east or north boundary). 
		
		Recall that $\DD_1$ consists of all vertices $(i,j)$ of $\DD$ such that $i+j$ is even (or, equivalently, $i-j$ is even). Successive application of Lemma \ref{fork} allows us to prove that, informally speaking, every second diagonal of $\DD_1$ is not in $\CC$, which is the content of the following lemma.

\begin{lemma}\label{claim}		
		In the above notation, with even $m$ and $n$, suppose that $(i,j)$ is a vertex of $\DD_1$ {\rm(}i.e. $i+j$ and $i-j$ are even, $1 \leq i \leq m-1$, $1 \leq j \leq n-1${\rm)}. If one of the following conditions holds:
		\begin{enumerate}		
		\item $i-j$ is $2$ modulo $4$;  
		\item $(m-i) - (n-j)$ is $2$ modulo $4$, or, equivalently, $i-j+(n-m)$ is $2$ modulo $4$;
		\end{enumerate}
	then all the edges of $A_{i,j}$ are not in $\CC$. 
	If one of the following conditions holds:
		\begin{enumerate}[resume]		
		\item $i+j-m$ is $2$ modulo $4$;
		\item $i+j-n$ is $2$ modulo $4$.
		\end{enumerate}
	then all the edges of $B_{i,j}$ are not in $\CC$.
\end{lemma}
\begin{proof}		
		We first prove the first claim, so we suppose that $i-j$ is $2$ modulo $4$, and we need to prove that all the edges of $A_{i,j}$ are not in $\CC$. According to the remarks above, it suffices to prove the claim when $i=1$ or $j=1$. 
		
		We start by proving the claim for $A_{1,3}$ and $A_{3,1}$. The idea is to proceed diagonally from NW to SE  successively applying Lemma \ref{fork}.  Note that 
		$$
		A_{1,3} = \{ e_{1,3}^{2,4}, e_{2,4}^{3,5}, \ldots, e_{n-4,n-2}^{n-3,  n-1} \}, \quad A_{3,1}=\{ e_{3,1}^{4,2}, e_{4,2}^{5,3}, \ldots, e_{n,n-2}^{n+1,n-1} \}. 
		$$ 
		Note that the conditions of Lemma \ref{fork} are satisfied for $Q_{1,2}$ and $Q_{2,1}$ (since $Q_{1,2}$ is west boundary face, and $Q_{2,1}$ is north boundary face), so applying this lemma we deduce that $e_{1,3}^{2,4}, \: e_{3,1}^{4,2} \notin \CC$. Furthermore, $Q_{2,3}$ and $Q_{3,2}$ satisfy conditions of Lemma \ref{fork} (since the west corner of $Q_{2,3}$ is on the west boundary of $\DD$, and the north corner of $Q_{3,2}$ is on the north boundary), so applying this lemma we deduce that $e_{2,4}^{3,5}, \: e_{4,2}^{5,3} \notin \CC$. 
	
	Proceeding by induction, we prove that $e_{k-2,k}^{k-1,k+1}, \: e_{k,k-2}^{k+1,k-1} \notin \CC$ for $3 \leq k \leq n-2$.  Indeed, suppose that this is true for $3 \leq k \leq k_0$, for some $4 \leq k_0 \leq n-3$ (which is the case for $k_0=3, 4$ as shown above), and prove it for $k=k_0+1$. The faces $Q_{k_0-1,k_0}$ and $Q_{k_0,k_0-1}$ satisfy the conditions of Lemma \ref{fork} (since the NW edge at the west corner of $Q_{k_0-1,k_0}$ is $e_{k_0-3,k_0-1}^{k_0-2,k_0} \notin \CC$ by the induction hypothesis for $k=k_0-1$, and the NW edge at the north corner of $Q_{k_0,k_0-1}$ is $e_{k_0-1,k_0-3}^{k_0,k_0-2} \notin \CC$ by the  induction hypothesis for $k=k_0-1$). So, applying Lemma \ref{fork}, we get that $e_{k_0-1,k_0+1}^{k_0, k_0+2}, \: e_{k_0+1,k_0-1}^{k_0+2,k_0} \notin \CC$, which is exactly what we wanted. 
	
		This already shows that all the edges in $A_{1,3}$ are not in $\CC$, and just two more edges from $A_{3,1}$ remain. Applying Lemma \ref{fork} to the faces $Q_{n-3,n-2}$ and $Q_{n-2,n-3}$ (which is possible since we proved above that $e_{n-5,n-3}^{n-4,n-2}, \: e_{n-3,n-5}^{n-2,n-4} \notin \CC$), we obtain that $e_{n-1,n-3}^{n,n-2} \notin \CC$. Applying Lemma \ref{fork} to the faces $Q_{n-2,n-1}$ and $Q_{n-1,n-2}$ (which is possible since we proved above that $e_{n-4,n-2}^{n-3,n-1}, \: e_{n-2,n-4}^{n-1,n-3} \notin \CC$), we obtain that $e_{n,n-2}^{n+1,n-1} \notin \CC$. We conclude that all the edges of $A_{1,3}$ and $A_{3,1}$ are not in $\CC$, see Figure \ref{fig:21}.
		
		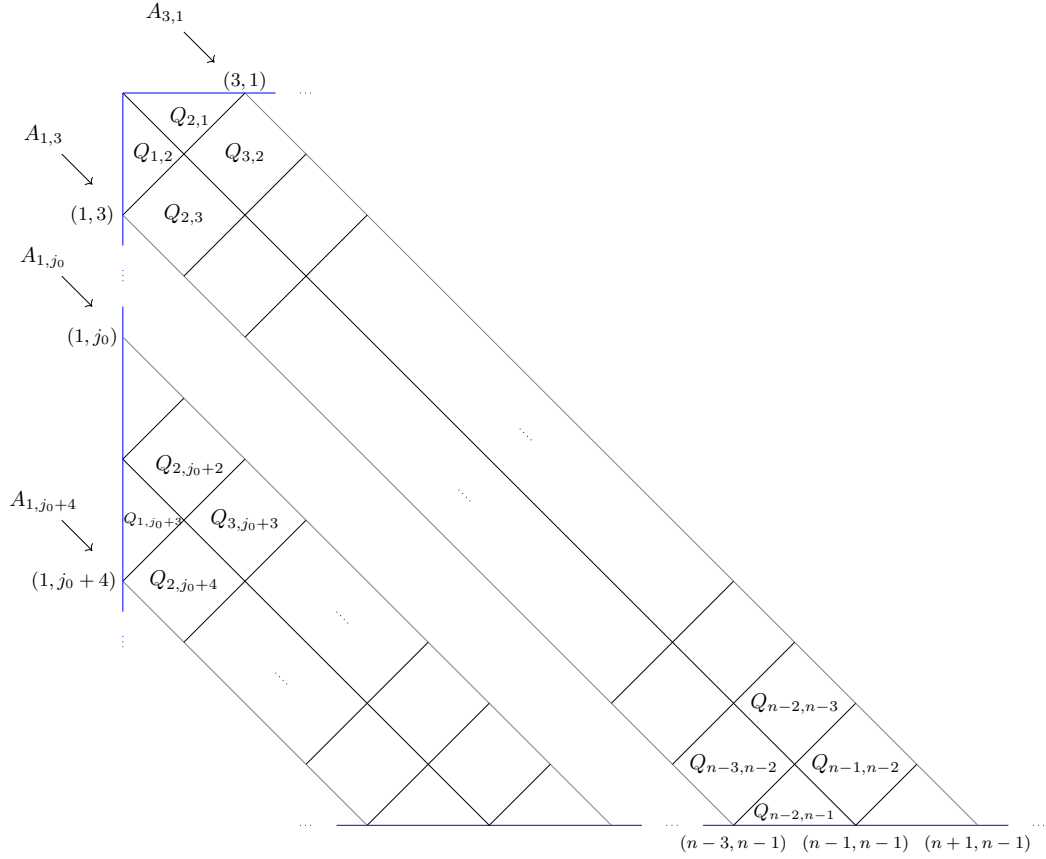
\begin{figure}[!h]
			\resizebox{400pt}{!}{
				\begin{tikzpicture}

				\draw[color=blue] (0,10) -- (0,7.5);
				\draw[color=blue,dotted] (0,7.1) -- (0, 6.9);
				\draw[color=blue] (0,6.5)-- (0,1.5);
				\draw[color=blue,dotted] (0,1.1) -- (0,0.9);
				
				\draw[color=blue] (0,10) -- (2.5,10);
				\draw[color=blue, dotted] (2.9,10)--(3.1,10);
				
				\draw[color=blue] (9.5,-2)--(14.5,-2);
				\draw[color=blue,dotted] (14.9,-2)--(15.1,-2);
				\draw[color=blue,dotted] (8.9,-2)--(9.1,-2);
				
				\draw[color=blue] (3.5,-2) -- (8.5, -2);
				\draw[color=blue, dotted] (2.9,-2)--(3.1,-2);

				\draw (0,10) -- (12,-2);
				
				\draw[color=gray] (0,8) -- (10,-2) (2,10) -- (14,-2);
				
				\draw (0,4) -- (6,-2);
				\draw[color=gray] (0,6) -- (8,-2) (0,2)--(4,-2);

				\draw (0,8) -- (2,10) (1,7)-- (3,9) (2,6) -- (4,8);

				\draw (0,4) -- (1,5) (0,2) -- (2,4);
				\draw (1,1) -- (3,3) (3,-1) -- (5,1) (4,-2) -- (6,0) (6,-2)--(7,-1);
				
				\draw (12,-2) -- (13,-1) (10,-2) -- (12,0) (9,-1) -- (11,1) (8,0) -- (10,2);
				
				\draw[dotted] (2.5,0.5) -- (2.7,0.3) (3.5,1.5) -- (3.7, 1.3);
				\draw[dotted] (5.5,3.5) -- (5.7,3.3) (6.5,4.5) -- (6.7, 4.3);

				\draw (0.5,9)  node {$Q_{1,2}$};
				\draw (1.1,9.6)  node {$Q_{2,1}$};
				\draw (1,8)  node {$Q_{2,3}$};
				\draw (2,9)  node {$Q_{3,2}$};
				
				\draw (10,-1)  node {$Q_{n-3,n-2}$};
				\draw (11,0)  node {$Q_{n-2,n-3}$};
				\draw (12,-1)  node {$Q_{n-1,n-2}$};
				\small
				\draw (11,-1.8)  node {$Q_{n-2,n-1}$};
				\footnotesize
				\draw (10,-2.3)  node {$(n-3,n-1)$};
				\draw (12,-2.3)  node {$(n-1,n-1)$};
				\draw (14,-2.3)  node {$(n+1,n-1)$};
				
				\small
				\draw (-0.5,8)  node {$(1,3)$};
				\draw (2, 10.2)  node {$(3,1)$};
				
				\draw[->] (-1, 9)--(-0.5,8.5);
				\draw[->] (1, 11)--(1.5,10.5);
				
				\normalsize
				
				\draw (-1.3,9.3)  node {$A_{1,3}$};
				\draw (0.7,11.3)  node {$A_{3,1}$};
				
				\draw[->] (-1, 7)--(-0.5,6.5);
				\draw[->] (-1, 3)--(-0.5,2.5);
				
				\small 
				
				\draw (-0.5,6)  node {$(1,j_0)$};
				\draw (-0.8, 2)  node {$(1,j_0+4)$};
				
				\normalsize
				
				\draw (-1.3,7.3)  node {$A_{1,j_0}$};
				\draw (-1.3,3.3)  node {$A_{1,j_0+4}$};
				
				\scriptsize 
				\draw (0.5,3)  node {$Q_{1,j_0+3}$};
				\normalsize
				\draw (1.1,3.9)  node {$Q_{2,j_0+2}$};
				\draw (1,2)  node {$Q_{2,j_0+4}$};
				\draw (2,3)  node {$Q_{3,j_0+3}$};
				
				\end{tikzpicture}
			}
			\caption{\small Part of the graph $\mathfrak{D}_1$ in the case of even $m$ and $n$, as in the proof of Lemma \ref{claim}. The black edges are in $\mathfrak{D}_1$, and the blue edges are the boundary edges of $\mathfrak{D}_1'$, which are not in $\mathfrak{D}_1$.} \label{fig:21}
		\end{figure}
		
		We now prove that all the edges in $A_{1,j}$ are not in $\CC$ for all $3 \leq j \leq n-1$, $j$ is $3$ modulo $4$ (so that $1-j$ is $2$ modulo $4$). For $j=3$ this is proved above. We proceed by induction. Let $3 \leq j_0 \leq n-5$, $j_0$ is $3$ modulo $4$, and suppose this is true for all $j$ which are $3$ modulo $4$, $3 \leq j \leq j_0$, we will prove this is also true for $j=j_0+4$. Note that if $j_0+4=n-1$, then the claim is vacuous, since $A_{1,j_0+4}$ is empty, so we can suppose that $j_0+4 \leq n-3$, or $j_0 \leq n-7$.
		
		Note that 
		$$
		A_{1,j_0}= \{ e_{1,j_0}^{2,j_0+1}, e_{2,j_0+1}^{3,j_0+3}, \ldots, e_{n-j_0-1,n-2}^{n-j_0,n-1} \} , \quad A_{1,j_0+4} = \{e_{1,j_0+4}^{2,j_0+5}, e_{2,j_0+5}^{3,j_0+6}, \ldots,  e_{n-j_0-5,n-2}^{n-j_0-4,n-1} \}
		$$

	   We know by induction hypothesis that all the edges in $A_{1,j_0}$ are not in $\CC$, and need to prove the same for $A_{1,j_0+4}$. Applying Lemma \ref{fork} to $Q_{1,j_0+3}$ and $Q_{2,j_0+2}$ (this is possible, since $Q_{1,j_0+3}$ is a west boundary face, and the NW edge at the north corner of $Q_{2,j_0+2}$ is $e_{1,j_0}^{2,j_0+1} \in A_{1,j_0}$, so it is not in $\CC$), we obtain that $e_{1,j_0+4}^{2,j_0+5} \notin \CC$ (since it is the SE edge at the south vertex of $Q_{1,j_0+3}$).

	  Now apply Lemma \ref{fork} to  $Q_{2,j_0+4}$ and $Q_{3,j_0+3}$ (this is possible, since the west corner of $Q_{2,j_0+4}$ is on the west boundary of $\DD$, and the NW edge at the north corner of $Q_{3,j_0+3}$ is $e_{2,j_0+1}^{3,j_0+2} \in A_{1,j_0}$, so it is not in $\CC$), we obtain that $e_{2,j_0+5}^{3,j_0+6} \notin \CC$ (since it is the SE edge at the south vertex of $Q_{2,j_0+4}$). 
	  
	   If $j_0=n-7$, so $j_0+4=n-3$ and $A_{1,j_0+4}$ contains only two edges, then we are done; otherwise, $j_0 \leq n-9$, and we proceed by (local) induction (inside the main induction) to show that $e_{1+k,j_0+4+k}^{2+k,j_0+5+k } \notin \CC$ for $0 \leq k \leq n-j_0-6$. Suppose this is true for all $0 \leq k \leq k_0$, for some $2 \leq k_0 \leq n-j_0-7$ (which is the case for $k_0=0$ and $k_0=1$, as shown above), and we need to prove it for $k=k_0+1$. Indeed, we can apply Lemma \ref{fork} to $Q_{k_0+2, j_0+k_0+4}$ and $Q_{k_0+3,j_0+k_0+3}$ (this is possible, since the NW edge at the west corner of $Q_{k_0+2,j_0+k_0+4}$ is $e_{k_0,j_0+k_0+3}^{k_0+1,j_0+k_0+4}$, which is not in $\CC$ by the (local) induction hypothesis for $k=k_0-1$, and the NW edge at the north corner of $Q_{k_0+3,j_0+k_0+3}$ is $e_{k_0+2,j_0+k_0+1}^{k_0+3,j_0+k_0+2}$, which is in $A_{1,j_0}$, and so also not in $\CC$). So, we get that the SE edge at the south vertex of $Q_{k_0+2, j_0+k_0+4}$, which is $e_{k_0+2, j_0+k_0+5}^{k_0+3,j_0+k_0+6}$, is not in $\CC$, and this is exactly what we need (in the local induction). 
	   
	   This shows that all the edges of $A_{1, j_0+4}$ are not in $\CC$, and this is what we need in the main induction. Thus all the edges in $A_{1,j}$ for $3 \leq j \leq n-1$, $j$ is $3$ modulo $4$, are not in $\CC$, see Figure \ref{fig:21}.
	   
To prove the lemma, it remains to show that all the edges in $A_{i,1}$ for $3 \leq i \leq m-1$, $i$ is $3$ modulo $4$, are not in $\CC$. The proof is similar to the one for edges $A_{1,j}$, but formally we need to consider two cases, depending on whether $A_{i,1}$ finishes on the south or east boundary of $\DD$: when $3 \leq i \leq m-n+1$, and when $m-n+1 < i \leq m-1$.

 	Suppose first that $3 \leq i \leq m-n+1$. If $i=3$, then the claim is already proved above. We proceed by induction. Let $3 \leq i_0 \leq m-n-3$, and suppose all the edges of $A_{i,1}$ are not in $\CC$ for all $i$ which are $3$ modulo $4$, $3 \leq i \leq i_0$, we need to prove that all the edges of $A_{i_0+4,1}$ are also not in $\CC$. This can be done by Lemma \ref{fork}, applied successively (by induction, as above) to the pairs of faces ($Q_{i_0+2,2}$, $Q_{i_0+4,1}$), ($Q_{i_0+3,3}$, $Q_{i_0+5,2}$), $\ldots$,  ($Q_{i_0+n-1,n-1}, Q_{i_0+n,n-2}$).

	Finally, the proof that all the edges in $A_{i,1}$ for $m-n+1 < i \leq m-1$, $i$ is $3$ modulo $4$, are not in $\CC$, is  analogous to the proofs above; we omit the details. This proves the first claim of the lemma.

	Now note that the second, third and fourth claims of the lemma can be obtained from the first claim by rotating all the directions by $\pi$, $\pi/2$ and $3\pi/2$ respectively. Rotation by $\pi/2$ means replacing $i$ by $m-i$, leaving $j$ unchanged, and changing $A$'s to $B$'s; rotation by $\pi$ means replacing $i$ by $m-i$ and $j$ by $n-j$, without changing the $A$'s; and rotation by $3\pi/2$ means leaving $i$ unchanged, replacing $j$ by $n-j$ and changing $A$'s to $B$'s. So the proof is similar to the one above, with application of the corresponding claims of Lemma \ref{fork}. This proves the lemma.
\end{proof}

Recall that both $m$ and $n$ are even. Suppose first that one of $m$ and $n$ is $0$ modulo $4$, and the other is $2$ modulo $4$. Then every vertex $(i,j)$ of $\DD_1$ satisfies one of the first two conditions and one of the last two conditions of Lemma \ref{claim}. Hence, all the edges of $A_{i,j}$ and $B_{i,j}$ (for all vertices $(i,j)$ of $\DD_1$), which are all the edges of $\DD_1$, are not in $\CC$, and this is a contradiction, so in this case $\GG(P_m)$ and $\GG(P_n)$ are not commensurable.

Suppose now that both $m$ and $n$ are $0$ modulo $4$. In particular, $n \geq 8$, $m \geq n+4 \geq 12$. By Lemma \ref{claim}, for a vertex $(i,j)$ of $\DD_1$ we have that, if $i-j$ is $2$ modulo $4$, then all the edges of $A_{i,j}$ are not in $\CC$, and if $i+j$ is $2$ modulo $4$, then all the edges of $B_{i,j}$ are not in $\CC$.
In other words, we have 
\begin{equation}\label{edge}
 e_{i,j}^{i+1,j+1} \notin \CC, \: \text{if} \: \:  i-j =2 \: \: \modu  \: 4, \quad e_{i',j'}^{i'+1,j'-1} \notin \CC, \: \text{if} \: \: i'+j' =2 \: \: \modu \: 4.
\end{equation}

 Let $m_0 = m/2, \: n_0=n/2$, so that, when considered on the plane, the vertices of the form $(i,n_0)$ of $\DD_1$ are on the horizontal axis of symmetry of $\DD_1$, the vertices of the form $(m_0,j)$ are on the vertical axis of symmetry of $\DD_1$, and $(m_0,n_0)$ is the ``central'' vertex of $\DD_1$.

Let $V$ be the set of vertices $(i,j)$ of $\DD_1$ such that either $j \leq n_0$ and $j < i < m-j$, or $j \geq n_0$ and $n-j < i < m-n+j$.  Let also $E$ be the set of edges of $\DD_1$ which have at least one of the end vertices in $V$.
 Note that all the vertices of $\DD_1$ of the form $(m_0, j)$, for all $j=1,\ldots,n-1$, are in $V$,  so all the edges adjacent to them are in $E$, and this means that the graph obtained from $\DD_1$ by deleting all the edges from $E$ does not project surjectively to $\ti{\Gamma}_1$. 
We claim that all the edges of $E$ are not in $\CC$. This will immediately imply a contradiction by Lemma \ref{DD}.

Let $V=V_1 \cup V_2$, where $V_1$ are all vertices of $V$ of the form $(i,j)$ with $j \leq n_0$, and $V_2$ are the vertices of $V$ of the form $(i,j)$ with $j \geq n_0$.
Let also $E=E_1 \cup E_2$, where $E_1$ are all edges in $E$ with both end vertices in $V_1$, and $E_2$ are all edges in $E$ with both end vertices in $V_2$. It suffices to prove that all the edges in $E_1$ are not in $\CC$. The proof for $E_2$ is similar. 

Let $V_1=U_1 \cup U_2 \cup \ldots \cup U_{n_0}$, where $U_j$, for $1 \leq j \leq n_0$, are all the vertices of $V_1$ of the form $(i,j)$, for all admissible $i$, namely $j < i < m-j$. Let also $E_1= C_1 \cup C_2 \cup \ldots \cup C_{n_0-1}$, where $C_k$, for $1 \leq k \leq n_0-1$, consists of all edges of $E_1$ connecting vertices with second coordinate equal to $k$ to vertices of $E_1$ with second coordinate equal to $k+1$.
This means that 
$$
\begin{array}{l}
C_k = \{ e_{k+1,k+1}^{k+2,k}, \: e_{k+2,k}^{k+3,k+1}, \: e_{k+3,k+1}^{k+4,k}, \ldots, e_{m-k-3, k+1}^{m-k-2,k}, \: e_{m-k-2,k}^{m-k-1,k+1}    \}, \\
 U_k = \{ (k+2,k), (k+4,k), \ldots, (m-k-2,k)  \}.
\end{array}
$$ 
In particular, $C_k$ contains $m-2k-2$ edges, see Figure \ref{fig:22}.

\begin{figure}[!h]
	\begin{tikzpicture}
	
	\draw[dashed] (0,3) -- (5,3) (0,0) -- (0,3) (0,0) -- (5,0) (5,0)--(5,3);
	\draw (0,0) -- (1.5,1.5)--(0,3) (5,0) -- (3.5,1.5)--(5,3);
	\draw[dashed] (1.5,1.5) -- (3.5,1.5);
	\draw[dotted] (0,1.5)--(1.5,1.5) (3.5,1.5)--(5,1.5);
	\draw (2.5,2)  node {$V_1$};
	\draw (2.5,1)  node {$V_2$};

	\filldraw(0,3) circle (1pt) node[left,above]{$(1,1)$};
	\filldraw(5,3) circle (1pt) node[right,above]{$(m,1)$};
	\filldraw(0,0) circle (1pt) node[left,below]{$(1,n)$};
	\filldraw(5,0) circle (1pt) node[right,below]{$(m,n)$};
	
	\filldraw(2.5,3)  circle (1pt) node[above]{$(m_0,1)$};
	\filldraw(2.5,0)  circle (1pt) node[below]{$(m_0,n)$};
	\filldraw(0,1.5)  circle (1pt) node[left]{$(1,n_0)$};
	\filldraw(5,1.5)  circle (1pt) node[right]{$(m,n_0)$};
	
	\small
	\draw(0.5,0.5) node[right=2pt]{$i=n-j$};
	\draw(0.5,2.5) node[right=2pt]{$i=j$};
	\draw(4.5,2.5) node[left=2pt]{$j=m-i$};
	\draw(4.5,0.5) node[left=2pt]{$m-i=n-j$};
	
	\end{tikzpicture} 
	
	\vspace{0.3 cm}
	\begin{tikzpicture}
	\draw[dashed] (3,0)--(6,0);
	\draw (3,0)--(0,3);
	\draw (6,0)--(9,3);
	\draw[dashed] (0,3)--(9,3);
	
	\draw[dashed] (2.5,0.5)--(6.5,0.5) (2,1)--(7,1) (1.5,1.5)--(7.5,1.5) (1,2)--(8,2) (0.5,2.5)--(8.5,2.5);
	\draw (4,0)--(3,1) (2.5,1.5)--(1,3) (5,0)--(4,1) (3.5,1.5)--(2,3) (6,0)--(5,1) (4.5,1.5)--(3,3) 
	(6.5,0.5)--(6,1) (5.5,1.5)--(4,3) (6.5,1.5)--(5,3) (7.5,1.5)--(6,3) (8,2)--(7,3) (8.5,2.5)--(8,3);
	\draw (3,0)--(4,1) (4.5,1.5)--(6,3) (4,0)--(5,1) (5.5,1.5)--(7,3) (5,0)--(6,1) (6.5,1.5)--(8,3) (2.5,0.5)--(3,1) 
	(3.5,1.5)--(5,3) (2.5,1.5)--(4,3) (1.5,1.5)--(3,3) (1,2)--(2,3) (0.5,2.5)--(1,3);
	\draw[dotted] (4,1.25)--(5,1.25);	
	
	\draw(8.7,2.7) node[right=2pt]{$C_1$};
	\draw(8.2,2.2) node[right=2pt]{$C_2$};
	\draw(7.7,1.7) node[right=2pt]{$C_3$};
	\draw(6.7,0.7) node[right=2pt]{$C_{n_0-2}$};
	\draw(6.2,0.2) node[right=2pt]{$C_{n_0-1}$};
	
	\draw(0,3) node[left=2pt]{$U_1$};		
	\draw(0.5,2.5) node[left=2pt]{$U_2$};
	\draw(1,2) node[left=2pt]{$U_3$};
	\draw(2.5,0.5) node[left=2pt]{$U_{n_0-1}$};
	\draw(3,0) node[left=2pt]{$U_{n_0}$};
	
	\draw[color=blue] (1,3)--(1.5,2.5) -- (2,3) (3,3)--(3.5,2.5)--(4,3) (5,3)--(5.5,2.5)--(6,3) (7,3)--(7.5,2.5)--(8,3);	
	\scriptsize
	\filldraw (0,3) circle (1pt) node[right=2pt,above]{$(1,1)$};
	\filldraw (9,3) circle (1pt) node[right,above]{$(m,1)$};
	\filldraw (3,0) circle (1pt) node[right,below=2pt]{$(n_0,n_0)$};
	\filldraw (6,0) circle (1pt) node[right=2pt,below]{$(m-n_0,n_0)$};
	
	\end{tikzpicture}
	\caption{\small Above: schematic figure depicting the graph $\mathfrak{D}_1$ in the case when both $m$ and $n$ are $0$ modulo $4$, with the regions containing the vertices in $V_1$ and $V_2$ marked. Below: the region with vertices in $V_1$, with blue edges not in $\mathfrak{C}$ by Lemma \ref{claim}.} \label{fig:22}
\end{figure}
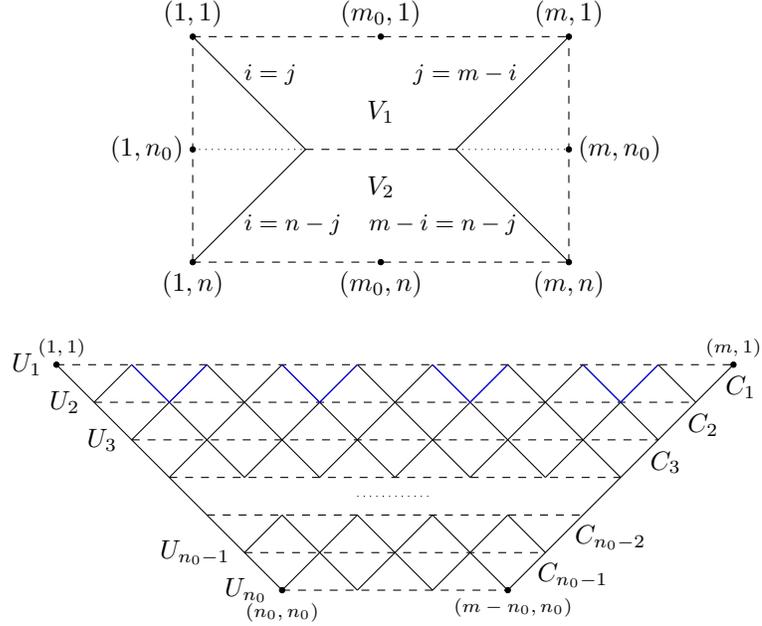 	

We now prove by induction on $k$ that all the edges in $C_k$ are not in $\CC$, for all $1 \leq k \leq n_0-1$. First consider the case $k=1$. We have 
$$ 
\begin{array}{l}
C_1=\{e_{2,2}^{3,1}, \: e_{3,1}^{4,2}, \: e_{4,2}^{5,1}, \ldots, e_{m-5,1}^{m-4,2}, \: e_{m-4,2}^{m-3,1}, \: e_{m-3,1}^{m-2,2} \}, \\
 U_1= \{ (3,1), (5,1), \ldots, (m-3,1) \} .
\end{array} 
$$
By (\ref{edge}), we know that all the edges from $C_1$ of the from $e_{4i-1,1}^{4i,2}$ and $e_{4i,2}^{4i+1,1}$, where $i=1, \ldots, m/4-1$, are not in $\CC$. But, since $e_{4i-1,1}^{4i,2} \notin \CC$, the vertex $(4i-1, 1)$ of $U_1$ is not  in $\CC$  by Lemma \ref{DD} (local surjectivity), so $e_{4i-2,2}^{4i-1,1} \notin \CC$, for all $i=1, 2, \ldots, m/4-1$. In the same way, since $e_{4i,2}^{4i+1,1} \notin \CC$, also $(4i+1,1) \notin \CC$ by Lemma \ref{DD}, and so $e_{4i+1,1}^{4i+2,2} \notin \CC$, for all $i=1, 2, \ldots, m/4-1$. This shows that indeed all the edges of $C_1$ are not in $\CC$. 

Now suppose that for some $1 \leq k \leq n_0-2$ all the edges in $C_k$ are not in $\CC$, and we need to prove that all the edges in $C_{k+1}$ are also not in $\CC$.  
We have $C_k$ and $U_k$ as above, so
$$
\begin{array}{l}
C_{k+1} = \{ e_{k+2,k+2}^{k+3,k+1}, \: e_{k+3,k+1}^{k+4,k+2}, \: e_{k+4,k+2}^{k+5,k+1}, \ldots, e_{m-k-4, k+2}^{m-k-3,k+1}, \: e_{m-k-3,k+1}^{m-k-2,k+2}    \}, \\
 U_{k+1} = \{ (k+3,k+1), (k+5,k+1), \ldots, (m-k-3,k+1)  \}.
\end{array}
$$ 
By the induction hypothesis, all the NE and NW edges at all the vertices of $U_{k+1}$ are not in $\CC$, since they belong to $C_k$.  By Lemma \ref{DD} (local surjectivity), all the vertices of $U_{k+1}$ are also not in $\CC$, and so all the SE and SW edges at all the vertices of $U_{k+1}$ are not in $\CC$, and these are exactly all the edges in $C_{k+1}$.

Thus, all the edges in $C_k$, for $1 \leq k \leq n_0-1$, which are all the edges of $E_1$, are not in $\CC$. This means that $\GG(P_m)$ and $\GG(P_n)$ are not commensurable if both $m$ and $n$ are $0$ modulo $4$.

It remains to consider the case when both $m$ and $n$ are $2$ modulo $4$. According to Theorem \ref{non-com} in this case, if $m=4k+2$ and $n=4l+2$, for some $k \neq l$, $k,l \geq 1$, $\GG(P_m)$ is commensurable to $\GG(T_{k,k+1})$ and $\GG(P_n)$ is commensurable to $\GG(T_{l,l+1})$, but $\GG(T_{k,k+1})$ and $\GG(T_{l,l+1})$ are not commensurable, according to  \cite[Theorem 4.5]{CKZ}, so $\GG(P_m)$ and $\GG(P_n)$ are also not commensurable. Note that this can also be proved directly using the above equations, but it is not completely straightforward, and we omit the argument.

This finishes the proof of Theorem \ref{thm1}.







\begin{thebibliography}{10}
	\bibitem[A13]{Agol} I.~Agol \textit{The virtual Haken Conjecture}. Doc. Math. 18: 1045-1087.
		\bibitem[B72]{3} H.~Bass, \textit{The degree of polynomial growth of finitely generated nilpotent groups}, Proc. London Math. Soc. 3, 4 (1972), 603-614.
		\bibitem[BJN09]{5} J.~Behrstock, T.~Januszkiewicz, W.~Neumann, \textit{Commensurability and QI classification
		of free products of finitely generated abelian groups}. Proc. Am. Math. Soc. 137, 3 (2009), 811-813.
		\bibitem[BN08]{BN} J.~Behrstock,  W.~Neumann, \textit{Quasi-isometric classification of graph manifold
		groups}. Duke Math. J. 141, 2 (2008), 217-240.
		\bibitem[BM00]{17} M.~Burger, S.~Mozes, \textit{Lattices in product of trees}. Inst. Hautes Études Sci. Publ.
		Math., 92 (2000), 151-194.
		\bibitem[CKZ]{CKZ} M.~Casals-Ruiz, I.~Kazachkov, A.~Zakharov, {
		\it On commensurability of some right-angled Artin groups}, preprint: arxiv.org/abs/1611.01741, submitted.
	\bibitem[CH17]{CH} D. Chistikov, C. Haase, \textit{On the complexity of quantified integer programming}, In ICALP'17, vol. 80 of LIPIcs, pp 94:1-94:13, 2017.
		\bibitem[DW93]{13} P. Deligne and G. D. Mostow, \textit{Commensurabilities among lattices in $PU(1,n)$}, Annals of Mathematics Studies, 132, Princeton University Press, 1993.
		 \bibitem[Ga16]{Gar} A.~Garrido, \textit{Abstract commensurability and the Gupta-Sidki group}, Groups Geom. Dyn. 10 (2016), 523-543.
		\bibitem[GrW03]{GrW} R.~Grigorchuk, J.~Wilson, \textit{A structural property concerning abstract commensurability of subgroups}. J. London Math. Soc. (2), 68(3):671-682, 2003.
		\bibitem[Gr81]{30} M.~Gromov, \textit{Groups of polynomial growth and expanding maps}. Inst. Hautes Études Sci. Publ.
		Math., 53 (1981), 53-78.
		\bibitem[Gr93]{31} M.~Gromov, \textit{Geometric group theory, vol. 2: Asymptotic invariants of infinite groups.} Lond. Math. Soc. Lecture Notes 182, (1993), 1-295.
		\bibitem[H16]{Huang} J.~Huang, \textit{Commensurability of groups quasi-isometric to RAAGs}, arXiv:1603.08586v2
		\bibitem[KPS73]{47} A.~Karrass, A. Pietrowski, D.~Solitar, \textit{Finite and infinite cyclic extensions of
			free groups},  J. Australian Math. Soc 16, 04 (1973), 458-466.
	\bibitem[KK13]{KK} S.~Kim,  T.~Koberda, \textit{Embedability between right-angled Artin groups}, Geom. Topol. 17 (2013), 493-530.
	\bibitem[KK14]{KKi} S.~Kim,  T.~Koberda, \textit{The geometry of the curve graph of a right-angled Artin group}, Int. J. Algebra Comput. 24 (2014), 121-169. 
		\bibitem[Mar73]{20} G. A. Margulis, \textit{Arithmeticity of nonuniform lattices}, Funkcional. Anal. i Prilozen., 7(3):88-89, 1973.
		\bibitem[Sch95]{25} R. E. Schwartz, \textit{The quasi-isometry classification of rank one lattices}, Publ. Math. Inst. Hautes Etudes Sci., 82:133-168, 1995. 
		\bibitem[Serre]{Serre} J.-P. Serre, \textit{Trees}, Springer-Verlag, 1980.
		\bibitem[Si43]{26} C. L. Siegel, \textit{Symplectic geometry}, Amer. J. Math., 65:1-86, 1943.
		\bibitem[St68]{56} J.R. Stallings, \textit{On torsion-free groups with infinitely many ends}, Ann. Math. (1968), 312-334.
		\bibitem[We62]{Weichsel} P.~Weichsel, \textit{The Kronecker product of graphs}, Proc. Amer. Math. Soc., 13 (1962): 47-52.
		\bibitem[Wh10]{58} K.~Whyte, \textit{Coarse bundles}, arXiv:1006.3347.
		\bibitem[Wi96]{60} D.~Wise, \textit{Non-positively curved squared complexes aperiodic tilings and non-residually	finite groups}, Princeton University, 1996.
		
	\end{thebibliography}
\end{document}